\documentclass{amsart}
\usepackage{amssymb}
\usepackage{algorithm}
\usepackage{algpseudocode}
\usepackage{enumitem}
\usepackage[colorlinks=true]{hyperref}

\theoremstyle{definition}
\newtheorem{thm}{Theorem}[subsection]
\newtheorem*{teo}{Theorem}
\newtheorem{coro}[thm]{Corollary}
\newtheorem{prop}[thm]{Proposition}

\newtheorem{defin}[thm]{Definition}
\newtheorem{ex}[thm]{Example}
\newtheorem{remark}[thm]{Remark}

\numberwithin{equation}{subsection}

\newenvironment{mydescription}{%
   
   \begin{description}%
}{%
   \end{description}%
}

\newcounter{algo}
\newenvironment{algo}[1][]{\refstepcounter{algo}\par\medskip\noindent%
   \hrule\smallskip
   \noindent\textbf{Algorithm~\thealgo: }#1
   
   \smallskip\hrule\smallskip}{\smallskip\hrule\smallskip}

\DeclareMathOperator{\Supp}{Supp}
\DeclareMathOperator{\LT}{LT}
\DeclareMathOperator{\weight}{weight}
\DeclareMathOperator{\GL}{GL}
\DeclareMathOperator{\Sym}{Sym}
\DeclareMathOperator{\im}{im}
\DeclareMathOperator{\id}{id}
\DeclareMathOperator{\Hom}{Hom}
\DeclareMathOperator{\Ext}{Ext}

\renewcommand{\mod}{\operatorname{mod}}

\usepackage{mathabx,epsfig}

\renewcommand{\geq}{\geqslant}
\renewcommand{\leq}{\leqslant}

\newcommand{\ZZ}{\mathbb{Z}}
\newcommand{\TT}{\mathbb{T}}
\newcommand{\KK}{\mathbb{K}}
\newcommand{\CC}{\mathbb{C}}

\begin{document}
\title{Propagating weights of tori along free resolutions}
\author{Federico Galetto}
\address{Department of Mathematics \& Statistics, Queen's University,
48 University Avenue, Kingston, ON K7L 3N6,
Canada}
\email{\href{mailto:galetto@mast.queensu.ca}{galetto@mast.queensu.ca}}
\urladdr{\href{http://math.galetto.org}{math.galetto.org}}
\date{\today}
%\date{January 11, 2014}

\begin{abstract}
The action of a torus on a graded module over a polynomial ring extends to the entire minimal free resolution of the module. We explain how to determine the action of the torus on the free modules in the resolution, when the resolution can be calculated explicitly. The problem is reduced to analyzing how the weights of a torus propagate along an equivariant map of free modules. The results obtained are used to design algorithms which have been implemented in the software system Macaulay2.
\end{abstract}

\maketitle

\tableofcontents

\section{Introduction}
This paper is structured as follows.
In this section you will find an overview of the questions we
address (\S\ref{sec:motivation}) followed by an example
(\S\ref{sec:an-intr-example}) presented with the minimum amount
of technical background.
Section \ref{basics} introduces some basic concepts of commutative algebra and the representation theory of tori, and proceeds to describe their natural interactions. In section \ref{propagation}, we analyze how weights of tori propagate along equivariant maps of free modules, first in the easier case of bases of weight vectors (\S\ref{bases_of_weight_vectors}) and then in a more general setting (\S\ref{general_case}). Our last section is devoted to the design of various algorithms: to propagate weights along an equivariant map of free modules from codomain to domain (\S\ref{sec_algo_map}), to propagate weights `forward' from domain to codomain (\S\ref{sec_algo_forward}), for resolutions (\S\ref{sec_algo_res}), and, as a bonus, an algorithm to determine the weights of graded components of modules (\S\ref{sec_algo_comps}). Finally, in \S\ref{subfields}, we discuss the possibility of carrying out all such computations over subfields.

An implementation of the algorithms of this paper for semisimple complex algebraic groups is included, under
the package name \texttt{HighestWeights}, with version 1.7 of the software system \texttt{Macaulay2} \cite{Grayson:uq} and is
documented in \cite{Galetto:2013ab}.

The author wishes to thank Jerzy Weyman, for suggesting the project, Claudiu Raicu, for an interesting conversation on the subject, and the entire \texttt{Macaulay2} community.
Additional thanks go the anonymous referees that provided many useful suggestions for improving this work.
The author was partially supported through an NSERC grant.

\subsection{Motivation}
\label{sec:motivation}

Every finitely generated module over a polynomial ring with coefficients in a field has a finite minimal free resolution which is unique up to isomorphism. It is typically used to produce numerical invariants such as projective dimension, regularity, (graded) Betti numbers and the Hilbert series of a module. While there are descriptions for certain classes of modules, finding a minimal free resolution of a module is, in general, a very difficult problem. Computational methods offer a solution to this problem in many cases, although they are limited in scope by time and memory constraints. As the matrices of the differentials grow in size, their description is often omitted.

Consider the case of a polynomial ring $A$ endowed with an action of a group $G$ which is compatible with grading and multiplication (see \S\ref{interactions} for the precise definitions). Let us denote $\mod_{\righttoleftarrow G} A$ the category of finitely generated graded $A$-modules with a compatible action of $G$ and homogeneous $G$-equivariant maps. If $M$ is an object in $\mod_{\righttoleftarrow G} A$, then the action of $G$ extends to the entire minimal free resolution of $M$. A free $A$-module $F$ is isomorphic to $(F/\mathfrak{m}F) \otimes A$, where $\mathfrak{m}$ denotes the maximal ideal generated by the variables of $A$. The representation theoretic structure of $F$, i.e. the action of $G$ on $F$, is then controlled by the representation $F/\mathfrak{m}F$. Therefore, if the complex $F_\bullet$:
\[0\to F_n \xrightarrow{d_n} F_{n-1} \to \ldots \to F_i \xrightarrow{d_i} F_{i-1} \to \ldots \to F_1 \xrightarrow{d_1} F_0\]
denotes a minimal free resolution of $M$, we could try to determine the action of $G$ on each representation $F_i/\mathfrak{m}F_i$.

The representation theoretic structure of $F_\bullet$ may offer some insight into the maps of the complex. Consider the situation of a differential $d_i\colon F_i\to F_{i-1}$, with $F_i/\mathfrak{m}F_i$ an irreducible representation of $G$. The map $d_i$ is completely determined by its image on a basis of $F_i$; hence we can reduce to a map of representations $F_i/\mathfrak{m}F_i \to F_{i-1}/\mathfrak{m}F_{i-1} \otimes A$. If the decomposition of the tensor product in the codomain contains only one copy of the irreducible $F_i/\mathfrak{m}F_i$ in the right degree, then Schur's lemma \cite[Ch. XVII, Prop. 1.1]{MR1878556} implies that the map is uniquely determined up to multiplication by a constant. In some cases, this information is enough to reconstruct the map completely (see \cite{Galetto:2013aa} for a few examples).

The representation theoretic structure of $F_\bullet$ may also be used to determine the class $[M]$, of a module $M$, in $K_0^G (\mod_{\righttoleftarrow G} A)$, the equivariant Grothendieck group of the category $\mod_{\righttoleftarrow G} A$. By construction of the Grothendieck group,
\[[M] = \sum_{i=0}^n (-1)^i [F_i/\mathfrak{m}F_i],\]
where the right hand side is the equivariant Euler characteristic of the complex $F_\bullet$.

Motivated by the discussion above, we pose the following question: is it possible to determine the action of a group on a minimal free resolution of a module computationally? The first assumption is that the resolution itself can be computed explicitly in a reasonable amount of time. Secondly, we restrict to a class of groups whose representation theory is well understood and manageable: tori. Every representation of a torus is semisimple, with irreducible representations being one dimensional and parametrized by weights. Moreover, weights can be conveniently represented by integer vectors. More importantly, finite dimensional representations of connected reductive algebraic groups over algebraically closed fields of characteristic zero are uniquely determined by the weights (with multiplicity) of a maximal torus. Therefore successfully developing the case of tori will provide a positive answer to our question for a larger class of groups.

Let $\varphi\colon E\to F$ denote a minimal presentation of a module $M$. In our experience, the presentation of a module with an action of a reductive group can often be written with respect to bases of weight vectors. Suppose $\{\tilde{e}_1,\ldots,\tilde{e}_r\}$ and $\{\tilde{f}_1,\ldots,\tilde{f}_s\}$ are bases of weight vectors of $E$ and $F$ respectively.
For every $\tilde{e}_j$, there exist polynomials $a_{i,j}\in A$ such that $\varphi (\tilde{e}_j) = \sum_{i=1}^s a_{i,j} \tilde{f}_i$; moreover, each $a_{i,j}$ is a weight vector in $A$ and 
\[\weight (\tilde{e}_j) = \weight(a_{i,j}) + \weight (\tilde{f}_i),\]
whenever $a_{i,j}$ is non zero (prop. \ref{weight_in_homogeneous_basis}). If the variables in $A$ are all weight vectors, then the weights of the $a_{i,j}$ can be easily obtained (prop. \ref{weight_of_poly}). Then, knowing the weights of the $\tilde{f}_i$, it is possible to recover the weights of the $\tilde{e}_j$ thus describing $E$ as a representation. A minimal presentation is the first differential in a minimal resolution. The weights of the other free modules in the resolution can be found by iterating the process just described, with the only catch that the remaining differentials, as obtained computationally, may not be expressed using bases of weight vectors. This situation is handled by changing to a suitable basis as we indicate in our main result (thm. \ref{main_thm}).

\subsection{A preliminary example}
\label{sec:an-intr-example}

Consider the polynomial ring $A = \CC [x_1,x_2,x_3]$ with the
standard grading.
The ideal $I=(x_1,x_1+x_2,x_1+x_3)$ defines the origin in the
affine space $\mathbb{A}^3$, which has codimension 3.
Moreover, $I$ is generated by 3 elements,
so the generators form a regular sequence.
It is well known that $I$ is resolved by a Koszul complex,
which can be written as:
\[0\to F_3 \xrightarrow{ \left(
    \begin{smallmatrix}
      x_1+x_3\\
      -x_1-x_2\\
      x_1
    \end{smallmatrix}
  \right) } F_2 \xrightarrow{ \left(
    \begin{smallmatrix}
      -x_1 -x_2 & -x_1 -x_3 & 0 \\
      x_1 & 0 & -x_1 -x_3\\
      0 & x_1 & x_1 +x_2
    \end{smallmatrix}
  \right) } F_1 \xrightarrow{ \left(
    \begin{smallmatrix}
      x_1 & x_1 + x_2 & x_1 + x_3
    \end{smallmatrix}
  \right) } F_0
\]
The free module $F_i$ is isomorphic to $\bigwedge^i A(-1)^3$.

The vector space spanned by the variables $x_1,x_2,x_3$ is
isomorphic to $\CC^3$ and therefore we can identify $A$ with the
symmetric algebra $\Sym (\CC^3)$.
Note that the group $\GL_3 (\CC)$ acts naturally on $\CC^3$ and this
action extends to $A$ in a manner that preserves degrees and
multiplication.
Moreover, the ideal $I$ is stable under the action of $\GL_3 (\CC)$.
This might be more evident from a geometric perspective as the
origin of $\mathbb{A}^3$ is fixed by the dual action of $\GL_3 (\CC)$
on $\mathbb{A}^3$.
Because $I$ is stable for the action of $\GL_3 (\CC)$, the group
action extends to a minimal free resolution of $A/I$;
in particular, the action extends to the Koszul complex above.
The module $F_i$ is also isomorphic to $\bigwedge^i \CC^3 \otimes_\CC
A(-i)$.
This form emphasizes the presence of the group action by expressing
the free module as the tensor product of a finite dimensional
representation of $\GL_3 (\CC)$ and a free $A$-module of rank one.

A computer can be taught to calculate a minimal free resolution
which provides such information as the ranks of the free modules
in the complex, as well as the degrees they are generated in.
For example, computing a minimal free resolution of $A/I$ would
give the free modules $F_0, F_1, F_2, F_3$ with ranks $1, 3, 3, 1$
and generated in degrees $0, 1, 2, 3$.
But how can a computer algebra system be taught to recognize the
action of $\GL_3 (\CC)$ on this complex?
In other words, how can a piece of software identify
the representations appearing in our complex as the exterior
powers of $\CC^3$?

To illustrate our method, we recall the concept of weight in a 
representation.
Consider the subgroup $T$ of diagonal matrices in $\GL_3 (\CC)$.
If $e_1,e_2,e_3$ are the vectors in the coordinate basis of $\CC^3$,
then, for $i\in\{1,2,3\}$,
\[
\begin{pmatrix}
  t_1 & 0 & 0\\
  0 & t_2 & 0\\
  0 & 0 & t_3
\end{pmatrix}
e_i = t_i e_i.
\]
Notice how $e_i$ is an eigenvector for any diagonal matrix,
the eigenvalue being a function of $t_1,t_2,t_3$.
To give another example, consider the vector space
$\bigwedge^2 \CC^3$ with the $\GL_3 (\CC)$ action.
For $1\leq i<j\leq 3$,
\[
\begin{pmatrix}
  t_1 & 0 & 0\\
  0 & t_2 & 0\\
  0 & 0 & t_3
\end{pmatrix}
e_i \wedge e_j = t_i t_j e_i\wedge e_j.
\]
Again we see that $e_i\wedge e_j$ is an eigenvector for any diagonal
matrix, and the eigenvalue depends on $t_1,t_2,t_3$.
In general, we refer to these eigenvectors as weight vectors.
Every finite dimensional representation of $\GL_3 (\CC)$ has a 
basis of weight vectors. Moreover, the eigenvalue of a weight
vector is always a monomial $t_1^{\alpha_1} t_2^{\alpha_2} t_3^{\alpha_3}$,
and the tuple $(\alpha_1,\alpha_2,\alpha_3)$ is called the weight of
the weight vector.
Weights identify a representation up to isomorphism.
For example, our computation above shows that $e_1,e_2,e_3$ are a
basis of weight vectors of $\CC^3$ with weights $(1,0,0)$, $(0,1,0)$, $(0,0,1)$. 
The theory says any other representation with these three weights is
isomorphic to $\CC^3$.
Similarly, $e_1 \wedge e_2$, $e_1\wedge e_3$, $e_2\wedge e_3$ are a basis
of weight vectors of $\bigwedge^2 \CC^3$ with weights
$(1,1,0)$, $(1,0,1)$, $(0,1,1)$, and any other representation with these
weights is isomorphic to $\bigwedge^2 \CC^3$.
Since weights are represented by tuples of integers, they can
easily be handled by software.
But how do weights figure in our free resolution?

Since we are resolving $A/I$, we know the module $F_0$ is simply $A$.
The free module $A$ has a single basis vector, namely $1_A$, on which
$\GL_3 (\CC)$ acts trivially.
The group $T$ also acts trivially on $1_A$, hence $1_A$ is a weight
vector with weight $(0,0,0)$.
Next we can attach weights to the variables in $A$. In fact, if we
identify $x_1,x_2,x_3$ with the basis $e_1,e_2,e_3$ of $\CC^3$, we
see that the variables in $A$ are weight vectors with weights $(1,0,0)$, $(0,1,0)$ and $(0,0,1)$.
We will use this fact to recover the weights for the module
$F_1$ by propagating the weight $(0,0,0)$ in $F_0$ along the first
map of the complex.

In order to do that, we need some tools from computational
commutative algebra.
We equip $A$ with the degree reverse
lexicographic term ordering such that $x_1 > x_2 > x_3$.
This allows us to calculate leading terms of polynomials and
Gr\"obner bases.
Now the first map in the complex is represented by the matrix
\[M_1 =
\begin{pmatrix}
  x_1 & x_1 + x_2 & x_1 + x_3
\end{pmatrix}.
\]
Each column has one entry corresponding to a generator of $I$.
If we compute a Gr\"obner basis of $I$ and replace each column in
the matrix above with an element of the Gr\"obner basis we obtain
\[G_1 = 
\begin{pmatrix}
  x_3 & x_2 & x_1
\end{pmatrix}.
\]
Alternatively, we can think $G_1$ is obtained from $M_1$ after
operating the change of basis given by the matrix
\[C_1 =
\begin{pmatrix}
  -1 & -1 & 1\\
  0 & 1 & 0\\
  1 & 0 & 0
\end{pmatrix}
\]
in the domain $F_1$.
As we mentioned earlier, $F_0$, the codomain of $G_1$, has a basis
consisting of a single weight vector with weight $(0,0,0)$.
We think of the weight $(0,0,0)$ as attached to the single row 
of $G_1$.
Observe how each column of $G_1$ contains a single variable,
and each variable is a weight vector.
To propagate the weights through this map we proceed as follows.
For each column, add the weight of the row $(0,0,0)$ to the weight
of the variable in the column.
The process gives the following weights:
\begin{align*}
  &(0,0,0) + (0,0,1) = (0,0,1),\\
  &(0,0,0) + (0,1,0) = (0,1,0),\\
  &(0,0,0) + (1,0,0) = (1,0,0).
\end{align*}
For example, the first line comes from adding the weight $(0,0,0)$
of the row to the weight $(0,0,1)$ of the variable $x_3$ found in
the first column.
We think of these weights as attached to the columns of $G_1$.
As such, these weights correspond to some weight vectors in $F_1$,
the domain of $G_1$.
As a free module, $F_1$ is isomorphic to $V \otimes_\CC A(-1)$,
where $V$ is a representation of $\GL_3 (\CC)$ of dimension 3.
At the same time, we previously mentioned that there is only one 
representation $V$ with the weights $(1,0,0)$, $(0,1,0)$ and $(0,0,1)$,
namely $\CC^3$.
Therefore $F_1 \cong \CC^3 \otimes_\CC A(-1)$.

Now we would like to apply the same process to the next map in the
complex.
First we notice that, since we changed basis in $F_1$, the same
change of basis should be applied to the second map in the complex.
To achieve this, we multiply the second map on the left by the inverse of
$C_1$, which gives
\begin{align*}
  M_2 &= C_1^{-1}
  \begin{pmatrix}
    -x_1 -x_2 & -x_1 -x_3 & 0 \\
    x_1 & 0 & -x_1 -x_3\\
    0 & x_1 & x_1 +x_2
  \end{pmatrix}
  =\\
  &=
  \begin{pmatrix}
    0 & 0 & 1\\
    0 & 1 & 0\\
    1 & 1 & 1
  \end{pmatrix}
  \begin{pmatrix}
    -x_1 -x_2 & -x_1 -x_3 & 0 \\
    x_1 & 0 & -x_1 -x_3\\
    0 & x_1 & x_1 +x_2
  \end{pmatrix}
  =\\
  &=
  \begin{pmatrix}
    0 & x_1 & x_1+x_2\\
    x_1 & 0 & -x_1-x_3\\
    -x_2 & -x_3 & x_2-x_3
  \end{pmatrix}.
\end{align*}
The weights $(0,0,1),(0,1,0),(1,0,0)$ that were attached to the
columns of $G_1$ are now also attached to the rows of $M_2$,
since the basis used is the same.
Next we compute a Gr\"obner basis of the image of $M_2$ and arrange
its elements in the matrix
\[G_2=
\begin{pmatrix}
  x_2 & x_1 & 0\\
  -x_3 & 0 & x_1\\
  0 & -x_3 & -x_2
\end{pmatrix}.
\]
Implicit in this process is the choice of a module term ordering
on $F_1$, i.e. a way to sort terms in the free module $F_1$.
Our choice is the \emph{term over position up} ordering.
We illustrate with an example.
The last column of $M_2$ can be written as a linear combination of
terms as
\[
\begin{pmatrix}
  x_1+x_2\\
  -x_1-x_3\\
  x_2-x_3
\end{pmatrix}
=
-
\begin{pmatrix}
  0\\
  x_1\\
  0
\end{pmatrix}
+
\begin{pmatrix}
  x_1\\
  0\\
  0
\end{pmatrix}
+
\begin{pmatrix}
  0\\
  0\\
  x_2
\end{pmatrix}
+
\begin{pmatrix}
  x_2\\
  0\\
  0
\end{pmatrix}
-
\begin{pmatrix}
  0\\
  0\\
  x_3
\end{pmatrix}
-
\begin{pmatrix}
  0\\
  x_3\\
  0
\end{pmatrix}.
\]
The terms on the right hand side of the equality are in decreasing
order from left to right, the first one being the leading term.
To compare two terms, we compare their non zero entries using the
degree reverse lexicographic order.
If the non zero entry is the same, then we compare their position
within the column; the one closer to the bottom wins.
Equivalently, we could say $G_2$ is obtained from $M_2$ by
applying the change of basis given by the matrix
\[C_2=
\begin{pmatrix}
  1 & 0 & 1\\
  -1 & 1 & 0\\
  1 & 0 & 0
\end{pmatrix}
\]
in the domain $F_2$.
Notice that the basis for the codomain $F_1$ has not changed,
thus we can still think of the weights $(0,0,1),(0,1,0),(1,0,0)$ being attached to the rows of $G_2$.
To propagate weights along the map $G_2$ we proceed as follows.
First we identify the leading term of each column of $G_2$:
\[
\begin{pmatrix}
  x_2 & x_1 & 0\\
  0 & 0 & x_1\\
  0 & 0 & 0
\end{pmatrix}.
\]
Then we locate the non zero term in each column and add the weight attached to its row to the weight of the term.
The process gives the following weights:
\begin{align*}
  &(0,0,1) + (0,1,0) = (0,1,1),\\
  &(0,0,1) + (1,0,0) = (1,0,1),\\
  &(0,1,0) + (1,0,0) = (1,1,0);
\end{align*}
for example, the first line comes from the weight attached to the
first row, namely $(0,0,1)$, plus the weight $(0,1,0)$ of $x_2$.
Now we can think of the weights $(0,1,1)$, $(1,0,1)$, $(1,1,0)$ as
being attached to the columns of $G_2$, so they must correspond
to some weight vectors in $F_2$.
As observed earlier, the only representation of $\GL_3 (\CC)$
with these weights is $\bigwedge^2 \CC^3$.
Therefore $F_2 \cong \bigwedge^2 \CC^3 \otimes_\CC A(-2)$.

Finally the process extends to the last map of the complex.
The matrix of this map, as we described it at the beginning of this 
section, must be expressed in the new basis we picked for $F_2$;
this is achieved by multiplying on the left by the inverse of $C_2$:
\[M_3 = C_2^{-1}
\begin{pmatrix}
  x_1+x_3\\
  -x_1-x_2\\
  x_1
\end{pmatrix}
=
\begin{pmatrix}
  0 & 0 & 1\\
  0 & 1 & 1\\
  1 & 0 & -1
\end{pmatrix}
\begin{pmatrix}
  x_1+x_3\\
  -x_1-x_2\\
  x_1
\end{pmatrix}
=
\begin{pmatrix}
  x_1\\
  -x_2\\
  x_3
\end{pmatrix}.
\]
Now the image of $M_3$ is generated by a single column, so a
Gr\"obner basis is given by $G_3 = M_3$ itself.
Its leading term is the column:
\[
\begin{pmatrix}
  x_1\\
  0\\
  0
\end{pmatrix}.
\]
From our previous step, we know the weight attached to the first row
is the same as the weight that was attached to the first column of
$G_2$, namely $(0,1,1)$.
By adding this weight to the weight of $x_1$, we obtain:
\[(0,1,1) + (1,0,0) = (1,1,1).\]
Thus the weight $(1,1,1)$ is attached to the column of $G_3$ and
corresponds to a weight vector of $F_3$.
The unique representation of $\GL_3 (\CC)$ with weight $(1,1,1)$ is
$\bigwedge^3 \CC^3$.
Therefore $F_3 \cong \bigwedge^3 \CC^3 \otimes_\CC A(-3)$.

As you see, this process allowed us to propagate weights along
the maps in a minimal free resolution which enabled us to
identify each free module as a group representation.
What makes these kind of weight computations feasible?
We develop the theory of weight propagation in section
\ref{propagation}.
Here we state our main result in loose terms.
\begin{teo}[Informal statement of theorem \ref{main_thm}]
Consider a matrix $M$ such as those appearing in a minimal free resolution of an $A$-module and assume this matrix is compatible
with the action of the group $T$.
If each row of $M$ comes with a weight attached to it,
then, after a suitable change of basis in the domain of $M$,
we can attach a weight to each column of the matrix using the
following recipe:
\begin{align*}
  &{}\text{weight of column } c =\\
  ={}&{}\text{weight of row containing the leading term of } c \:+\\
  +{\:}&{}\text{weight of the leading term of }c.
\end{align*}
Moreover, the weights found this way completely describe
the domain of $M$ as a representation of $T$.
% Let $E$, $F$ be free $A$-modules and let $\varphi \colon E\to F$
% be a map such as those appearing in a minimal free resolution of
% an $A$-module.
% Assume the group $T$ acts on $E$ and $F$, and the map
% $\varphi$ commutes with the action of $T$.
% Moreover, suppose $\varphi$ is given as a matrix $M$ and that we have
% a way to attach the weights of $F$ to the rows of $M$
\end{teo}

The procedure can be carried out by a computer and will be codified
into an algorithm in section \ref{sec_algo_res}.

\section{Basic notions and notations}\label{basics}
We begin by recalling some concepts from commutative algebra (\S\ref{comm_alg}) and representation theory (\S\ref{rep_theory}). In sections \ref{interactions} and \ref{free_mod}, we illustrate some interactions and establish a few basic facts that will be used throughout this work.

\subsection{Commutative algebra}\label{comm_alg}
Our main source for (computational) commutative algebra are \cite{MR1790326,MR2159476}. Most of our notations are lifted from those volumes, with some minor additions that we introduce below, and many facts we cite can be found there.
Additional facts we may refer to can be found in \cite{MR1251956}.

Let $\KK$ be a field. Fix a polynomial ring $A=\KK [x_1,\ldots,x_n]$ with a positive $\ZZ^m$-grading in the sense of \cite[Defin. 4.2.4]{MR2159476}.
This means that:
\begin{itemize}[wide]
\item the degree of each variable $x_i$ is a non zero vector in
  $\ZZ^m$ and its first non zero component is positive;
\item when the degrees of the variables are written as column
  vectors and arranged in a matrix, the rows of such a matrix are
  linearly independent.
\end{itemize}
\begin{ex}
  Consider the polynomial ring $A=\CC [x,y]$ with the standard
  grading. Each variable has degree $1$ and the degrees
  can be arranged in a matrix
  \[
  \begin{pmatrix}
    1 & 1
  \end{pmatrix}
  \]
  with a single row. Hence the standard grading is a positive
  grading.

  Now suppose $x$ has degree $(1,0)$ and $y$ has degree $(0,1)$.
  Both degrees satisfy the first requirement and they give rise
  to the matrix
  \[
  \begin{pmatrix}
    1 & 0\\
    0 & 1
  \end{pmatrix}
  \]
  whose rows are linearly independent. Hence this is also a
  positive grading.
\end{ex}

One important consequence of positive gradings is that for every finitely generated graded $A$-module $M$ and $\forall d\in \ZZ^m$, the graded component $M_d$ is a finite dimensional $\KK$-vector space \cite[Prop. 4.1.19]{MR2159476}. All $A$-modules we consider are finitely generated and graded.

Every free $A$-module $F$ has a basis $\mathcal{F} = \{f_1,\ldots,f_s\}$ consisting (necessarily) of homogeneous elements. In order to distinguish the set $\mathcal{F}$ from the basis of a $\KK$-vector subspace of $F$, we refer to it as a `homogeneous basis'. Kreuzer and Robbiano denote $\TT^n \langle f_1,\ldots,f_s\rangle$ the set of all terms $tf_i$ in $F$, where $t\in \TT^n$ is a term in $A$; we use $\TT^n \langle \mathcal{F} \rangle$ to denote the same set. At times, we may need to use more than one homogeneous basis of the same free module $F$; because of this, we denote the support of an element $f\in F$ by $\Supp_{\mathcal{F}} (f)$, with an explicit dependence on the homogeneous basis used.

\begin{ex}\label{ex_terms1}
Let $A=\CC [x,y]$ and consider the free $A$-module $F=A^4$,
with homogeneous basis $f_1 = (1,0,0,0)$, $f_2 = (0,1,0,0)$,
$f_3 = (0,0,1,0)$, $f_4 = (0,0,0,1)$ (the coordinate basis).
We set $\mathcal{F}= \{f_1,f_2,f_3,f_4\}$.
Consider the element of $F$ defined by
\[f = y f_1 + x f_2 + x f_3 + y f_4.\]
The right hand side shows how $f$ is written as a linear
combination of terms in $\mathbb{T}^4 \langle \mathcal{F} \rangle$.
It follows that $\Supp_{\mathcal{F}} (f) = \{y f_1,x f_2,x f_3,y f_4\}$.
Consider now the homogeneous basis $\mathcal{F}' = \{f'_1,f'_2,f'_3,f'_4\}$ of
$F$ defined by $f'_1 = f_1 + f_4$, $f'_2 = f_2 + f_3$, $f'_3 = f_3$,
$f'_4 = f_4$.
Then
\[f = y f'_1 + x f'_2\]
and $\Supp_{\mathcal{F}'} (f) = \{y f'_1,x f'_2\}$.
\end{ex}

\begin{defin}\label{module_term_orderings}
Let $\sigma$ denote a term ordering on $\TT^n$.
Let $F$ be a free $A$-module and $\mathcal{F} = \{f_1,\ldots,f_s\}$ a homogeneous basis of $F$.
We define some module term orderings on $\TT^n \langle \mathcal{F}\rangle$ as follows.
For $t_1,t_2\in \TT^n$ and $i,j \in \{1,\ldots,s\}$,
\begin{itemize}[wide]
\item \emph{term over position up}: $t_1 f_i \geq t_2 f_j \iff t_1 >_\sigma t_2$ or $t_1=t_2$ and $i>j$;
\item \emph{position up over term}: $t_1 f_i \geq t_2 f_j \iff i>j$ or $i=j$ and $t_1 >_\sigma t_2$;
\item \emph{term over position down}: $t_1 f_i \geq t_2 f_j \iff t_1 >_\sigma t_2$ or $t_1=t_2$ and $i<j$;
\item \emph{position down over term}: $t_1 f_i \geq t_2 f_j \iff i<j$ or $i=j$ and $t_1 >_\sigma t_2$.
\end{itemize}
We refer to the first two orderings together as \emph{position up} module term orderings, and to the last two as \emph{position down} module term orderings.
\end{defin}

Let $F$ be a free $A$-module and $\mathcal{F}$ a homogeneous
basis of $F$.
If $\mathbb{T}^n \langle \mathcal{F}\rangle$ is equipped with one
of the module term orderings above, then the leading term of an
element $f\in F$ is the maximum of $\Supp_{\mathcal{F}} (f)$ and will
be denoted $\LT (f)$.
\begin{ex}
  Consider the element $f$ from example \ref{ex_terms1}.
  Suppose the set of terms of $A$, i.e. $\mathbb{T}^4$, has
  the degree reverse lexicographic term ordering and $x>y$.
  Depending on the module term ordering on $\mathbb{T}^4 \langle
  \mathcal{F} \rangle$, $f$ may have a different leading term.
  In fact:
  \begin{itemize}[wide]
  \item if $\mathbb{T}^4 \langle \mathcal{F} \rangle$ has the
    term over position up ordering, then $\LT (f) = x f_3$;
  \item if $\mathbb{T}^4 \langle \mathcal{F} \rangle$ has the
    position up over term ordering, then $\LT (f) = y f_4$;
  \item if $\mathbb{T}^4 \langle \mathcal{F} \rangle$ has the
    term over position down ordering, then $\LT (f) = x f_2$;
  \item if $\mathbb{T}^4 \langle \mathcal{F} \rangle$ has the
    position down over term ordering, then $\LT (f) = y f_1$.
  \end{itemize}
\end{ex}

Any term ordering on $\TT^n$ is allowed in what follows. The term ordering will be fixed and so we will drop all references to it. As for module term orderings, we only allow position up/down orderings.

In some parts of this work, we will be using the graded hom functor on the category of graded $A$-modules. The definition can be found in \cite[p. 33]{MR1251956} and the notation is $\sideset{^*}{} \Hom (-,-)$.

\subsection{Representation theory of tori}\label{rep_theory}
This section contains a brief summary of the representation theory of tori. Our main reference on the subject is \cite[\S11.4, Ch. 16]{MR0396773}. A few facts are presented in the form of propositions so that they can be referenced later. The proofs are standard and typically present in textbooks on the subject,  therefore they will be omitted.

An algebraic torus over $\KK$ is an algebraic group $T$ which is isomorphic to $\KK^\times \times \ldots \times \KK^\times$, a finite direct product of copies of the multiplicative group of the field $\KK$. The character group of $T$, denoted $X(T)$, is the set of all algebraic group homomorphisms $\chi \colon T\rightarrow \KK^\times$. The set $X(T)$ is an abelian group under pointwise multiplication.

A representation of $T$ over $\KK$ is a vector space with a $\KK$-linear action of $T$. For every finite dimensional representation $V$ of $T$ over $\KK$, there exists a unique decomposition $V = \bigoplus_{\chi \in X(T)} V_\chi$, where $V_\chi = \{v\in V \mid \forall \tau\in T, \tau\cdot v = \chi (\tau) v\}$. A character $\chi\in X(T)$ such that $V_\chi \neq 0$ is called a weight of $V$ and $\dim (V_\chi)$ is called the multiplicity of $\chi$ in $V$. The list of weights of $V$, considered with their multiplicity, uniquely determines $V$ as a representation of $T$. Each $V_\chi$ is called a weight space of $V$ and its non zero elements are called weight vectors with weight $\chi$.

\begin{ex}\label{ex_weight1}
  Let
  \[T=\left\{
    \begin{pmatrix}
      t_1 & 0\\
      0 & t_2
    \end{pmatrix}
    \middle\vert
    t_1,t_2\in\CC^\times\right\}
  \]
  with matrix multiplication as group operation.
  Then $T \cong (\CC^\times)^2$ so it is a torus.

  Let $V=\CC^2$ with coordinate basis $\{v_1,v_2\}$. We have:
  \begin{align*}
    &
    \begin{pmatrix}
      t_1 & 0\\
      0 & t_2
    \end{pmatrix}
    v_1 =
    \begin{pmatrix}
      t_1 & 0\\
      0 & t_2
    \end{pmatrix}
    \begin{pmatrix}
      1\\ 0
    \end{pmatrix}
    =
    \begin{pmatrix}
      t_1\\ 0
    \end{pmatrix}
    = t_1 v_1,\\
    &
    \begin{pmatrix}
      t_1 & 0\\
      0 & t_2
    \end{pmatrix}
    v_2 =
    \begin{pmatrix}
      t_1 & 0\\
      0 & t_2
    \end{pmatrix}
    \begin{pmatrix}
      0\\ 1
    \end{pmatrix}
    =
    \begin{pmatrix}
      0\\ t_2
    \end{pmatrix}
    = t_2 v_2.
  \end{align*}
  This shows $v_1$ and $v_2$ are both weight vectors.

  Define the functions
  \begin{align*}
    &\chi_{(1,0)} \colon T \to \CC^\times,
    \quad
    \begin{pmatrix}
      t_1 & 0\\
      0 & t_2
    \end{pmatrix}
    \mapsto t_1,\\
    &\chi_{(0,1)} \colon T \to \CC^\times,
    \quad
    \begin{pmatrix}
      t_1 & 0\\
      0 & t_2
    \end{pmatrix}
    \mapsto t_2.
  \end{align*}
  Both functions are characters of $T$.
  Moreover, $v_1$ is a weight vector of $V$ with weight $\chi_{(1,0)}$
  and $v_2$ is a weight vector of $V$ with weight $\chi_{(0,1)}$.
\end{ex}

\begin{ex}\label{ex_weight2}
  Consider the setup of example \ref{ex_weight1} and
  let $W = \Sym^2 (V)$, the second symmetric power of $V$.
  The set $\{v_1^2,v_1 v_2,v_2^2\}$ is a basis of $W$.
  If we take the element
  \[\tau = 
    \begin{pmatrix}
      t_1 & 0\\
      0 & t_2
    \end{pmatrix}
    \in T,
  \]
  then we have
  \begin{align*}
    &\tau \cdot v_1^2 = (\tau \cdot v_1)^2 = (t_1 v_1)^2 = t_1^2 v_1^2,\\
    &\tau \cdot v_1 v_2 = (\tau \cdot v_1) (\tau \cdot v_2) = (t_1 v_1) (t_2 v_2) = t_1 t_2 v_1 v_2,\\
    &\tau \cdot v_2^2 = (\tau \cdot v_2)^2 = (t_2 v_2)^2 = t_2^2 v_2^2.
  \end{align*}
  The functions
  \begin{align*}
    &\chi_{(2,0)} \colon T \to \CC^\times,
    \quad
    \begin{pmatrix}
      t_1 & 0\\
      0 & t_2
    \end{pmatrix}
    \mapsto t_1^2,\\
    &\chi_{(1,1)} \colon T \to \CC^\times,
    \quad
    \begin{pmatrix}
      t_1 & 0\\
      0 & t_2
    \end{pmatrix}
    \mapsto t_1 t_2,\\
    &\chi_{(0,2)} \colon T \to \CC^\times,
    \quad
    \begin{pmatrix}
      t_1 & 0\\
      0 & t_2
    \end{pmatrix}
    \mapsto t_2^2
  \end{align*}
  are characters of $T$, and
  \begin{itemize}[wide]
  \item $v_1^2$ is a weight vector of $W$ with weight $\chi_{(2,0)}$,
  \item $v_1 v_2$ is a weight vector of $W$ with weight $\chi_{(1,1)}$,
  \item $v_2^2$ is a weight vector of $W$ with weight $\chi_{(0,2)}$.
  \end{itemize}
\end{ex}

If $T \cong (\KK^\times)^n$, then $X(T)\cong \ZZ^n$ as abelian groups. This is important from a computational standpoint because characters can be represented by elements of $\ZZ^n$, i.e. by lists of integers. If $v\in V$ is a weight vector, we write $\weight (v)$ to identify the weight of $v$ as an element of $\ZZ^n$. When operating with weights as elements of $\ZZ^n$, we use an additive notation for the group operation.

\begin{ex}\label{ex_weight3}
  In example \ref{ex_weight1}, $\weight (v_1) = (1,0)$ and 
  $\weight (v_2) = (0,1)$.
  In example \ref{ex_weight2}, $\weight (v_1^2) = (2,0)$,
  $\weight (v_1 v_2) = (1,1)$ and $\weight (v_2) = (0,2)$.
\end{ex}

If $V,W$ are two finite dimensional representations of $T$ over $\KK$, then so is the tensor product $V\otimes_\KK W$ with the action $\tau \cdot (v\otimes w) = (\tau \cdot v) \otimes (\tau \cdot w)$, $\forall v \in V$, $\forall w \in W$, $\forall \tau \in T$.% If $v\in V$ is a weight vector with weight $\chi_1$ and $w\in W$ is a weight vector with weight $\chi_2$, then $v\otimes w\in V\otimes_\KK W$ is also a weight vector with weight $\chi_1 \chi_2$. Additively we shall write $\weight (v\otimes w) = \weight (v) + \weight (w)$.

\begin{prop}\label{weight_in_tensor_product}
If $v\in V$ and $w\in W$ are weight vectors, then so is $v\otimes w\in V\otimes_\KK W$ and $\weight (v\otimes w) = \weight (v) + \weight (w)$.
\end{prop}
%\begin{proof}
%Suppose $v$ has weight $\chi_1$ and $w$ has weight $\chi_2$. Then $\forall \tau \in T$
%\begin{align*}
%\tau \cdot (v\otimes w) &= (\tau \cdot v) \otimes (\tau \cdot w) = (\chi_1(\tau) v) \otimes (\chi_2(\tau) w) =\\
%&= \chi_1 (\tau) \chi_2 (\tau) (v\otimes w) = \chi_1 \chi_2 (\tau) (v\otimes w).
%\end{align*}
%This shows the weight of $v\otimes w$ is $\chi_1 \chi_2$ which, written additively, gives the equality in the thesis.
%\end{proof}

\begin{ex}\label{ex_weight4}
  Building upon example \ref{ex_weight1}, consider the 
  representation $V\otimes V$ of $T$.
  If $\tau$ denotes the element
  \[
    \begin{pmatrix}
      t_1 & 0\\
      0 & t_2
    \end{pmatrix}
    \in T,
  \]
  then
  \begin{align*}
    &\tau \cdot (v_1 \otimes v_1) = (\tau \cdot v_1) \otimes
    (\tau \cdot v_1) = (t_1 v_1) \otimes (t_1 v_1) =
    t_1^2 (v_1 \otimes v_1),\\
    &\tau \cdot (v_1 \otimes v_2) = (\tau \cdot v_1) \otimes
    (\tau \cdot v_2) = (t_1 v_1) \otimes (t_2 v_2) =
    t_1 t_2 (v_1 \otimes v_2),
  \end{align*}
  so $\weight (v_1 \otimes v_1) = (2,0)$ and
  $\weight (v_1 \otimes v_2) = (1,1)$.
  Note that
  \begin{align*}
    &\weight (v_1) + \weight (v_1) = (1,0) + (1,0) = (2,0) =
    \weight (v_1 \otimes v_1),\\
    &\weight (v_1) + \weight (v_2) = (1,0) + (0,1) = (1,1) =
    \weight (v_1 \otimes v_2).
  \end{align*}
  % Similarly, $\weight (v_2 \otimes v_1) = (1,1)$ and
  % $\weight (v_2 \otimes v_2) = (0,2)$.
\end{ex}

\begin{defin}
Let $V,W$ be two representations of $T$ over $\KK$. A $\KK$-linear map $\varphi \colon V\rightarrow W$ is $T$-\emph{equivariant}, or a map of representations of $T$, if $\forall \tau\in T$, $\forall v\in V$ we have $\varphi (\tau \cdot v) = \tau \cdot \varphi (v)$.
\end{defin}
\begin{prop}\label{weight_in_equivariant_map}
Let $\varphi \colon V\rightarrow W$ be $T$-equivariant. If $v\in V$ is a weight vector and $\varphi (v)\neq 0$, then $\varphi (v)$ is a weight vector in $W$ and $\weight (\varphi (v)) = \weight (v)$.
\end{prop}
%\begin{proof}
%Suppose $v$ has weight $\chi$. Then $\forall \tau \in T$
%\[\tau \cdot \phi (v) = \phi (\tau \cdot v) = \phi (\chi (\tau) v) = \chi (\tau) \phi (v).\]
%This shows the weight of $\phi (v)$ is also $\chi$.
%\end{proof}

\begin{ex}\label{ex_weight5}
  In the context of examples \ref{ex_weight1}, \ref{ex_weight2} and
  \ref{ex_weight4}, introduce a map
  \begin{align*}
    \varphi \colon V\otimes V &\longrightarrow \Sym^2 (V)\\
    v_i \otimes v_j &\longmapsto v_i v_j.
  \end{align*}
  Note that
  \begin{align*}
    \varphi (\tau \cdot (v_1 \otimes v_1)) &= \varphi (t_1^2 
    (v_1 \otimes v_1)) = t_1^2 \varphi (v_1 \otimes v_1) =\\
    &= t_1^2 v_1^2 = \tau \cdot v_1^2 = \tau \cdot \varphi 
    (v_1 \otimes v_1),\\
    \varphi (\tau \cdot (v_1 \otimes v_2)) &= \varphi (t_1 t_2 
    (v_1 \otimes v_2)) = t_1 t_2 \varphi (v_1 \otimes v_2) =\\
    &= t_1 t_2 v_1 v_2 = \tau \cdot v_1 v_2 = \tau \cdot \varphi 
    (v_1 \otimes v_2).
  \end{align*}
  A similar behavior holds for $v_2 \otimes v_1$ and
  $v_2 \otimes v_2$, so the map $\varphi$ is $T$-equivariant.
  Moreover,
  \begin{align*}
    &\weight (v_1 \otimes v_1) = (2,0) = \weight (v_1^2),\\
    &\weight (v_1 \otimes v_2) = (1,1) = \weight (v_1 v_2),
  \end{align*}
  and similarly for the other basis vectors.
\end{ex}

If $V$ is a representation of $T$, then the dual (or contragredient) representation is the vector space $V^* := \Hom_\KK (V,\KK)$. An element $\tau \in T$ acts on $V^*$ by $[\tau \cdot f](v) = f (\tau^{-1} \cdot v)$,  $\forall f\in V^*$, $\forall v\in V$.

\begin{prop}\label{weight_in_dual}
Let $V$ be a finite dimensional representation of $T$ with a basis of weight vectors $\{v_1,\ldots,v_r\}$. Then the dual basis $\{v_1^*,\ldots,v_r^*\}$ of $V^*$ is a basis of weight vectors and $\weight(v_i^*) = -\weight (v_i)$, $\forall i\in\{1,\ldots,r\}$.
\end{prop}
%\begin{proof}
%That $v_1^*,\ldots,v_n^*$ form a basis is well known. Recall that, for $v = \sum_{j=1}^n c_j v_j\in V$, we have $v_i^* (v)=c_i$.
%Now $\forall \tau \in T$ and $v = \sum_{j=1}^n c_j v_j\in V$,
%\begin{align*}
%(\tau \cdot v_i^*)(v) &{}= v_i^* (\tau^{-1} \cdot v) = v_i^* \left(\sum_{j=1}^n c_j (\tau^{-1} \cdot v_j)\right) =\\
%&{}= v_i^* \left(\sum_{j=1}^n c_j \chi_j (\tau^{-1}) v_j\right) = c_i \chi_i (\tau^{-1}) = c_i \chi_i^{-1} (\tau).
%\end{align*}
%At the same time we have
%\[(\chi_i^{-1} (\tau) v_i^*) (v) = \chi_i^{-1} (\tau) v_i^* (v) = \chi_i^{-1} (\tau) c_i,\]
%which shows $\tau \cdot v_i^* = \chi_i^{-1} (\tau) v_i^*$ and hence the weight of $v_i^*$ is $\chi_i^{-1}$.
%\end{proof}

\begin{ex}\label{ex_weight6}
  Using the setup of example \ref{ex_weight1}, consider
  the representation $V^*$ and the dual basis $\{v_1^*, v_2^*\}$.
  We have
  \begin{align*}
    &(\tau \cdot v_1^*) (v_1) = v_1^* (\tau^{-1} \cdot v_1) = 
    v_1^* (t_1^{-1} v_1) = t_1^{-1},\\
    &(\tau \cdot v_1^*) (v_2) = v_1^* (\tau^{-1} \cdot v_2) = 
    v_1^* (t_2^{-1} v_2) = 0.
  \end{align*}
  Thus $\tau \cdot v_1^* = t_1^{-1} v_1^*$ and therefore
  \[\weight (v_1^*) = (-1,0) = -\weight (v_1).\]
  Similarly $\tau \cdot v_2^* = t_2^{-1} v_2^*$ and
  \[\weight (v_2^*) = (0,-1) = -\weight (v_2).\]
\end{ex}

\subsection{Compatible actions of tori on rings and modules}\label{interactions}
Let $M$ be an $A$-module.
\begin{defin}
A $\KK$-linear action of $T$ on $M$ is \emph{compatible with grading} if $\forall \tau \in T$, $\forall d\in \ZZ^m$ we have $\tau \cdot M_d \subseteq M_d$.
\end{defin}

\begin{defin}
A $\KK$-linear action of $T$ on $M$ is \emph{compatible with multiplication} if $\forall d,d'\in \ZZ^m$ the map induced by multiplication on $A_d \otimes_\KK M_{d'} \rightarrow M_{d+d'}$ is $T$-equivariant. In other words $\forall \tau \in T$, $\forall a\in A_d$, $\forall m\in M_{d'}$ we have $\tau \cdot (am) = (\tau \cdot a) (\tau \cdot m)$.
\end{defin}

All actions we consider are compatible with grading and multiplication. Indeed we can introduce the category $\mod_{\righttoleftarrow T} A$ with
\begin{itemize}[wide]
\item objects: all finitely generated graded $A$-modules with a $\KK$-linear action of $T$ compatible with grading and multiplication;
\item maps: all homogeneous $A$-linear $T$-equivariant maps,
\end{itemize}
and work in this category unless otherwise stated.
\begin{remark}
For any group $G$ with a $\KK$-linear action on $A$ that is compatible with grading and multiplication, we could similarly define a category $\mod_{\righttoleftarrow G} A$.
\end{remark}

A homogeneous polynomial $p\in A$ of degree $d\in\ZZ^m$ is said to be a weight vector if it is a weight vector in $A_d$. We always assume that all variables in $A$ are weight vectors, which can always be obtained up to a linear change of variables in $A$.
\begin{prop}\label{weight_of_term}
If each $x_i\in A$ is a weight vector, then every term $t\in \TT^n$ is a weight vector of $A$. Moreover, if $t = x_1^{\alpha_1} \ldots x_n^{\alpha_n}$, then
\[\weight (t) = \sum_{i=1}^n \alpha_i \weight (x_i).\]
\end{prop}
\begin{proof}
Suppose the variable $x_i$ has degree $d_i\in\ZZ^m$. The degree of $t$ is $d = \sum_{i=1}^n \alpha_i d_i$. The map induced by multiplication on
\[A_{d_1}^{\otimes \alpha_1} \otimes_\KK \ldots \otimes_\KK A_{d_n}^{\otimes \alpha_n} \longrightarrow A_d\]
is $T$-equivariant because the action of $T$ on $A$ is compatible with multiplication. By proposition \ref{weight_in_equivariant_map}, $t$ is a weight vector in $A_d$ because it is the image of a tensor product of weight vectors under the map above. Then the formula for $\weight (t)$ follows from propositions \ref{weight_in_equivariant_map} and \ref{weight_in_tensor_product}.
\end{proof}

The previous proposition implies that every graded component $A_d$ of $A$ has a basis of weight vectors consisting of all terms of degree $d$, i.e. the elements of the set $\TT^n \cap A_d$. Moreover, if $\chi$ is a weight of $A_d$, then the weight space $(A_d)_\chi$ has a basis of weight vectors consisting of all terms of degree $d$ and weight $\chi$, i.e. the elements of the set $\TT^n \cap (A_d)_\chi$.

\begin{ex}\label{ex1}
%We set up an example that will be gradually explored as we proceed through our work.

Let $V=\CC^3$ and $G=\GL (V)$. Let $A=\Sym (V)$, the symmetric algebra over $V$; $A$ is a $\ZZ$-graded $\CC$-algebra with graded components $A_d = \Sym^d (V)$ given by the symmetric powers of $V$. The group $G$ has a natural $\CC$-linear action on each component $\Sym^d (V)$ which is compatible with multiplication.

By choosing a basis $\{v_1,v_2,v_3\}$ of $V$, we can identify $G$ with $\GL_3 (\CC)$. If $T$ denotes the maximal torus of $G$ corresponding to diagonal matrices in $\GL_3 (\CC)$, then the elements $v_1,v_2,v_3$ form a basis of weight vectors of $V$. Denoting $x_1,x_2,x_3$ the elements $v_1,v_2,v_3 \in \Sym^1 (V) = V$, we can also identify $A$ with the standard graded polynomial ring $\CC [x_1,x_2,x_3]$. For $g \in G$ and $x_1^{\alpha_1} x_2^{\alpha_2} x_3^{\alpha_3}$ a term in $A$, the action of $G$ is determined by
\[g \cdot x_1^{\alpha_1} x_2^{\alpha_2} x_3^{\alpha_3} = (g\cdot x_1)^{\alpha_1} (g\cdot x_2)^{\alpha_2} (g\cdot x_3)^{\alpha_3}.\]
Moreover, the term $x_1^{\alpha_1} x_2^{\alpha_2} x_3^{\alpha_3}$ is a weight vector with weight $(\alpha_1,\alpha_2,\alpha_3) \in \ZZ^3$. In particular, the variables $x_1,x_2,x_3$ have weight $(1,0,0)$, $(0,1,0)$ and $(0,0,1)$.
\end{ex}

\begin{prop}\label{weight_of_poly}
If $p\in A_d$ is a weight vector, then $\forall t\in \Supp (p)$ we have $\weight (t) =\weight (p)$.
\end{prop}
\begin{proof}
Using the basis of terms of $A_d$, we can write
\[p = \sum_{t\in \TT^n\cap A_d} c_t t.\]
Suppose $p$ has weight $\chi$, so that $p\in (A_d)_\chi$. If $t\in \Supp (p)$, then $c_t \neq 0$; this forces $t\in (A_d)_\chi$, otherwise we would have $p \notin (A_d)_\chi$.
\end{proof}

\begin{ex}
  Let $A=\CC [x_1,x_2,x_3]$ with the standard grading.
  Let $T = \CC^\times$ act on $A$ by setting, $\forall t\in T$,
  $\forall p\in A$,
  \[t\cdot p (x_1,x_2,x_3) := p (t x_1,t x_2,t x_3).\]
  If $p\in A$ is homogeneous of degree $d$, then
  $t\cdot p = t^d p$, i.e. $p$ is a weight vector of weight $d$.
  By proposition \ref{weight_of_poly}, all terms in the support of
  $p$ must be weight vectors of the same degree.
  This is clear since all terms in the support of $p$ have
  degree $d$.
\end{ex}

Let $M$ be an object in $\mod_{\righttoleftarrow T} A$. A homogeneous element $m\in M$ of degree $d$ is a weight vector if it is a weight vector in $M_d$.
Let $\mathfrak{m}$ be the maximal ideal of $A$ generated by the variables $x_1,\ldots,x_n$. The quotient $M/\mathfrak{m} M$ is a finite dimensional graded $\KK$-vector space with the action $\tau \cdot (m+\mathfrak{m} M) = (\tau \cdot m)+\mathfrak{m} M$, $\forall m \in M$, $\forall \tau \in T$. Moreover, it is an object in $\mod_{\righttoleftarrow T} A$. Observe that if $m$ is a weight vector in $M$ and $m+\mathfrak{m} M$ is non zero, then $m+\mathfrak{m} M$ is a weight vector in $M/\mathfrak{m} M$.

\begin{ex}\label{ex2}
In the setting of example \ref{ex1}, the ideal $\mathfrak{m} = (x_1,x_2,x_3)$ and the quotient $A/\mathfrak{m}$ are examples of modules in $\mod_{\righttoleftarrow T} A$. From the point of view of representation theory, $\mathfrak{m}$ is generated by a copy of $V$ in degree 1, whereas $A/\mathfrak{m}$ is the trivial representation of $G$ in degree 0.
\end{ex}

If $M,N$ are objects in $\mod_{\righttoleftarrow T} A$, there is a natural action of $T$ on $\Hom_A (M,N)$ given by $[\tau \cdot \psi] (m) := \psi (\tau^{-1} \cdot m)$, $\forall \tau \in T$, $\forall \psi \in \Hom_A (M,N)$, $\forall m\in M$. This action restricts to $\sideset{^*}{}\Hom_A (M,N)$ as we will illustrate in the next result.

\begin{prop}
Let $N$ be an object of $\mod_{\righttoleftarrow T} A$. The restriction of the functor $\sideset{^*}{} \Hom_A(-,N)$ to $\mod_{\righttoleftarrow T} A$ is an endofunctor on $\mod_{\righttoleftarrow T} A$.
\end{prop}
\begin{proof}
We must show that applying $\sideset{^*}{} \Hom_A(-,N)$ to an object or morphism in $\mod_{\righttoleftarrow T} A$ will land us again in $\mod_{\righttoleftarrow T} A$. We do this in three steps. Let $M,M_1,M_2,N$ be objects in $\mod_{\righttoleftarrow T} A$ and $\varphi \colon M_1\rightarrow M_2$ a morphism in $\mod_{\righttoleftarrow T} A$. Recall that the graded component of $\sideset{^*}{} \Hom_A(M,N)$ of degree $d\in\ZZ^m$ is
\[\Hom_d (M,N) = \{\psi\in \Hom_A (M, N) \mid \forall d'\in\ZZ^m,\ \psi(M_{d'}) \subseteq N_{d+d'} \}.\]

\begin{description}[wide,itemsep=1ex]
\item[Claim 1] The action of $T$ on $\Hom_A (M,N)$ restricts to $\Hom_d (M,N)$, for all $d\in \ZZ^m$, i.e. the action of $T$ on $\sideset{^*}{} \Hom_A (M,N)$ is compatible with the grading.

For all $d'\in\ZZ^m$, $\tau\in T$, and $\psi \in \Hom_d (M,N)$, we have that
\[[\tau \cdot \psi] (M_{d'}) = \psi (\tau^{-1} \cdot M_{d'}) \subseteq \psi (M_{d'}) \subseteq N_{d+d'}\]
because the action of $T$ on $M$ is compatible with the grading. This shows $\tau \cdot \psi \in \Hom_d (M,N)$.
\item[Claim 2] The action of $T$ on $\sideset{^*}{} \Hom_A (M,N)$ is compatible with multiplication.

Let $d,d'\in \ZZ^m$ be arbitrary degrees. For all $\tau\in T$, $a\in A_d$, $\psi \in \Hom_{d'} (M,N)$ and $m\in M$, we get
\begin{gather*}
[\tau \cdot (a\psi) ] (m) = a\psi (\tau^{-1} \cdot m) = \psi (a (\tau^{-1} \cdot m)) = \psi (\tau^{-1} \cdot ( (\tau \cdot a) m)) =\\
= [\tau \cdot \psi] ( (\tau \cdot a) m) = (\tau \cdot a) [\tau \cdot \psi] (m) = [(\tau \cdot a) (\tau \cdot \psi)] (m)
\end{gather*}
because the action of $T$ on $M$ is compatible with multiplication. This shows that $\tau \cdot (a\psi) = (\tau \cdot a) (\tau \cdot \psi)$.
\item[Claim 3] The map $\sideset{^*}{} \Hom_A (\varphi,N) \colon \sideset{^*}{} \Hom_A (M_2,N)\rightarrow \sideset{^*}{} \Hom_A (M_1,N)$ is $T$-equivariant.

For simplicity denote $\sideset{^*}{} \Hom_A (\varphi,N)$ by $\varphi^*$. For all $\tau \in T$, $\psi \in \sideset{^*}{} \Hom_A (M_2,N)$ and $m\in M_1$, we obtain
\begin{gather*}
[\varphi^* (\tau \cdot \psi)] (m) = (\tau \cdot \psi)( \varphi (m)) = \psi (\tau^{-1} \cdot \varphi (m)) =\\
\psi (\varphi (\tau^{-1} \cdot m)) = [\varphi^* (\psi)] (\tau^{-1} \cdot m) = [\tau \cdot \varphi^* (\psi)] (m) 
\end{gather*}
because $\varphi$ is $T$-equivariant. This shows $\varphi^* (\tau \cdot \psi) = \tau \cdot \varphi^* (\psi)$.
\end{description}
\end{proof}

We are interested in applying the functor $\sideset{^*}{} \Hom_A(-,N)$ when $N=A$.

\begin{defin}\label{dual}
For an object $M$ in $\mod_{\righttoleftarrow T} A$, we set $M^{\vee} := \sideset{^*}{} \Hom_A (M,A)$ and call it the \emph{dual} of $M$ in $\mod_{\righttoleftarrow T} A$. For a morphism $\varphi \colon M_1\rightarrow M_2$ in $\mod_{\righttoleftarrow T} A$, we set $\varphi^{\vee} := \sideset{^*}{} \Hom_A (\varphi,A)$ and call it the \emph{dual} of $\varphi$ in $\mod_{\righttoleftarrow T} A$.
\end{defin}

\subsection{Free modules}\label{free_mod}
In the category of finitely generated $\ZZ^m$-graded $A$-modules, every free module has the form
\[\bigoplus_{d\in\ZZ^m} A(-d)^{\beta_d},\]
where $\beta_d\in\mathbb{N}$ $\forall d\in\ZZ^m$, and only finitely many $\beta_d$ are non zero.
Here $A(-d)$ denotes a free module of rank one generated in degree
$d$, i.e. with its generator $1_A$ artificially shifted to 
degree $d$.
In $\mod_{\righttoleftarrow T} A$, for the same choice of natural numbers $\beta_d$, there may be more than one isomorphism class of free modules depending on how $T$ acts on a homogeneous basis
of the free module.

Let $V$ be a finite dimensional $\ZZ^m$-graded representation of $T$, i.e. a finite dimensional $\ZZ^m$-graded $\KK$-vector space with a $\KK$-linear action of $T$ that is compatible with the grading. The $\KK$-vector space $V \otimes_\KK A$ is naturally graded by
\[(V \otimes_\KK A)_d := \bigoplus_{d'+d'' = d} V_{d'} \otimes_\KK A_{d''}.\]
Then $V \otimes_\KK A$ becomes a graded $A$-module with multiplication given by $a (v\otimes b) := v\otimes (ab)$, $\forall a,b\in A$, $\forall v\in V$. The usual action of $T$ on the tensor product $V \otimes_\KK A$ is compatible with grading and multiplication so $V \otimes_\KK A$ is an object in $\mod_{\righttoleftarrow T} A$.

\begin{ex}\label{ex3}
Using the setup of example \ref{ex1}, for each integer $i\in\ZZ$ we have a different one dimensional free module in $\mod_{\righttoleftarrow T} A$  that is generated in degree $d\in\ZZ$. These modules are given by
$(\bigwedge^3 V)^{\otimes i} \otimes_\CC A(-d)$, if $i\geq 0$, and by $(\bigwedge^3 V^*)^{\otimes -i} \otimes_\CC A(-d)$, if $i< 0$.
\end{ex}

\begin{defin}
A \emph{free module} in $\mod_{\righttoleftarrow T} A$ is an object which is isomorphic to $V\otimes_\KK A$ for some finite dimensional graded representation $V$ of $T$.
\end{defin}

The rank of $V\otimes_\KK A$ is $\dim_\KK V$. If $\dim_\KK V_d = \beta_d$, then
\[V\otimes_\KK A \cong \bigoplus_{d\in\ZZ^m} A(-d)^{\beta_d},\]
as graded $A$-modules (disregarding the action of $T$).

\begin{prop}\label{free_mod_structure}
If $F$ is a free module in $\mod_{\righttoleftarrow T} A$, then $F \cong (F/\mathfrak{m} F) \otimes_\KK A$.
\end{prop}

\begin{proof}
By definition of free module, $\exists V$, finite dimensional graded representation of $T$, such that $F\cong V\otimes_\KK A$. Notice that
\begin{gather*}
F / \mathfrak{m} F \cong F \otimes_A (A/\mathfrak{m}) \cong (V\otimes_\KK A) \otimes_A (A/\mathfrak{m}) \cong\\
\cong V \otimes_\KK (A \otimes_A (A/\mathfrak{m})) 
\cong V \otimes_\KK (A/\mathfrak{m}) \cong V \otimes_\KK \KK \cong V
\end{gather*}
where each step holds as isomorphism of graded representations of $T$. Therefore $F \cong (F/\mathfrak{m} F) \otimes_\KK A$.
\end{proof}

By the previous proposition, $F_d$ is isomorphic to $\bigoplus_{d'+d''=d} (F/\mathfrak{m} F)_{d'} \otimes_\KK A_{d''}$ as a representation of $T$. Recall that $A_{d''}$ has a basis of weight vectors consisting of all terms of degree $d''$, so its weights can be recovered using proposition \ref{weight_of_term}. The weights in a tensor product can be obtained via proposition \ref{weight_in_tensor_product}. Thus the weights of $F_d$ can be described, for any degree $d\in \ZZ^m$, as long as the weights in $F/\mathfrak{m}F$ are known.

Let $F$ be a free module of rank $s$ in $\mod_{\righttoleftarrow T} A$. A homogeneous basis of weight vectors of $F$ is a homogeneous basis $\tilde{\mathcal{F}} = \{\tilde{f}_1,\ldots,\tilde{f}_s\}$ of $F$ such that each $\tilde{f}_j$ is a weight vector of $F$. We adopt the convention of decorating homogeneous bases of weight vectors and their elements with a tilde.
Notice that the residue classes $\tilde{f}_1+\mathfrak{m}F,\ldots,\tilde{f}_s+\mathfrak{m}F$ form a basis of weight vectors of $F/\mathfrak{m} F$ and $\weight (\tilde{f}_j+\mathfrak{m}F) = \weight (\tilde{f}_j)$.

As for graded $A$-modules, we can introduce a graded Hom functor $\sideset{^*}{} \Hom_\KK (-,-)$ in the category of graded $\KK$-vector spaces (or graded representations of $T$). The dual of a graded $\KK$-vector space (or representation of $T$) $V$ is $V^* := \sideset{^*}{} \Hom_\KK (V,\KK)$. This construction allows us to identify the dual of a free object in $\mod_{\righttoleftarrow T} A$.

\begin{remark}
Observe that for every graded $\KK$-vector space $V$ and degree $d\in\ZZ^m$, $(V^*)_d = \Hom_\KK (V_{-d},\KK)$.
\end{remark}

\begin{prop}\label{dual_free_mod_structure}
If $F$ is a free module in $\mod_{\righttoleftarrow T} A$, then $F^\vee \cong (F/\mathfrak{m} F)^* \otimes_\KK A$.
\end{prop}
\begin{proof}
The thesis follows from the following chain of isomorphisms in $\mod_{\righttoleftarrow T} A$:
\begin{gather*}
F^\vee \cong \sideset{^*}{_A} \Hom (F,A) \cong \sideset{^*}{_A} \Hom ((F/\mathfrak{m} F) \otimes_\KK A,A) \cong\\
\cong \sideset{^*}{_A} \Hom (A,\sideset{^*}{_\KK} \Hom (F/\mathfrak{m} F,A)) \cong \sideset{^*}{_\KK} \Hom (F/\mathfrak{m} F,A) \cong (F/\mathfrak{m} F)^* \otimes_\KK A.
\end{gather*}
\end{proof}

\begin{ex}
  Let us consider again free modules as in example \ref{ex3}.
  If $F = \bigwedge^3 V \otimes_\CC A(-d)$,
  then $F/\mathfrak{m} F \cong \bigwedge^3 V$ (cf. proposition
  \ref{free_mod_structure}).
  Furthermore, $(F/\mathfrak{m} F)^* \cong \bigwedge^3 V^*$ and
  therefore $F^\vee \cong \bigwedge^3 V^* \otimes_\CC A(d)$
  (cf. proposition \ref{dual_free_mod_structure}).
\end{ex}

\begin{remark}
The proofs of proposition \ref{free_mod_structure} and \ref{dual_free_mod_structure} use standard isomorphisms, like associativity of the tensor product or adjunction of Hom and tensor product. Since those isomorphisms hold in the categories of graded $A$-modules and representations of $T$, they immediately transfer to $\mod_{\righttoleftarrow T} A$.
\end{remark}

For every free module $V\otimes_\KK A$ in $\mod_{\righttoleftarrow T} A$, there is a natural injection of graded representations of $T$, $i_V\colon  V \to V\otimes_\KK A$ sending $v$ to $v\otimes 1_A$. Free modules in $\mod_{\righttoleftarrow T} A$ satisfy the following universal property.

\begin{prop}\label{univ_prop_free_mod}
Let $V$ be a finite dimensional graded representation of $T$. For every module $M$ in $\mod_{\righttoleftarrow T} A$ and any homogeneous map $\hat{\psi} \colon V\to M$ of graded representations of $T$, there exists a unique morphism $\psi \colon V\otimes_\KK A \to M$ in $\mod_{\righttoleftarrow T} A$ such that $\hat{\psi} = \psi \circ i_V$.
\end{prop}

\begin{proof}
  Consider the map
  \begin{align*}
    V\times A &\longrightarrow M\\
    (v,a) &\longmapsto a \hat{\psi} (v).
  \end{align*}
  This map is obviously $\KK$-bilinear so it induces
  a $\KK$-linear map $\psi \colon V\otimes_\KK A \to M$,
  which clearly satisfies $\hat{\psi} = \psi \circ i_V$.
  The proof that $\psi$ is an equivariant map of graded $A$-modules
  is straightforward and we leave it to the reader.
\end{proof}

\begin{prop}\label{equivariant_resolution}
Every object $M$ in $\mod_{\righttoleftarrow T} A$ admits a finite minimal free resolution in $\mod_{\righttoleftarrow T} A$, i.e. an exact complex
\[0\to F_n \xrightarrow{d_n} F_{n-1} \to \ldots \to F_1 \xrightarrow{d_1} F_0 \xrightarrow{d_0} M \to 0\]
of  maps and modules in $\mod_{\righttoleftarrow T} A$, such that each $F_i$ is free.
\end{prop}
\begin{proof}
Consider the projection $\pi \colon M\to M/\mathfrak{m}M$. As a map of graded representations of $T$, $\pi$ admits a section, i.e. a map $\hat{d}_0 \colon M/\mathfrak{m}M \to M$ such that $\pi\circ\hat{d}_0 = \id_{M/\mathfrak{m}M}$. By the universal property of free modules in $\mod_{\righttoleftarrow T} A$ (prop. \ref{univ_prop_free_mod}), there exists a unique map $d_0\colon (M/\mathfrak{m}M)\otimes_\KK A\to M$ such that $\hat{d}_0 = d_0\circ i_{M/\mathfrak{m}M}$. Set $F_0 := (M/\mathfrak{m}M)\otimes_\KK A$. Observe that $d_0$ maps a homogeneous basis of $F_0$ to a minimal generating set of $M$, which makes $d_0$ surjective.

Now suppose that all maps and modules in the complex have been constructed up to a certain index $i\geq 0$. To obtain $d_{i+1} \colon F_{i+1}\to F_i$, let $M_i := \ker d_i$ and repeat the previous construction using $M_i$ instead of $M$. As a result, $d_{i+1}$ will map a homogeneous basis of $F_{i+1}$ to a minimal generating set of $M_i$, guaranteeing exactness of the complex.

Notice that this construction will produce a complex $F_\bullet$ that is also a minimal free resolution of $M$ in the category of finitely generated graded $A$-modules, hence the process stops when $i>n$.
\end{proof}

Since minimal free resolutions are unique up to isomorphisms of complexes, we can reinterpret the result of proposition \ref{equivariant_resolution} by saying that any minimal free resolution of $M$ as an $A$-module carries an action of $T$ that commutes with the differentials.

\begin{remark}
Proposition \ref{equivariant_resolution} holds in $\mod_{\righttoleftarrow G} A$ for other groups $G$, as long as the category of finite dimensional graded representations of $G$ over $\KK$ is semisimple. This guarantees the existence of sections used in the proof.
\end{remark}

\begin{ex}\label{ex4}
Continuing example \ref{ex1}, we look at the complex of free modules in $\mod_{\righttoleftarrow G} A$ given by
\[0\to \bigwedge^3 V \otimes_\CC A(-3) \xrightarrow{d_3} \bigwedge^2 V \otimes_\CC A(-2) \xrightarrow{d_2} V \otimes_\CC A(-1) \xrightarrow{d_1} A\]
where the maps are defined as follows:
\begin{align*}
d_j \colon \bigwedge^j V \otimes_\CC A(-j) &\longrightarrow \bigwedge^{j-1} V \otimes_\CC A(-j+1)\\
v_{i_1} \wedge \ldots \wedge v_{i_j} &\longmapsto \sum_{k=1}^j (-1)^{k+1} x_k v_{i_1} \wedge \ldots \wedge v_{i_{k-1}} \wedge v_{i_{k+1}} \ldots \wedge v_{i_j}.
\end{align*}
This is in fact the Koszul complex $K_\bullet (x_1,x_2,x_3)$ on the variables of $A$, which is a minimal free resolution of $A/\mathfrak{m}$ \cite[Cor. 1.6.14]{MR1251956}. The action of $G$ on each free module is dictated by the exterior powers of $V$. The differential $d_j$ maps the generators of $\bigwedge^j V \otimes_\CC A(-j)$, which live in degree $j$, into $\bigwedge^{j-1} V \otimes_\CC A_1$. Notice that $A_1 = V$ so
\[\bigwedge^{j-1} V \otimes_\CC A_1 \cong \bigwedge^j V \oplus \mathbb{S}_{(2,1^{j-2})} V,\]
where $\mathbb{S}_{(2,1^{j-2})}$ denotes a Schur functor \cite[\S2.1]{MR1988690} and both summands on the right hand side are irreducible representations of $G$. Such a decomposition can be obtained using Pieri's formula \cite[Cor. 2.3.5]{MR1988690}. By Schur's lemma \cite[Ch. XVII, Prop. 1.1]{MR1878556}, we deduce there is a unique $G$-equivariant map $d_j$ up to multiplication by a scalar.
This map can be described as the diagonal map between exterior powers \cite[p. 3]{MR1988690}.
\end{ex}

\section{Propagating weights}\label{propagation}
This section contains a detailed description of how the weights for the action of a torus propagate along an equivariant map of free modules. We build gradually towards a fairly general result by first examining the case of maps expressed with respect to bases of weight vectors for the domain and codomain (\S\ref{bases_of_weight_vectors}). The main theorem is presented in section \ref{general_case}.

\subsection{Bases of weight vectors}\label{bases_of_weight_vectors}
Suppose $\varphi\colon E \rightarrow F$ is a map of free modules in $\mod_{\righttoleftarrow T} A$. Assume $E$ has rank $r$ and a homogeneous basis of weight vectors $\tilde{\mathcal{E}} = \{\tilde{e}_1,\dots, \tilde{e}_r\}$, and that $F$ has rank $s$ and a homogeneous basis of weight vectors $\tilde{\mathcal{F}} = \{\tilde{f}_1,\dots, \tilde{f}_s\}$. The goal of this section is to explain how to recover $\weight (\tilde{e}_1),\ldots,\weight (\tilde{e}_r)$ if we assume that $\weight (\tilde{f}_1),\ldots,\weight (\tilde{f}_s)$ are known. This will provide a complete list of weights of $E$ and hence it will identify $E$ as a representation of $T$.

\begin{prop}\label{weight_in_homogeneous_basis}
Let  $\varphi\colon E \rightarrow F$ be a map of free modules in $\mod_{\righttoleftarrow T} A$. Let $\tilde{e}\in E$ be a homogeneous weight vector and assume $\varphi (\tilde{e}) \neq 0$. If $\varphi (\tilde{e}) = \sum_{j=1}^s p_j \tilde{f}_j$ for some homogeneous polynomials $p_1,\ldots,p_s\in A$, then each non zero $p_j$ is a weight vector in $A$ and
\[\weight (\tilde{e}) = \weight (p_j) + \weight (\tilde{f_j}).\]
\end{prop}
\begin{proof}
By proposition \ref{weight_in_equivariant_map}, $\varphi (\tilde{e})$ is a weight vector and $\weight (\tilde{e}) = \weight (\varphi(\tilde{e}))$. Suppose $\tilde{e}$ has weight $\chi\in X(T)$ and $\forall j\in \{1,\ldots,s\}$, $\tilde{f}_j$ has weight $\chi_i\in X(T)$. Then $\forall \tau\in T$
\[\tau \cdot \varphi (\tilde{e}) = \varphi (\tau \cdot \tilde{e}) = \varphi (\chi(\tau) \tilde{e}) = \chi(\tau) \varphi (\tilde{e}) = \sum_{j=1}^s (\chi (\tau) p_j) \tilde{f}_j\]
and, at the same time,
\[\tau \cdot \varphi (\tilde{e}) = \sum_{j=1}^s (\tau \cdot p_j) (\tau \cdot \tilde{f}_j) = \sum_{j=1}^s ( \tau \cdot p_j) (\chi_j (\tau)\tilde{f}_j).\]
Because $\tilde{\mathcal{F}}$ is a homogeneous basis of $F$, we deduce that $\chi_j (\tau) (\tau\cdot p_j) = \chi (\tau) p_j$ so
\[ \tau\cdot p_j = \chi(\tau) \chi_j(\tau)^{-1} p_j = \chi \chi_j^{-1}(\tau) p_j\]
$\forall j\in \{1,\ldots,s\}$. This implies that each non zero $p_j$ is a weight vector with weight $\chi \chi_j^{-1}$. Additively, we may write $\weight (p_j) = \weight (\tilde{e})-\weight (\tilde{f}_j)$ which gives the equality in the thesis.
\end{proof}

\begin{coro}\label{weight_in_homogeneous_basis_via_term}
If $\hat{t} \tilde{f}_{\hat{\jmath}} \in \Supp_{\tilde{\mathcal{F}}} (\varphi (\tilde{e}))$, then
\[\weight (\tilde{e}) = \weight (\hat{t}) + \weight (\tilde{f}_{\hat{\jmath}}).\]
\end{coro}
\begin{proof}
If $\varphi (\tilde{e}) = \sum_{j=1}^s p_j \tilde{f}_j$, the hypothesis implies $\hat{t} \in \Supp (p_{\hat{\jmath}})$; in particular, $p_{\hat{\jmath}}\neq 0$ so that $\varphi (\tilde{e}) \neq 0$ as well. By the previous proposition, $p_{\hat{\jmath}}$ is a weight vector and $\weight (p_{\hat{\jmath}}) = \weight (\hat{t})$ by proposition \ref{weight_of_poly}. The thesis follows using the weight equality in the previous proposition.
\end{proof}

When applying corollary \ref{weight_in_homogeneous_basis_via_term}, any term $\hat{t} \tilde{f}_{\hat{\jmath}} \in \Supp_{\tilde{\mathcal{F}}} (\varphi (\tilde{e}))$ may be used. %How should we instruct a machine to select a term from the support of $\varphi (\tilde{e})$?
In a computational setting, the natural choice of term is $\LT (\varphi (\tilde{e}))$, the leading term of $\varphi (\tilde{e})$ in a module term ordering on $\TT^n \langle \tilde{\mathcal{F}}\rangle$.

If $\varphi (\tilde{e}) = 0$, then we cannot recover the weight of $\tilde{e}$ using the map $\varphi$. Indeed, if we expect to use $\varphi$ to extract information about weight vectors of $E$, then $\varphi$ should preserve such information.

\begin{defin}\label{minimality_def}
Let $\varphi \colon E\rightarrow F$ be a map of free modules in
$\mod_{\righttoleftarrow T} A$ and let $\mathcal{E} = \{e_1,\ldots,e_r\}$
be a homogeneous basis of $E$.
The map $\varphi$ is called \emph{minimal} if $\varphi (e_1),\ldots,
\varphi (e_r)$ are minimal generators of $\im \varphi$.
\end{defin}

It is clear that the definition above does not depend on the choice
of homogeneous basis of $E$.
As we will see from the next proposition, the notion of minimal
map is well suited to preserve information on weight vectors.

\begin{prop}\label{minimality}
Let $\varphi \colon E\rightarrow F$ be a map of free modules in $\mod_{\righttoleftarrow T} A$. Let $\mathcal{E} = \{e_1,\ldots,e_r\}$ be a homogeneous basis of $E$ and let $\langle \mathcal{E}\rangle_\KK$ denote the $\KK$-vector subspace of $E$ generated by the elements of $\mathcal{E}$.
\begin{enumerate}[label=\Roman*.,ref=\Roman*,wide]
\item\label{min1} If $\varphi$ is minimal, then the restriction
  of $\varphi$ to $\langle \mathcal{E}\rangle_\KK$ is injective.
\item\label{min2} If $E$ is generated in a single degree and the
restriction of $\varphi$ to $\langle \mathcal{E}\rangle_\KK$ is
injective, then $\varphi$ is minimal.
\end{enumerate}
\end{prop}

\begin{proof}
It is obvious that $\varphi (e_1), \ldots, \varphi (e_r)$ generate $\im\varphi$. Then the graded version of Nakayama's lemma \cite[Prop. 4.1.22]{MR2159476} implies that $\varphi (e_1), \ldots, \varphi (e_r)$ are minimal generators of $\im\varphi$ if and only if the residue classes $\varphi (e_1) + \mathfrak{m} \im\varphi, \ldots, \varphi (e_r) + \mathfrak{m} \im\varphi$ are linearly independent in $\im\varphi /\mathfrak{m} \im\varphi$.

For the proof of \ref{min1}, assume $\varphi$ is minimal.
If $c_1,\ldots,c_r \in \KK$, then:
\begin{gather*}
\varphi \left( \sum_{i=1}^r c_i e_i \right) = 0 \Longrightarrow \sum_{i=1}^r c_i \varphi ( e_i ) = 0 \Longrightarrow \\
\Longrightarrow \sum_{i=1}^r c_i \varphi ( e_i ) \in \mathfrak{m} \im\varphi
\Longrightarrow  \sum_{i=1}^r c_i [\varphi \left( e_i \right) + \mathfrak{m} \im\varphi ] = \mathfrak{m} \im\varphi.
\end{gather*}
By our preliminary observation, $\varphi (e_1) + \mathfrak{m} \im\varphi, \ldots, \varphi (e_r) + \mathfrak{m} \im\varphi$ are linearly independent in $\im\varphi /\mathfrak{m} \im\varphi$.
This implies $\forall i\in\{1,\ldots,r\}$ $c_i = 0$, which shows
the restriction of $\varphi$ to $\langle \mathcal{E}\rangle_\KK$ is injective.

For the proof of \ref{min2}, assume $E$ is generated in degree
$d\in\ZZ^m$ and the restriction of $\varphi$ to
$\langle \mathcal{E}\rangle_\KK$ is injective.
If $c_1,\ldots,c_r \in \KK$, then:
\[
\sum_{i=1}^r c_i [\varphi \left( e_i \right) + \mathfrak{m}
\im\varphi ] = \mathfrak{m} \im\varphi \Longrightarrow
\sum_{i=1}^r c_i \varphi ( e_i ) \in \mathfrak{m} \im\varphi.
\]
Note that the coefficients $c_i$ have degree zero; however
multiplying elements of $\im\varphi$ by coefficients in 
$\mathfrak{m}$ increases the degree.
Since all elements $\varphi (e_1),\ldots,\varphi (e_r)$ have
degree $d$, the only possible way
$\sum_{i=1}^r c_i \varphi ( e_i ) \in \mathfrak{m} \im\varphi$
is if it is zero.
Now
\[
\sum_{i=1}^r c_i \varphi ( e_i ) = 0 \Longrightarrow
\varphi \left( \sum_{i=1}^r c_i e_i \right) = 0;
\]
thus $\forall i\in\{1,\ldots,r\}$ $c_i =0$ because $\varphi$ restricted to
$\langle \mathcal{E}\rangle_\KK$ is injective.
This shows $\varphi (e_1) + \mathfrak{m} \im\varphi, \ldots, \varphi (e_r) + \mathfrak{m} \im\varphi$ are linearly independent in $\im\varphi /\mathfrak{m} \im\varphi$.
Using our preliminary considerations, we deduce
$\varphi (e_1),\ldots, \varphi (e_r)$ are minimal generators of
$\im \varphi$, and therefore $\varphi$ is minimal.
\end{proof}

\begin{remark}
The differentials in a minimal free resolution are minimal maps
by construction.
\end{remark}

The ideas of this subsection are summarized in the following result.
\begin{prop}\label{summary_hbwv}
Let $\varphi \colon E\rightarrow F$ be a minimal map of free modules in $\mod_{\righttoleftarrow T} A$. Let $\tilde{\mathcal{F}} = \{\tilde{f}_1,\ldots,\tilde{f}_s\}$ be a homogeneous basis of weight vectors of $F$ and assume that $\TT^n \langle \tilde{\mathcal{F}} \rangle$ is equipped with a module term ordering. If $\tilde{\mathcal{E}} = \{\tilde{e}_1,\ldots,\tilde{e}_r\}$ is a homogeneous basis of weight vectors of $E$, then $\forall i\in \{1,\ldots,r\}$
\begin{itemize}[wide]
\item $\LT (\varphi (\tilde{e}_i))$ is a weight vector of $F$ and $\weight (\LT (\varphi (\tilde{e}_i))) = \weight (\tilde{e}_i)$;
\item if $\LT (\varphi (\tilde{e}_i)) = \hat{t} \tilde{f}_{\hat{\jmath}}$, for some term $\hat{t} \tilde{f}_{\hat{\jmath}} \in \TT^n \langle \tilde{\mathcal{F}} \rangle$, then
\[\weight (\tilde{e}_i) = \weight (\hat{t}) + \weight (\tilde{f}_{\hat{\jmath}}).\]
\end{itemize}
\end{prop}
If $\varphi$ is provided as a matrix written with respect to $\tilde{\mathcal{E}}$ and $\tilde{\mathcal{F}}$, then this proposition gives a concrete method to obtain a complete list of weights of $E$ using $\varphi$.

\begin{ex}\label{ex5}
We resume the discussion of our running example from section \ref{basics} and we look at the resolution from example \ref{ex4}. The free module $F_j = \bigwedge^j V \otimes_\CC A(-j)$ has a homogeneous basis of weight vectors
\[\tilde{\mathcal{F}}_j = \{v_{i_1} \wedge \ldots \wedge v_{i_j} \otimes 1_A \mid 1 \leq i_1 < \ldots < i_j \leq 3\},\]
where $\{v_1,v_2,v_3\}$ is the coordinate basis of $V=\CC^3$. Since the weights of $v_1,v_2,v_3$ are $\omega_1 =(1,0,0)$, $\omega_2 =(0,1,0)$, $\omega_3 =(0,0,1)$ respectively, then
\[\weight (v_{i_1} \wedge \ldots \wedge v_{i_j} \otimes 1_A ) = \omega_{i_1} + \ldots + \omega_{i_j}.\]
We will show how to obtain the same weights using proposition \ref{summary_hbwv}.

Using the homogeneous bases $\tilde{\mathcal{F}}_j$, the maps of the Koszul complex can be written in matrix form:
\[0\to F_3
\xrightarrow{
	\left(
	\begin{smallmatrix}
	x_3\\
	-x_2\\
	x_1
	\end{smallmatrix}
	\right)
}
F_2
\xrightarrow{
	\left(
	\begin{smallmatrix}
	-x_2 & -x_3 & 0 \\
	x_1 & 0 & -x_3\\
	0 & x_1 & x_2
	\end{smallmatrix}
	\right)
}
F_1
\xrightarrow{
	\left(
	\begin{smallmatrix}
	x_1 & x_2 & x_3
	\end{smallmatrix}
	\right)
}
F_0
\]
Assume $A$ is endowed with a lexicographic term ordering such that $x_1>x_2>x_3$ and the modules $F_j$ are equipped with the term over position up module term ordering. Notice that $F_0=A$ and that $G$ acts trivially on its basis element $1_A$. Therefore:
\begin{align*}
&\LT (d_1 (v_1 \otimes 1_A)) = x_1 1_A \Rightarrow \weight (v_1 \otimes 1_A) = (1,0,0),\\
&\LT (d_1 (v_2 \otimes 1_A)) = x_2 1_A \Rightarrow \weight (v_2 \otimes 1_A) = (0,1,0),\\
&\LT (d_1 (v_3 \otimes 1_A)) = x_3 1_A \Rightarrow \weight (v_3 \otimes 1_A) = (0,0,1).
\end{align*}
Knowing the weights of the elements in $\tilde{\mathcal{F}}_1$, we can proceed with the map $d_2$:
\begin{align*}
\LT &(d_2 (v_1 \wedge v_2 \otimes 1_A)) = x_1 v_2 \otimes 1_A \Rightarrow \\
&\Rightarrow \weight (v_1 \wedge v_2 \otimes 1_A) = (1,0,0) + (0,1,0) = (1,1,0), \\
\LT &(d_2 (v_1 \wedge v_3 \otimes 1_A)) = x_1 v_3 \otimes 1_A \Rightarrow \\
&\Rightarrow \weight (v_1 \wedge v_3 \otimes 1_A) = (1,0,0) + (0,0,1) = (1,0,1),\\
\LT &(d_2 (v_2 \wedge v_3 \otimes 1_A)) = x_2 v_3 \otimes 1_A \Rightarrow \\
&\Rightarrow \weight (v_2 \wedge v_3 \otimes 1_A) = (0,1,0) + (0,0,1) = (0,1,1).
\end{align*}
Finally the map $d_3$:
\begin{gather*}
\LT (d_3 (v_1 \wedge v_2 \wedge v_3 \otimes 1_A)) = x_1 v_2 \wedge v_3 \otimes 1_A \Rightarrow\\
\Rightarrow \weight (v_1 \wedge v_2 \wedge v_3 \otimes 1_A) = (1,0,0) + (0,1,1) = (1,1,1).
\end{gather*}
The weights we found can be used to identify each module $F_j$ with $\bigwedge^j V \otimes_\CC A(-j)$.
\end{ex}

\subsection{The general case}\label{general_case}
In general, we cannot expect every map $\varphi \colon E\rightarrow F$ of free modules in $\mod_{\righttoleftarrow T} A$ to be written with respect to homogeneous bases of weight vectors. Nevertheless, under reasonable assumptions, it is still possible to recover the weights of $E$ using the weights of $F$ and the map $\varphi$.
\begin{ex}\label{ex6}
To illustrate an issue that can occur, we write the matrix of the map $d_2$ from example \ref{ex4} with respect to the following homogeneous basis of $F_2$
\[\{(v_1 \wedge v_2 - v_1 \wedge v_3) \otimes 1_A, (v_1 \wedge v_2 + v_1 \wedge v_3) \otimes 1_A, (v_1 \wedge v_3 + v_2 \wedge v_3) \otimes 1_A\},\]
which is created using linear combinations of elements in the homogeneous basis $\tilde{\mathcal{F}}_2$ from example \ref{ex5}. The matrix looks like this:
\[
	\begin{pmatrix}
	-x_2 + x_3 & -x_2-x_3 & -x_3 \\
	x_1 & x_1 & -x_3\\
	-x_1 & x_1 & x_1+x_2
	\end{pmatrix}.
\]
Observe that all columns have leading term $x_1 v_3 \otimes 1_A$. If we proceed to calculate weights as indicated in proposition \ref{summary_hbwv}, we will obtain the weight $(1,0,1)$ three times, which does not fit the representation $\bigwedge^2 V$.
\end{ex}
The situation presented in the example above is somewhat artificial, however it highlights the following fact: unless we are working with homogeneous bases of weight vectors, we are not guaranteed to obtain meaningful lists of weights. The issue is especially relevant in the case of free resolutions constructed using computational techniques. In fact, many algorithms that compute free resolutions express the matrices of the differentials with respect to homogeneous bases that do not, typically, consist of weight vectors.

We adopt the notation of \cite[p. 503]{MR1878556} for the matrix of a map of free modules. Let $\varphi \colon E\rightarrow F$ be a map of free modules; let $\mathcal{E} = \{e_1,\ldots,e_r\}$ and $\mathcal{F} = \{f_1,\ldots,f_s\}$ be homogeneous bases of $E$ and $F$ respectively. If $\varphi (e_j) = \sum_{i=1}^s a_{i,j} f_i$, then $\mathcal{M}^{\mathcal{E}}_{\mathcal{F}} (\varphi) = (a_{i,j})$ is the matrix of $\varphi$ with respect to $\mathcal{E}$ and $\mathcal{F}$.
If $\mathcal{F}$ and $\mathcal{F}'$ are homogeneous bases of a free module $F$, then $\mathcal{M}^{\mathcal{F}}_{\mathcal{F}'} (\id_F)$ is the matrix of the change of basis from $\mathcal{F}$ to $\mathcal{F}'$.

Here is our main result.
\begin{thm}\label{main_thm}
Let $\varphi \colon E\rightarrow F$ be a minimal map of free modules in $\mod_{\righttoleftarrow T} A$.
Suppose that:
\begin{enumerate}[label=H\arabic*.,ref=H\arabic*]
\item\label{h1} $\mathcal{F} = \{f_1,\ldots,f_s\}$ is a homogeneous basis of $F$ and $\TT^n \langle \mathcal{F} \rangle$ is equipped with a position up module term ordering;
\item\label{h2} $F$ admits a homogeneous basis of weight vectors $\tilde{\mathcal{F}} = \{\tilde{f}_1,\ldots,\tilde{f}_s\}$ such that $\mathcal{M}^{\mathcal{F}}_{\tilde{\mathcal{F}}} (\id_F)$ is upper triangular;
\item\label{h3} $E$ admits a homogeneous basis $\mathcal{E} = \{e_1,\ldots,e_r\}$ such that
\[ \LT (\varphi (e_1)) < \ldots < \LT (\varphi (e_r))\]
in $\TT^n \langle \mathcal{F} \rangle$.
\end{enumerate}
Then:
\begin{enumerate}[label=T\arabic*.,ref=T\arabic*]
\item\label{t1} $E$ admits a homogeneous basis of weight vectors $\tilde{\mathcal{E}} = \{\tilde{e}_1,\ldots,\tilde{e}_r\}$ such that $\mathcal{M}^{\mathcal{E}}_{\tilde{\mathcal{E}}} (\id_E)$ is upper triangular;
\item\label{t2} if $\LT (\varphi (e_i)) = \hat{t} {f}_{\hat{\jmath}}$, for some $\hat{t} {f}_{\hat{\jmath}} \in \TT^n \langle \mathcal{F} \rangle$, then
\[\weight (\tilde{e}_i) = \weight (\hat{t}) + \weight (\tilde{f}_{\hat{\jmath}}).\]
\end{enumerate}
\end{thm}

In essence, the theorem says the weights of $E$ can be recovered using the map $\varphi$ and a homogeneous basis $\mathcal{E}$ of $E$ as in \ref{h3}. This basis is connected to a homogeneous basis of weight vectors of $E$ by a triangular change of basis, as long as a similar property holds for suitable bases of $F$. While it may seem that the hypotheses of the theorem are quite restrictive, they are easily met in general enough settings where we wish to apply our algorithm.

We postpone the proof of theorem \ref{main_thm} to \S\ref{proof_main_thm}.

\begin{ex}\label{ex7}
Similarly to example \ref{ex6}, we write the matrix of the map $d_2$ from example \ref{ex4} with respect to the following homogeneous basis of $F_2$
\[\{v_2 \wedge v_3 \otimes 1_A, v_1 \wedge v_2 \otimes 1_A, (v_1 \wedge v_2 + v_1 \wedge v_3) \otimes 1_A\}.\]
We get the following matrix:
\[
	\begin{pmatrix}
	0 & -x_2 & -x_2-x_3\\
	-x_3 & x_1 & x_1\\
	x_2 & 0 & x_1
	\end{pmatrix}.
\]
Now the columns have the same leading terms as the matrix for $d_2$ that was written in example \ref{ex5}, although permuted to form an increasing sequence. Since the homogeneous basis $\tilde{\mathcal{F}}_1$ used for $F_1$ consists of weight vectors and the module term ordering used is position up, the hypotheses of theorem \ref{main_thm} are satisfied. The calculation for the weights of $F_2$ proceeds as in example \ref{ex5}, thus we obtain meaningful weights even though the homogeneous basis we are using for $F_2$ does not consist of weight vectors. The change of basis from this homogeneous basis to $\tilde{\mathcal{F}}_2$ is given by the matrix
\[
	\begin{pmatrix}
	1 & -1 & 0\\
	0 & 1 & 0\\
	0 & 0 & 1
	\end{pmatrix}.
\]
which is upper triangular.
\end{ex}

For another example of theorem \ref{main_thm} in action,
we point the reader to example \ref{exa:1} which recasts
the result of the theorem in the context of a more
algorithmic approach.

\begin{remark}
The module term ordering on $\TT^n \langle \mathcal{F}\rangle$ in \ref{h1} could be taken to be a position down ordering instead of position up. The statement of theorem \ref{main_thm} would need to be modified as indicated below:
\begin{itemize}[wide]
\item in \ref{h1}, $\TT^n \langle \mathcal{F}\rangle$ is equipped with a position down module term ordering;
\item in \ref{h2}, $\mathcal{M}^{\mathcal{F}}_{\tilde{\mathcal{F}}} (\id_F)$ is lower triangular;
\item in \ref{h3}, $\LT (\varphi (e_1)) > \ldots > \LT (\varphi (e_r))$;
\item in \ref{t1}, $\mathcal{M}^{\mathcal{E}}_{\tilde{\mathcal{E}}} (\id_E)$ is lower triangular.
\end{itemize}
It might be possible to state a more general version of theorem \ref{main_thm} that holds with any module term ordering on $\TT^n \langle \mathcal{F}\rangle$.
However, the position up/down module term orderings seem to be the most commonly implemented in software, with one of them often being the default option. Therefore we decided to take a practical approach and limit ourselves to those scenarios that matter for the applications.
\end{remark}

\subsection{Proof of the main theorem}\label{proof_main_thm}

We begin this section by outlining the strategy of our 
proof of theorem \ref{main_thm}.
\begin{enumerate}[label=\arabic*),wide]
% \item We reduce to problem to the case of $E$,
%   the domain of the map $\varphi$, being generated
%   in a single degree.
%   This is always possible by writing $E$ as a direct
%   sum of free modules each generated in a single degree.
\item We illustrate how to calculate the weight of a weight
  vector in $E$, using the map $\varphi$ and the triangular
  change of basis provided by hypothesis \ref{h2}
  (proposition \ref{weight_transfer}).
  The result leverages our work on bases of weight vectors
  from section \ref{bases_of_weight_vectors}.
\item We construct a homogeneous basis of weight
  vectors of $E$, with the additional property
  that the images of the basis elements under $\varphi$
  all have different leading terms.
  The result is obtained in corollary \ref{nice_basis_wv}
  and uses a Gr\"obner basis argument developed in
  \ref{gb_basis}.
\item We show that given two bases of $E$, if the images of their
  elements under $\varphi$ all have different leading terms, then
  those leading terms are the same for both bases.  This result
  follows from an observation in commutative algebra, namely that the
  minimal generators of a monomial module are uniquely determined.
\item We compare the homogeneous basis of $E$ in hypothesis
  \ref{h3} with the homogeneous basis of weight vectors of $E$
  constructed earlier.
  Because of the way the leading terms are sorted
  and because the leading terms are the same,
  we conclude the matrix of the change of basis from one basis
  to the other must be upper triangular.
  Part \ref{t1} of the thesis follows.
\item Finally, we apply our weight calculation formula to the
  situation at hand to deduce part \ref{t2} of the thesis.
\end{enumerate}

The first step towards the proof of theorem \ref{main_thm} is to show how a triangular change of basis can be exploited to calculate the weight of a homogeneous weight vector.

\begin{prop}\label{weight_transfer}
Let $\varphi \colon E\rightarrow F$ be a map of free modules in $\mod_{\righttoleftarrow T} A$.
Suppose hypotheses \ref{h1} and \ref{h2} of theorem \ref{main_thm} hold.
If $\tilde{e}\in E$ is a homogeneous weight vector with $\LT (\varphi (\tilde{e})) = \hat{t} f_{\hat{\jmath}}$, for some $\hat{t} {f}_{\hat{\jmath}} \in \TT^n \langle \mathcal{F} \rangle$, then:
\begin{enumerate}[label=\Roman*.,wide]
\item $\hat{t} \tilde{f}_{\hat{\jmath}}\in\Supp_{\tilde{\mathcal{F}}} (\varphi (\tilde{e}))$;
\item $\weight (\tilde{e}) = \weight (\hat{t}) + \weight (\tilde{f}_{\hat{\jmath}})$.
\end{enumerate}
\end{prop}
\begin{proof}
To prove I, write $\varphi(\tilde{e})$ as a $\KK$-linear combination of terms in $\TT^n \langle \mathcal{F} \rangle$:
\begin{equation}\label{weight_transfer:eq1}
\varphi(\tilde{e}) = \sum_{t\in\TT^n} \sum_{j=1}^s c_{t,j} t f_j.
\end{equation}
If $\mathcal{M}^{\mathcal{F}}_{\tilde{\mathcal{F}}} (\id_F) = (u_{i,j})$, then
\begin{equation}\label{weight_transfer:eq2}
\varphi(\tilde{e}) = \sum_{t\in\TT^n} \sum_{j=1}^s c_{t,j} t \left(\sum_{i=1}^s u_{i,j} \tilde{f}_i \right) = \sum_{t\in\TT^n} \sum_{i=1}^s \left(\sum_{j=1}^s c_{t,j} u_{i,j} \right) t \tilde{f}_i .
\end{equation}
The coefficient of $\hat{t} \tilde{f}_{\hat{\jmath}}$ in equation (\ref{weight_transfer:eq2}) is
\begin{equation}\label{weight_transfer:eq3}
\sum_{j=1}^s c_{\hat{t},j} u_{\hat{\jmath},j}.
\end{equation}
Since $\LT (\varphi (\tilde{e})) = \hat{t} f_{\hat{\jmath}}$ and $\TT^n \langle\mathcal{F}\rangle$ is equipped with a position up module term ordering (\ref{h1}), $c_{\hat{t},j} = 0$, for all $j>\hat{\jmath}$ and $c_{\hat{t},\hat{\jmath}} \neq 0$. Moreover, $u_{\hat{\jmath},\hat{\jmath}}\neq 0$ and $u_{\hat{\jmath},j} = 0$ for all $j<\hat{\jmath}$ because the matrix $\mathcal{M}^{\mathcal{F}}_{\tilde{\mathcal{F}}} (\id_F) = (u_{i,j})$ is upper triangular (\ref{h2}). Therefore the coefficient of the term $\hat{t} \tilde{f}_{\hat{\jmath}}$ is $c_{\hat{t},\hat{\jmath}} u_{\hat{\jmath},\hat{\jmath}}$. Since this coefficient is non zero, we conclude that $\hat{t} \tilde{f}_{\hat{\jmath}}\in\Supp_{\tilde{\mathcal{F}}} (\varphi (\tilde{e}))$.

Now II is an immediate consequence of I and corollary \ref{weight_in_homogeneous_basis_via_term}.
\end{proof}

\begin{coro}\label{same_LT_same_weight}
Let $\varphi \colon E\rightarrow F$ be a map of free modules in $\mod_{\righttoleftarrow T} A$.
Suppose hypotheses \ref{h1} and \ref{h2} of theorem \ref{main_thm} hold.
If $\tilde{e}_1,\tilde{e}_2\in E$ are homogeneous weight vectors and $\LT (\varphi (\tilde{e}_1)) = \LT (\varphi (\tilde{e}_2))$, then $\weight (\tilde{e}_1) = \weight (\tilde{e}_2)$.
\end{coro}
\begin{proof}
Suppose $\LT (\varphi (\tilde{e}_1)) = \LT (\varphi (\tilde{e}_2)) = \hat{t} f_{\hat{\jmath}}$. Then, by the proposition \ref{weight_transfer},
\[\weight (\tilde{e}_1) = \weight (\hat{t}) + \weight (\tilde{f}_{\hat{\jmath}}) = \weight (\tilde{e}_2).\]
\end{proof}

Let $M$ be a graded submodule of a graded free $A$-module $F$ endowed with a module term ordering. Recall that the reduced Gr\"obner basis of $M$ (defined in \cite[\S 2.4.C]{MR1790326}) consists of homogeneous elements \cite[Prop. 4.5.1]{MR2159476}. Also, given a homogeneous Gr\"obner basis $\mathcal{G}$ of $M$, set $\mathcal{G}_{\leq d} = \{g\in \mathcal{G}\mid \deg (g) \leq d\}$, the truncation of $\mathcal{G}$ at degree $d$ (see \cite[\S4.5.B]{MR2159476} for truncated Gr\"obner bases).

\begin{prop}\label{gb_basis}
Let $\varphi \colon E\rightarrow F$ be a minimal map of free modules in $\mod_{\righttoleftarrow T} A$ with $E$ generated in a single degree $d\in \ZZ^m$. Let $\mathcal{F}$ be a homogeneous basis of $F$ and assume $\TT^n \langle\mathcal{F}\rangle$ is equipped with a module term ordering.
Let $Z$ be a $\KK$-vector subspace of $E_d$ and let $M$ be the $A$-submodule of $F$ generated by $\varphi (Z)$. If $\mathcal{G}$ is the reduced Gr\"obner basis of $M$, then $\varphi^{-1} (\mathcal{G}_{\leq d})$ is a $\KK$-basis of $Z$.
\end{prop}
\begin{proof}
Notice that $M$ is generated in degree $d$ and $M_d = \varphi (Z)$. Thus all elements of $\mathcal{G}_{\leq d}$ have degree equal to $d$. Since $\mathcal{G}$ generates $M$ as an $A$-module, $\mathcal{G}_{\leq d}$ generates $M_d$ as a $\KK$-vector space. The elements of $\mathcal{G}$ all have different leading terms because $\mathcal{G}$ is reduced; in particular, the elements of $\mathcal{G}$ are $\KK$-linearly independent. We conclude that $\mathcal{G}_{\leq d}$ is a $\KK$-basis of $M_d$.

To get the thesis, observe that the restriction of $\varphi$ to $Z$ is injective by part \ref{min1} of proposition \ref{minimality} and the assumption that $Z\subseteq E_d$. Therefore the preimage of a basis of $M_d$ is a basis of $\varphi^{-1} (M_d) = Z$. In particular, $\varphi^{-1} (\mathcal{G}_{\leq d})$ is a basis of $Z$.
\end{proof}

\begin{coro}\label{nice_basis_wv}
Let $\varphi \colon E\rightarrow F$ be a minimal map of free modules in $\mod_{\righttoleftarrow T} A$. If hypotheses \ref{h1} and \ref{h2} of theorem \ref{main_thm} hold, then $E$ admits a homogeneous basis of weight vectors $\tilde{\mathcal{E}} = \{\tilde{e}_1,\ldots,\tilde{e}_r\}$ such that
\[\LT(\varphi(\tilde{e}_1)) < \ldots < \LT(\varphi(\tilde{e}_r)).\]
\end{coro}
\begin{proof}
Let us begin by considering the special case where $E$ is generated in a single degree $d\in \ZZ^m$.
Since $E_d$ is a finite dimensional representation of $T$, we have a decomposition into weight spaces $E_d = (E_d)_{\chi_1} \oplus \ldots \oplus (E_d)_{\chi_h}$, for some characters $\chi_1,\ldots,\chi_h$ of $T$. By proposition \ref{gb_basis}, each weight space $(E_d)_{\chi_i}$ admits a basis $\tilde{\mathcal{E}}_i$ such that the images of elements in $\tilde{\mathcal{E}}_i$ under $\varphi$ have different leading terms.

Now consider $\tilde{\mathcal{E}} := \tilde{\mathcal{E}}_1\cup \ldots \cup \tilde{\mathcal{E}}_h$. The set $\tilde{\mathcal{E}}$ is a $\KK$-basis of weight vectors of $E_d$. Since $E$ is generated in degree $d$, $\tilde{\mathcal{E}}$ is also a homogeneous basis of weight vectors $E$. Moreover if two elements of $\tilde{\mathcal{E}}$ have different weights, their images under $\varphi$ have different leading terms by corollary \ref{same_LT_same_weight}. Therefore we can index elements of $\tilde{\mathcal{E}}$ so that $\tilde{\mathcal{E}} = \{\tilde{e}_1,\ldots,\tilde{e}_r\}$ and $\LT(\varphi(\tilde{e}_1)) < \ldots < \LT(\varphi(\tilde{e}_r))$.

For the general case, simply observe that if two homogeneous elements have different degrees, their leading terms must be different. Therefore the conclusion holds for any $E$.
\end{proof}

We are finally ready to prove the main result of this work.

\begin{proof}[Proof of theorem \ref{main_thm}]\mbox{}

\begin{mydescription}[wide,style=nextline]
\item[Step 1: reduction to a single degree.] Without loss of generality we may assume that $E$ is generated in a single degree $d\in \ZZ^m$. If that is not the case, we may write $E = E_1 \oplus \ldots \oplus E_l$, where each $E_i$ is a free module in $\mod_{\righttoleftarrow T} A$ which is generated in a single degree $d_i\in \ZZ^m$. Then, by properties of the direct sum, we can define maps $\varphi_i \colon E_i\rightarrow F$ that satisfy the hypotheses of the theorem, thus reducing to the case where the domain is generated in a single degree.
\item[Step 2: a special basis of $E$.] By corollary \ref{nice_basis_wv}, $E$ admits a homogeneous basis of weight vectors $\tilde{\mathcal{E}} = \{\tilde{e}_1,\ldots,\tilde{e}_r\}$ such that
\begin{equation}\label{main_thm_proof:eq2}
\LT(\varphi(\tilde{e}_1)) < \ldots < \LT(\varphi(\tilde{e}_r)).
\end{equation}
\item[Step 3: compare leading terms.] Since $\varphi$ is minimal and $\mathcal{E}$ is a homogeneous basis of $E$, $\varphi(e_1), \ldots, \varphi(e_r)$ are minimal generators of $\im\varphi$. Because $E$ is generated in degree $d$, this implies $\dim_\KK (\im\varphi)_d = r$. Now let $\LT (\im\varphi)$ be the leading term module of $\im\varphi$ . Recall that $\LT (\im\varphi)$ is the monomial submodule of $F$ generated by the set
\[\{\LT(m) \in F \mid m\in \im\varphi\}.\]
It is known that $\im\varphi$ and $\LT (\im\varphi)$ have the same Hilbert function \cite[Prop. 5.8.9.f]{MR2159476}; in particular, $\dim_\KK (\LT(\im\varphi))_d = \dim_\KK (\im\varphi)_d = r$. By the string of inequalities in \ref{h3}, the terms $\LT(\varphi(e_1)), \ldots, \LT(\varphi(e_r))$ are all different and thus $\KK$-linearly independent. We conclude that they form a basis of $(\LT(\im\varphi))_d$.

Since $\tilde{\mathcal{E}}$ is another homogeneous basis of $E$, the terms $\LT(\varphi(\tilde{e}_1)), \ldots, \LT(\varphi(\tilde{e}_r))$ form another basis of $(\LT(\im\varphi))_d$, by the same argument. By \cite[Prop. 1.3.11]{MR1790326}, the minimal monomial generators of a monomial module are uniquely determined; this gives an equality of sets:
\[\{\LT(\varphi(e_1)), \ldots, \LT(\varphi(e_r))\} = \{\LT(\varphi(\tilde{e}_1)), \ldots, \LT(\varphi(\tilde{e}_r))\}.\]
Finally, the strings of inequalities in \ref{h3} and (\ref{main_thm_proof:eq2}) imply that
\begin{equation}\label{main_thm_proof:eq3}
\LT (\varphi (e_i)) = \LT (\varphi (\tilde{e}_i)),\ \forall i\in \{1,\ldots,r\}.
\end{equation}
\item[Step 4: the triangular change of basis.] Combining the string of inequalities (\ref{main_thm_proof:eq2})  and the equalities (\ref{main_thm_proof:eq3}), we get
\[\LT(\varphi(\tilde{e}_1)) < \ldots < \LT(\varphi(\tilde{e}_j)) = \LT(\varphi({e}_j)),\]
for all $j\in\{1,\ldots,r\}$. Thus, for all $j\in\{1,\ldots,r\}$, $\exists u_{1,j},\ldots,u_{j,j}\in \KK$ such that
\begin{equation}\label{main_thm_proof:eq4}
\varphi({e}_j) = \sum_{i=1}^j u_{i,j} \varphi(\tilde{e}_i) = \varphi\left(\sum_{i=1}^j u_{i,j} \tilde{e}_i\right),
\end{equation}
and $u_{j,j}\neq 0$. The matrix $(u_{i,j})$ is then upper triangular and invertible. The map $\varphi$ being minimal, its restriction to $E_d$ is injective by part \ref{min1} of proposition \ref{minimality}. Therefore equation (\ref{main_thm_proof:eq4}) implies
\[{e}_j = \sum_{i=1}^j u_{i,j} \tilde{e}_i.\] 
In other words, $\mathcal{M}^{\mathcal{E}}_{\tilde{\mathcal{E}}} (\id_E) = (u_{i,j})$ is upper triangular, which proves \ref{t1} in the thesis.
\item[Step 5: calculate the weights of $E$.] For each $i\in\{1,\ldots,r\}$, $\LT (\varphi (e_i)) = \LT (\varphi (\tilde{e}_i)) = \hat{t} {f}_{\hat{\jmath}}$, for some $\hat{t} {f}_{\hat{\jmath}} \in \TT^n \langle \mathcal{F} \rangle$; then
\[\weight (\tilde{e}_i) = \weight (\hat{t}) + \weight (\tilde{f}_{\hat{\jmath}}),\]
by proposition \ref{weight_transfer}. This proves \ref{t2} in the thesis.
\end{mydescription}
\end{proof}

\section{Algorithms and applications}
In this section, we present some applications of theorem \ref{main_thm}. We also design algorithms, in pseudocode, that can be implemented in a software system to carry out the necessary computations.
The software system must be able to compute over polynomial rings with positive $\ZZ^m$-gradings and determine:
\begin{itemize}
\item reduced Gr\"obner bases (truncated at a given degree);
\item the change of basis between the reduced Gr\"obner basis of a  module and some user-specified set of generators;
\item a basis for the graded components of a module (via Macaulay's theorem \cite[Cor. 2.4.11]{MR1790326}).
\end{itemize}

\subsection{Weight propagation along a map}\label{sec_algo_map}
Our first algorithm can be used to recover the weights of the domain a map from those of the codomain. We lay out some assumptions.
\begin{enumerate}[label=(\alph*),wide]
\item $\varphi\colon E\rightarrow F$ is a minimal map of free modules in $\mod_{\righttoleftarrow T} A$.
\item The map $\varphi$ is provided in matrix form $\mathcal{M}^{\mathcal{E}}_{\mathcal{F}} (\varphi)$, where $\mathcal{E}$ is a homogeneous basis of $E$ and $\mathcal{F}$ is a homogeneous basis of $F$.
\item $\TT^n \langle \mathcal{F} \rangle$ is equipped with a position up (resp. down) module term ordering.
\item $F$ admits a homogeneous basis of weight vectors $\tilde{\mathcal{F}} = \{\tilde{f}_1,\ldots,\tilde{f}_s\}$ such that $\mathcal{M}^{\mathcal{F}}_{\tilde{\mathcal{F}}} (\id_F)$ is upper (resp. lower) triangular.
\item $W = \{w_1,\ldots,w_s\}$ is an ordered list with $w_i = \weight (\tilde{f}_i)$, $\forall i\in\{1,\ldots,s\}$.
\end{enumerate}

First we deal with the case where $E$ is generated in a single degree $d\in\ZZ^m$.

\begin{algo}[weight propagation along a map with domain in a single degree]\label{algo_single_degree}
\algrenewcommand\algorithmicrequire{\textbf{Input:}}
\algrenewcommand\algorithmicensure{\textbf{Output:}}
\begin{algorithmic}[1]
\Require{\hfill
\begin{itemize}
\item $\mathcal{M}^{\mathcal{E}}_{\mathcal{F}} (\varphi)$, a matrix as in assumption (b) above
\item $W$, a list of weights as in assumption (e) above
\end{itemize}
}
\Ensure{\hfill
  \begin{itemize}
  \item $C = \mathcal{M}^{\mathcal{E}'}_{\mathcal{E}} (\id_E)$,
    a change of basis in $E$ such that the leading terms of the
    columns of $\mathcal{M}^{\mathcal{E}'}_{\mathcal{F}} (\varphi)$
    are all different
  \item $V = \{v_1,\ldots,v_r\}$, a list of weights such that
    $v_i = \weight (\tilde{e}_i)$ for a homogeneous basis of weight
    vectors $\{\tilde{e}_1,\ldots,\tilde{e}_r\}$ of $E$
  \end{itemize}
~\hfill
}
\Function{PropagateSingleDegree}{$\mathcal{M}^{\mathcal{E}}_{\mathcal{F}} (\varphi),W$}
   \State set $d:=$ the degree of the first column of $\mathcal{M}^{\mathcal{E}}_{\mathcal{F}} (\varphi)$
   \State compute $\mathcal{G}_{\leq d}$, the $d$-truncated reduced Gr\"obner basis of $\im\varphi$, using $\mathcal{M}^{\mathcal{E}}_{\mathcal{F}} (\varphi)$
   \If{the module term ordering on $\TT^n \langle \mathcal{F}\rangle$ is position up}
   	\State form $G$, a matrix whose columns are the component vectors (in the homogeneous basis $\mathcal{F}$) of the elements of $\mathcal{G}_{\leq d}$ arranged in increasing order of their leading terms
   \ElsIf{the module term ordering on $\TT^n \langle \mathcal{F}\rangle$ is position down}
   	\State form $G$, a matrix whose columns are the component vectors (in the homogeneous basis $\mathcal{F}$) of the elements of $\mathcal{G}_{\leq d}$ arranged in decreasing order of their leading terms
   \EndIf
   \State compute $C$, the matrix of the change of basis in $E$ such that $G = \mathcal{M}^{\mathcal{E}}_{\mathcal{F}} (\varphi) C$
   \State set $r:=$ number of columns of $\mathcal{M}^{\mathcal{E}}_{\mathcal{F}} (\varphi)$
   \For{$i\in\{1,\ldots,r\}$}
   	\State set $f:=$ the $i$-th column of $G$
	\State find $t\in\TT^n$ and $\hat{\jmath} \in \{1,\ldots,s\}$, such that $\LT (f) = \hat{t} f_{\hat{\jmath}}$
	\State set $v_i := \weight(\hat{t}) + w_{\hat{\jmath}}$ 
   \EndFor
   \State form the ordered list $V := \{v_1,\ldots,v_r\}$
   \State \textbf{return} $(C,V)$
\EndFunction
\end{algorithmic}
\end{algo}

\begin{ex}
  \label{exa:1}
  Let $A=\CC [x_1,x_2]$ equipped with the standard grading and the degree reverse lexicographic term ordering with $x_1>x_2$.
%All free $A$-modules we consider are equipped with the term over position up module term ordering.
We identify $A$ with $\Sym (\CC^2)$ and consider the action of the maximal torus of diagonal matrices $T\subseteq\GL_2 (\CC)$ on $A$.
The variables $x_1,x_2\in A$ are weight vectors with weights $(1,0)$
and $(0,1)$ respectively (cf. example \ref{ex1}).

  Let $F= A$ with homogeneous basis $\mathcal{F} = \{f_1\}$, where $f_1 =1_A$, and assume $F$ is equipped with the term over position 
up module term ordering (so \ref{h1} of theorem \ref{main_thm} is
met).
Notice $\tilde{f}_1=f_1=1_A$ is a weight vector with weight $(0,0)$, and $\tilde{\mathcal{F}} = \{\tilde{f}_1\} = \mathcal{F}$ is a homogeneous basis of weight vectors of $F$. Clearly $\mathcal{M}^{\mathcal{F}}_{\tilde{\mathcal{F}}} (\id_F) = (1)$ is upper triangular (which satisfies \ref{h2} of theorem \ref{main_thm}).
Let $E=A(-1)^2$ with homogeneous basis $\mathcal{E}=\{e_1,e_2\}$, where $e_1 = (1_A,0)$ and $e_2 = (0,1_A)$ (the coordinate basis). We introduce a map $\varphi \colon E\to F$ in $\mod_{\righttoleftarrow T} A$ via its matrix form
  \[\mathcal{M}^{\mathcal{E}}_{\mathcal{F}} (\varphi) = 
  \begin{pmatrix}
    x_1 & x_2
  \end{pmatrix};
  \]
note $\varphi$ is a minimal presentation for the coordinate ring of the origin in the affine space $\mathbb{A}^2$, therefore it is minimal.

We apply algorithm \ref{algo_single_degree} to the pair $(\mathcal{M}^{\mathcal{E}}_{\mathcal{F}} (\varphi),W)$, where $W=\{(0,0)\}$. All columns of $\mathcal{M}^{\mathcal{E}}_{\mathcal{F}} (\varphi)$ have degree $d=1$. We compute the reduced Gr\"obner basis of $\im\varphi$ truncated in degree 1 and, from its elements, form the matrix
  \[G=
  \begin{pmatrix}
    x_2 & x_1
  \end{pmatrix},
  \]
as indicated in lines 4 through 8 of the algorithm.
Note how the leading terms of the elements in the columns of $G$ are all
different and sorted in increasing order from left to right
(this ensures that hypothesis \ref{h3} in theorem \ref{main_thm}
is satisfied).
The matrix
  \[C=
  \begin{pmatrix}
    0 & 1\\
    1 & 0
  \end{pmatrix}\]
gives the change of basis $G = \mathcal{M}^{\mathcal{E}}_{\mathcal{F}} (\varphi) C$.

Note $G$ has $r=2$ columns. For the first column, $f = (x_2)$ which has leading term $x_2 f_1$. Hence we set
  \[v_1 = \weight (x_2) + w_1 = (0,1) + (0,0) = (0,1).\]
Similarly, for the second column, $f = (x_1)$ which has leading term $x_1 f_1$. Hence we set
  \[v_2 = \weight (x_1) + w_1 = (1,0) + (0,0) = (1,0).\]
Then the algorithm returns the pair $(C,V)$ with $C$ as above and $V=\{(0,1),(1,0)\}$.
Notice how the weight computations carried out use the formula
from \ref{t2} in theorem \ref{main_thm}.
By \ref{t1} we conclude that $E$ admits a homogeneous basis of weight vectors $\tilde{\mathcal{E}}=\{\tilde{e}_1,\tilde{e}_2\}$, with $\weight (\tilde{e}_1) = (0,1)$ and $\weight (\tilde{e}_2) = (1,0)$. Therefore $E\cong \CC^2 \otimes_\CC A(-1)$.
\end{ex}

\begin{ex}
  \label{exa:2}
  Let $A=\CC[y_1,y_2]$. Despite the different names for the variables, we assume the same setup as in the previous example. Let $F=A$, $E=A(-2)^3$, and $\mathcal{F}$, $\mathcal{E}$ their respective coordinate bases.
Assume $F$ has the term over position up module term ordering.
The map $\varphi \colon E\to F$ in $\mod_{\righttoleftarrow T} A$ given by the matrix
  \[\mathcal{M}^{\mathcal{E}}_{\mathcal{F}} (\varphi) = 
  \begin{pmatrix}
    y_1^2 & y_1 y_2 & y_2^2
  \end{pmatrix}\]
is a minimal presentation for a quotient of $A$ which gives a non reduced scheme structure for the origin of $\mathbb{A}^2$; in particular, $\varphi$ is a minimal map.

  We apply algorithm \ref{algo_single_degree} to the pair $(\mathcal{M}^{\mathcal{E}}_{\mathcal{F}} (\varphi),W)$, where $W=\{(0,0)\}$. The columns of $\mathcal{M}^{\mathcal{E}}_{\mathcal{F}} (\varphi)$ have degree 2 and a Gr\"obner basis computation yields
  \[G=
  \begin{pmatrix}
    y_2^2 & y_1 y_2 & y_1^2
  \end{pmatrix}.\]
The matrix of the change of basis calculated by the algorithm is
  \[C=
  \begin{pmatrix}
    0 & 0 & 1\\
    0 & 1 & 0\\
    1 & 0 & 0
  \end{pmatrix}.\]
Lines 13 through 17 result in the following assignments:
\begin{align*}
  &v_1 = \weight (y_2^2) + w_1 = (0,2) + (0,0) = (0,2),\\
  &v_2 = \weight (y_1 y_2) + w_1 = (1,1) + (0,0) = (1,1),\\
  &v_3 = \weight (y_1^2) + w_1 = (2,0) + (0,0) = (2,0).
\end{align*}
Therefore the algorithm returns $(C,V)$ with $V=\{(0,2),(1,1),(2,0)\}$. We deduce $E\cong \Sym^2 (\CC^2) \otimes_\CC A(-2)$.
\end{ex}

\begin{prop}\label{algo_single_degree_prop}
Under the assumptions of this section, applying algorithm \ref{algo_single_degree} yields a pair $(C, V=\{v_1,\ldots,v_r\})$ that satisfies the following properties:
\begin{enumerate}[label=\Roman*.,ref=\Roman*,wide]
\item\label{a11} $E$ admits a homogeneous basis $\mathcal{E}' = \{e'_1,\ldots,e'_r\}$ such that $C = \mathcal{M}^{\mathcal{E}'}_{\mathcal{E}} (\id_E)$ and the terms $\LT (\varphi (e'_1)),\ldots,\LT (\varphi (e'_r))$ are all different;
\item\label{a12} $E$ admits a homogeneous basis of weight vectors $\tilde{\mathcal{E}} = \{\tilde{e}_1,\ldots,\tilde{e}_r\}$ such that $\mathcal{M}^{\mathcal{E}'}_{\tilde{\mathcal{E}}} (\id_E)$ is upper (resp. lower) triangular and $\weight (\tilde{e}_i) = v_i$, for all $i \in \{1,\ldots,r\}$.
\end{enumerate}
\end{prop}
\begin{proof}
By its construction on line 9 of algorithm \ref{algo_single_degree},
$C$ is the matrix of a change of basis in $E$.
Let $\mathcal{E}'$ be the homogeneous basis of $E$ defined by $C$, so that the equality $C = \mathcal{M}^{\mathcal{E}'}_{\mathcal{E}} (\id_E)$ in \ref{a11} follows by construction.
Since $C = \mathcal{M}^{\mathcal{E}'}_{\mathcal{E}} (\id_E)$, the equality on line 9 implies $G = \mathcal{M}^{\mathcal{E}'}_{\mathcal{F}} (\varphi)$. This implies that the $i$-th column of $G$ is the component vector of $\varphi (e'_i)$ in the basis $\mathcal{F}$. Recall that $G$ is the matrix of a reduced Gr\"obner basis, therefore the leading terms of its columns are all different, which concludes the proof of \ref{a11}.

Now apply theorem \ref{main_thm}, with $\mathcal{E}'$ playing the role of the basis $\mathcal{E}$ in hypothesis \ref{h3} of the theorem. Then, by \ref{t1}, $E$ admits a homogeneous basis of weight vectors $\tilde{\mathcal{E}} = \{\tilde{e}_1,\ldots,\tilde{e}_r\}$ such that $\mathcal{M}^{\mathcal{E}'}_{\tilde{\mathcal{E}}} (\id_E)$ is upper (resp. lower) triangular.
If $\LT (\varphi (\tilde{e}_i)) = \hat{t} f_{\hat{\jmath}}$, then
\[\weight (\tilde{e}_i) = \weight (\hat{t}) + \weight (\tilde{f}_{\hat{\jmath}}) = \weight (\hat{t}) + w_{\hat{\jmath}} = v_i.\]
The first equality follows from \ref{t2} in theorem \ref{main_thm}; the second equality comes from the definition of $W$ and the last one descends from the construction of $v_i$ on line 14. This proves \ref{a12}.
\end{proof}

Let $(M_1 | \ldots | M_l)$ denote the block matrix obtained by juxtaposing matrices $M_1,\ldots,M_l$ all with the same number of rows. Also, given matrices $C_1,\ldots,C_l$, denote $C_1 \oplus \ldots \oplus C_l$ the block diagonal matrix
\[
\begin{pmatrix}
C_1 & 0 & \ldots & 0\\
0 & \ddots & \ddots & \vdots \\
\vdots & \ddots & \ddots & 0\\
0 & \ldots & 0 & C_l
\end{pmatrix}.
\]
We will employ these notations to generalize algorithm \ref{algo_single_degree} to the case where the domain $E$ of $\varphi$ is generated in multiple degrees.

\begin{algo}[weight propagation along a map]\label{algo_multiple_degrees}
\algrenewcommand\algorithmicrequire{\textbf{Input:}}
\algrenewcommand\algorithmicensure{\textbf{Output:}}
\begin{algorithmic}[1]
\Require{\hfill
\begin{itemize}
\item $\mathcal{M}^{\mathcal{E}}_{\mathcal{F}} (\varphi)$, a matrix as in assumption  (b) above
\item $W$, a list of weights as in assumption (e) above
\end{itemize}
}
\Ensure{\hfill
  \begin{itemize}
  \item $C = \mathcal{M}^{\mathcal{E}'}_{\mathcal{E}} (\id_E)$,
    a change of basis in $E$ such that the leading terms of the
    columns of $\mathcal{M}^{\mathcal{E}'}_{\mathcal{F}} (\varphi)$
    are all different
  \item $V = \{v_1,\ldots,v_r\}$, a list of weights such that
    $v_i = \weight (\tilde{e}_i)$ for a homogeneous basis of weight
    vectors $\{\tilde{e}_1,\ldots,\tilde{e}_r\}$ of $E$
  \end{itemize}
~\hfill
}
\Function{Propagate}{$\mathcal{M}^{\mathcal{E}}_{\mathcal{F}} (\varphi),W$}
	\State form $P$, a permutation matrix such that $\mathcal{M}^{\mathcal{E}}_{\mathcal{F}} (\varphi) P = (M_1 | \ldots | M_l)$, where, $\forall i\in\{1,\ldots,l\}$, the columns of $M_i$ have the same degree $d_i\in\ZZ^m$, and the degrees $d_1,\ldots,d_l$ are all different
	\For{$i\in\{1,\ldots,l\}$}
		\State set $(C_i,V_i) := \text{\textsc{PropagateSingleDegree}} (M_i, W)$
	\EndFor
	\State set $C := P (C_1\oplus\ldots\oplus C_l)$
	\State set $V := V_1\cup \ldots\cup V_l$
	\State \textbf{return} $(C,V)$
\EndFunction
\end{algorithmic}
\end{algo}

The symbol $\cup$ on line 7 denotes ``ordered union'' of ordered lists; in other words the list $V_2$ is appended to $V_1$, then $V_3$ is appended to $V_1\cup V_2$, etc.

It is enough to trace through the algorithm to realize that the conclusion of proposition \ref{algo_single_degree_prop} holds for algorithm \ref{algo_multiple_degrees} as well.

\begin{ex}
  \label{exa:3}
  Let $A=\CC [x_1,x_2,y_1,y_2]$; we introduce a positive $\ZZ^2$-grading on $A$ by setting the degree of $x_i$ equal to $(1,0)$ and the degree of $y_i$ equal to $(0,1)$, for $i\in\{1,2\}$. We equip $A$ with the degree reverse lexicographic term ordering with $x_1 > x_2 > y_1 > y_2$ and all free $A$-modules have the term over position up module term ordering. We identify $A$ with $\Sym (\CC^2 \oplus \CC^2)$. The group $\GL_2 (\CC) \times \GL_2 (\CC)$ acts naturally on $A$; the maximal torus $T$ of diagonal matrices in each copy of $\GL_2 (\CC)$ acts on $A$, making the variables $x_1,x_2,y_1,y_2$ into weight vectors of weights $(1,0,0,0),(0,1,0,0),(0,0,1,0),(0,0,0,1)$.

  Let $F=A$ and $E=A(-1,0)^2 \oplus A(0,-2)^3$, with $\mathcal{F},\mathcal{E}$ their respective coordinate bases. Since $F$ has rank 1, the unique element $f_1=1_A\in \mathcal{F}$ is a weight vector with weight $(0,0,0,0)$; hence we can take a homogeneous basis of weight vectors $\tilde{\mathcal{F}} = \mathcal{F}$ and $\mathcal{M}^{\mathcal{F}}_{\tilde{\mathcal{F}}} (\id_F)$ is upper triangular. The map $\varphi \colon E\to F$ represented by the matrix
  \[\mathcal{M}^{\mathcal{E}}_{\mathcal{F}} (\varphi) = 
  \begin{pmatrix}
    x_1 & x_2 & y_1^2 & y_1 y_2 & y_2^2
  \end{pmatrix}
  \]
is in $\mod_{\righttoleftarrow T} A$. It is a minimal presentation of a quotient of $A$ which gives a non reduced scheme structure for the origin of $\mathbb{A}^4$; in particular, it is minimal.

  We apply algorithm \ref{algo_multiple_degrees} to the pair $(\mathcal{M}^{\mathcal{E}}_{\mathcal{F}} (\varphi),W)$, where $W=\{(0,0,0,0)\}$. Line 2 of the algorithm sorts all columns of $\mathcal{M}^{\mathcal{E}}_{\mathcal{F}} (\varphi)$ of the same degree into a contiguous block.
In our case, $\mathcal{M}^{\mathcal{E}}_{\mathcal{F}} (\varphi) = (M_1 | M_2)$ with
\[M_1 =
\begin{pmatrix}
  x_1 & x_2
\end{pmatrix},
\qquad
M_2 = 
\begin{pmatrix}
  y_1^2 & y_1 y_2 & y_2^2
\end{pmatrix},\]
and the columns of $M_1$ have degree $d_1 = (1,0)$ whereas the columns of $M_2$ have degree $d_2 = (0,2)$.
Since the columns are already sorted, $P$ is simply the identity matrix. Next the algorithm calls
algorithm \ref{algo_single_degree} on the pair $(M_1,W)$, which returns
\[C_1=
\begin{pmatrix}
  0 & 1\\
  1 & 0
\end{pmatrix},
\qquad
V_1 = \{(0,1,0,0),(1,0,0,0)\},\]
and on the pair $(M_2,W)$, which returns
\[C_2=
\begin{pmatrix}
  0 & 0 & 1\\
  0 & 1 & 0\\
  1 & 0 & 0
\end{pmatrix},
\qquad
V_2 = \{(0,0,0,2),(0,0,1,1),(0,0,2,0)\}.\]
The computations for the pairs $(M_i,W)$ are essentially the same
as those carried out in examples \ref{exa:1} and \ref{exa:2},
the only difference being that weights are in $\ZZ^4$.

The outputs are then combined as indicated on lines 6 and 7:
\begin{gather*}
  C= C_1 \oplus C_2 =
  \begin{pmatrix}
    0 & 1 & 0 & 0 & 0\\
    1 & 0 & 0 & 0 & 0\\
    0 & 0 & 0 & 0 & 1\\
    0 & 0 & 0 & 1 & 0\\
    0 & 0 & 1 & 0 & 0
  \end{pmatrix}\\
  V = V_1 \cup V_2 = \{(0,1,0,0),(1,0,0,0),
  (0,0,0,2),(0,0,1,1),(0,0,2,0)\}.
\end{gather*}
Finally the algorithm returns the pair $(C,V)$.
By virtue of proposition, \ref{algo_single_degree_prop} we deduce:
\begin{enumerate}[label=\Roman*.,ref=\Roman*,wide]
\item $E$ admits a homogeneous basis $\mathcal{E}' = \{e'_1,\ldots,e'_5\}$ such that $C = \mathcal{M}^{\mathcal{E}'}_{\mathcal{E}} (\id_E)$ and the leading terms of the columns of $\mathcal{M}^{\mathcal{E}}_{\mathcal{F}} (\varphi) C$ are all different;
\item $E$ admits a homogeneous basis of weight vectors $\tilde{\mathcal{E}} = \{\tilde{e}_1,\ldots,\tilde{e}_5\}$ such that $\mathcal{M}^{\mathcal{E}'}_{\tilde{\mathcal{E}}} (\id_E)$ is upper triangular and $\weight (\tilde{e}_i) = v_i$, for all $i \in \{1,\ldots,5\}$.
\end{enumerate}
From the second item, we recover
\begin{align*}
  E\cong &\left(\CC^2 \otimes \CC\right) \otimes_\CC A(-1,0) \oplus\\
  \oplus &\left(\CC \otimes \Sym^2 \CC^2\right) \otimes_\CC A(0,-2)
\end{align*}
in terms of $\GL_2 (\CC) \times \GL_2 (\CC)$-representations.
\end{ex}

The following examples showcase more interesting applications
of algorithm \ref{algo_multiple_degrees} and illustrate how we can
propagate weights successfully along syzygies. The basic ideas
that appear in these examples  will be formalized and organized
in an algorithm in section \ref{sec_algo_res}.

\begin{ex}
  \label{exa:4}
  Consider the map $\varphi$ from example \ref{exa:3}.
  The minimal syzygies of this map, as calculated by our computer algebra system,
  are given by the matrix
  \[\bgroup\begin{pmatrix}{-{x}_{2}}&
{-{y}_{1}^{2}}&
0&
{-{y}_{1} {y}_{2}}&
0&
0&
{-{y}_{2}^{2}}&
0&
0\\
{x}_{1}&
0&
{-{y}_{1}^{2}}&
0&
{-{y}_{1} {y}_{2}}&
0&
0&
{-{y}_{2}^{2}}&
0\\
0&
{x}_{1}&
{x}_{2}&
0&
0&
{-{y}_{2}}&
0&
0&
0\\
0&
0&
0&
{x}_{1}&
{x}_{2}&
{y}_{1}&
0&
0&
{-{y}_{2}}\\
0&
0&
0&
0&
0&
0&
{x}_{1}&
{x}_{2}&
{y}_{1}\\
\end{pmatrix}\egroup;\]
the columns have degrees, from left to right,
$(2,0)$, $(1,2)$, $(1,2)$, $(1,2)$, $(1,2)$, $(0,3)$, $(1,2)$, $(1,2)$, $(0,3)$.
Note the map represented by this matrix is a minimal map in
$\mod_{\righttoleftarrow T} A$ by proposition \ref{equivariant_resolution}.

The codomain of the map represented by this matrix is the same $E$ we
had in example \ref{exa:3}. Multiplying on the left by the inverse
of the matrix $C$ from the end of example \ref{exa:3}, we get
\begin{equation}
\label{example_matrix}
\bgroup\begin{pmatrix}{x}_{1}&
0&
{-{y}_{1}^{2}}&
0&
{-{y}_{1} {y}_{2}}&
0&
0&
{-{y}_{2}^{2}}&
0\\
{-{x}_{2}}&
{-{y}_{1}^{2}}&
0&
{-{y}_{1} {y}_{2}}&
0&
0&
{-{y}_{2}^{2}}&
0&
0\\
0&
0&
0&
0&
0&
0&
{x}_{1}&
{x}_{2}&
{y}_{1}\\
0&
0&
0&
{x}_{1}&
{x}_{2}&
{y}_{1}&
0&
0&
{-{y}_{2}}\\
0&
{x}_{1}&
{x}_{2}&
0&
0&
{-{y}_{2}}&
0&
0&
0\\
\end{pmatrix}\egroup
\end{equation}
The matrix in equation (\ref{example_matrix}) is written with respect to
a homogeneous basis of the codomain that admits an upper triangular change of 
basis to a homogeneous basis of weight vectors.
To be more precise, let $E,F$ be free $A$-modules, $\mathcal{E},\mathcal{F}$ their respective
coordinate bases, and $\varphi \colon E\to F$ a map such that
$\mathcal{M}^{\mathcal{E}}_{\mathcal{F}} (\varphi)$ is the matrix in 
equation (\ref{example_matrix}).
Then, by example \ref{exa:3}, $F$ admits a basis of weight vectors
$\tilde{\mathcal{F}} = \{\tilde{f}_1,\ldots, \tilde{f}_5\}$ such
that $\mathcal{M}^{\mathcal{F}}_{\tilde{\mathcal{F}}} (\id_F)$ is
upper triangular. Moreover, if $W$ denotes the list of weights of
$\tilde{f}_1,\ldots, \tilde{f}_5$, then we know
\[W = \{(0,1,0,0),(1,0,0,0),
  (0,0,0,2),(0,0,1,1),(0,0,2,0)\}.\]
Our goal is to apply algorithm \ref{algo_multiple_degrees} to the
pair $(\mathcal{M}^{\mathcal{E}}_{\mathcal{F}} (\varphi),W)$.

The permutation matrix
\[P = 
\bgroup\begin{pmatrix}1&
0&
0&
0&
0&
0&
0&
0&
0\\
0&
1&
0&
0&
0&
0&
0&
0&
0\\
0&
0&
1&
0&
0&
0&
0&
0&
0\\
0&
0&
0&
1&
0&
0&
0&
0&
0\\
0&
0&
0&
0&
1&
0&
0&
0&
0\\
0&
0&
0&
0&
0&
0&
0&
1&
0\\
0&
0&
0&
0&
0&
1&
0&
0&
0\\
0&
0&
0&
0&
0&
0&
1&
0&
0\\
0&
0&
0&
0&
0&
0&
0&
0&
1\\
\end{pmatrix}\egroup\]
has the property that $\mathcal{M}^{\mathcal{E}}_{\mathcal{F}} (\varphi) P =
(M_1 | M_2 | M_3)$, where $M_1$, $M_2$ and $M_3$ are the matrices
\[
\begin{pmatrix}{x}_{1}\\
{-{x}_{2}}\\
0\\
0\\
0\\
\end{pmatrix},
\quad
\begin{pmatrix}0&
{-{y}_{1}^{2}}&
0&
{-{y}_{1} {y}_{2}}&
0&
{-{y}_{2}^{2}}\\
{-{y}_{1}^{2}}&
0&
{-{y}_{1} {y}_{2}}&
0&
{-{y}_{2}^{2}}&
0\\
0&
0&
0&
0&
{x}_{1}&
{x}_{2}\\
0&
0&
{x}_{1}&
{x}_{2}&
0&
0\\
{x}_{1}&
{x}_{2}&
0&
0&
0&
0\\
\end{pmatrix},
\quad
\begin{pmatrix}0&
0\\
0&
0\\
0&
{y}_{1}\\
{y}_{1}&
{-{y}_{2}}\\
{-{y}_{2}}&
0\\
\end{pmatrix}.
\]
Observe the columns of $M_i$ have degree $d_i$ and
$d_1 = (2,0)$, $d_2 = (1,2)$, $d_3 = (0,3)$.

Algorithm \ref{algo_single_degree} applied to the pair $(M_1,W)$ finds the following
Gr\"obner basis and change of basis matrices:
\[G_1=
\begin{pmatrix}{x}_{1}\\
{-{x}_{2}}\\
0\\
0\\
0\\
\end{pmatrix},
\qquad
C_1=
\begin{pmatrix}1\\
\end{pmatrix}.
\]
The leading term of the only column of $G_1$ is $x_1f_1$ and
\[\weight (x_1) + w_1 = (1,0,0,0) + (0,1,0,0) = (1,1,0,0).\]
Hence this run returns  the pair $(C_1, V_1)$, where $V_1 =
\{(1,1,0,0)\}$.

Now consider the pair $(M_2,W)$. Our computation yields the
following Gr\"obner basis and change of basis matrices:
\begin{gather*}
G_2=
\begin{pmatrix}{y}_{2}^{2}&
0&
{y}_{1} {y}_{2}&
0&
{y}_{1}^{2}&
0\\
0&
{y}_{2}^{2}&
0&
{y}_{1} {y}_{2}&
0&
{y}_{1}^{2}\\
{-{x}_{2}}&
{-{x}_{1}}&
0&
0&
0&
0\\
0&
0&
{-{x}_{2}}&
{-{x}_{1}}&
0&
0\\
0&
0&
0&
0&
{-{x}_{2}}&
{-{x}_{1}}\\
\end{pmatrix},\\
C_2=
\begin{pmatrix}0&
0&
0&
0&
0&
{-1}\\
0&
0&
0&
0&
{-1}&
0\\
0&
0&
0&
{-1}&
0&
0\\
0&
0&
{-1}&
0&
0&
0\\
0&
{-1}&
0&
0&
0&
0\\
{-1}&
0&
0&
0&
0&
0\\
\end{pmatrix}.
\end{gather*}
We list the leading terms of the columns of $G_2$ and the corresponding
weight calculations performed by algorithm \ref{algo_single_degree}:
\begin{align*}
  &y_2^2 f_1, & &\weight(y_2^2) + w_1 = (0,0,0,2) + (0,1,0,0) = (0,1,0,2),\\
  &y_2^2 f_2, & &\weight(y_2^2) + w_2 = (0,0,0,2) + (1,0,0,0) = (1,0,0,2),\\
  &y_1 y_2 f_1,& & \weight(y_1 y_2) + w_1 = (0,0,1,1) + (0,1,0,0) = (0,1,1,1),\\
  &y_1 y_2 f_2,& & \weight(y_1 y_2) + w_2 = (0,0,1,1) + (1,0,0,0) = (1,0,1,1),\\
  &y_1^2 f_1,& & \weight(y_1^2) + w_1 = (0,0,2,0) + (0,1,0,0) = (0,1,2,0),\\
  &y_1^2 f_2,& & \weight(y_1^2) + w_2 = (0,0,2,0) + (1,0,0,0) = (1,0,2,0).
\end{align*}
The pair thus produced is $(C_2,V_2)$, where $V_2$ is the ordered list
containing the weights above.

Finally consider the pair $(M_3,W)$. Starting from $M_3$ we obtain
the following Gr\"obner basis and change of basis matrices:
\[
G_3 = 
\begin{pmatrix}0&
0\\
0&
0\\
{y}_{1}&
0\\
{-{y}_{2}}&
{y}_{1}\\
0&
{-{y}_{2}}\\
\end{pmatrix},
\qquad
C_3 =
\begin{pmatrix}0&
1\\
1&
0\\
\end{pmatrix}.
\]
The leading terms of the columns of $G_3$ and the corresponding
weight computations are given by:
\begin{align*}
  &y_1 f_3,& &\weight (y_1) + w_3 = (0,0,1,0) + (0,0,0,2) = (0,0,1,2),\\
  &y_1 f_4,& &\weight (y_1) + w_4 = (0,0,1,0) + (0,0,1,1) = (0,0,2,1).
\end{align*}
This run of algorithm \ref{algo_single_degree} returns the pair
$(C_3,V_3)$, where $V_3 = \{(0,0,1,2),(0,0,2,1)\}$.

The final steps of algorithm \ref{algo_multiple_degrees} combine
the previous outputs to give the change of basis matrix:
\[C = P(C_1 \oplus C_2 \oplus C_3) =
\begin{pmatrix}1&
0&
0&
0&
0&
0&
0&
0&
0\\
0&
0&
0&
0&
0&
0&
{-1}&
0&
0\\
0&
0&
0&
0&
0&
{-1}&
0&
0&
0\\
0&
0&
0&
0&
{-1}&
0&
0&
0&
0\\
0&
0&
0&
{-1}&
0&
0&
0&
0&
0\\
0&
0&
0&
0&
0&
0&
0&
0&
1\\
0&
0&
{-1}&
0&
0&
0&
0&
0&
0\\
0&
{-1}&
0&
0&
0&
0&
0&
0&
0\\
0&
0&
0&
0&
0&
0&
0&
1&
0\\
\end{pmatrix},
\]
and the list of weights:
\begin{gather*}
  V = V_1 \cup V_2 \cup V_3 = \{(1,1,0,0),(0,1,0,2),(1,0,0,2),\\
  (0,1,1,1),(1,0,1,1),(0,1,2,0),(1,0,2,0),(0,0,1,2),(0,0,2,1)\}.
\end{gather*}
The pair $(C,V)$ is returned. From the weights in $V$ and the
degrees of elements in $\mathcal{E}$, we deduce
\begin{align*}
  E\cong &\left(\bigwedge^2 \CC^2 \otimes \CC\right) \otimes_\CC A(-2,0)\oplus\\
  \oplus &\left(\CC^2 \otimes \Sym^2 \CC^2\right) \otimes_\CC A(-1,-2) \oplus\\
  \oplus &\left(\CC \otimes \mathbb{S}_{(2,1)} \CC^2\right) \otimes_\CC A(0,-3)
\end{align*}
as representations of $\GL_2 (\CC) \times \GL_2 (\CC)$ ($\mathbb{S}_{(2,1)}$ denotes a Schur functor \cite[\S2.1]{MR1988690}).
\end{ex}

\begin{ex}
  \label{exa:5}
  In this example, we examine the second syzygies of the map
  $\varphi$ in example \ref{exa:3}.
  As returned by our computer algebra system,
  the matrix of the third syzygies is
  \[
  \begin{pmatrix}{y}_{1}^{2}&
{y}_{1} {y}_{2}&
0&
0&
{y}_{2}^{2}&
0&
0\\
{-{x}_{2}}&
0&
{y}_{2}&
0&
0&
0&
0\\
{x}_{1}&
0&
0&
{y}_{2}&
0&
0&
0\\
0&
{-{x}_{2}}&
{-{y}_{1}}&
0&
0&
{y}_{2}&
0\\
0&
{x}_{1}&
0&
{-{y}_{1}}&
0&
0&
{y}_{2}\\
0&
0&
{x}_{1}&
{x}_{2}&
0&
0&
0\\
0&
0&
0&
0&
{-{x}_{2}}&
{-{y}_{1}}&
0\\
0&
0&
0&
0&
{x}_{1}&
0&
{-{y}_{1}}\\
0&
0&
0&
0&
0&
{x}_{1}&
{x}_{2}\\
\end{pmatrix},
  \]
where the degrees of the columns are from left to right
$(2,2)$, $(2,2)$, $(1,3)$, $(1,3)$, $(2,2)$, $(1,3)$, $(1,3)$.
Multiplying on the left by the inverse of the matrix $C$
obtained at the end of example \ref{exa:4}, we get
\begin{equation}
\label{example_matrix_2}
\begin{pmatrix}{y}_{1}^{2}&
{y}_{1} {y}_{2}&
0&
0&
{y}_{2}^{2}&
0&
0\\
0&
0&
0&
0&
{-{x}_{1}}&
0&
{y}_{1}\\
0&
0&
0&
0&
{x}_{2}&
{y}_{1}&
0\\
0&
{-{x}_{1}}&
0&
{y}_{1}&
0&
0&
{-{y}_{2}}\\
0&
{x}_{2}&
{y}_{1}&
0&
0&
{-{y}_{2}}&
0\\
{-{x}_{1}}&
0&
0&
{-{y}_{2}}&
0&
0&
0\\
{x}_{2}&
0&
{-{y}_{2}}&
0&
0&
0&
0\\
0&
0&
0&
0&
0&
{x}_{1}&
{x}_{2}\\
0&
0&
{x}_{1}&
{x}_{2}&
0&
0&
0\\
\end{pmatrix}.
\end{equation}

Let $E,F$ be free $A$-modules, $\mathcal{E},\mathcal{F}$
their respective coordinate bases,
and $\varphi \colon E\to F$ a map such that
$\mathcal{M}^{\mathcal{E}}_{\mathcal{F}} (\varphi)$ is the matrix in 
equation (\ref{example_matrix_2}).
Then, by example \ref{exa:4}, $F$ admits a basis of weight vectors
$\tilde{\mathcal{F}} = \{\tilde{f}_1,\ldots, \tilde{f}_9\}$ such
that $\mathcal{M}^{\mathcal{F}}_{\tilde{\mathcal{F}}} (\id_F)$ is
upper triangular. Moreover, if $W$ denotes the list of weights of
$\tilde{f}_1,\ldots, \tilde{f}_9$, then we know
\begin{gather*}
  W = \{(1,1,0,0),(0,1,0,2),(1,0,0,2),(0,1,1,1),\\
  (1,0,1,1),(0,1,2,0),(1,0,2,0),(0,0,1,2),(0,0,2,1)\}.
\end{gather*}

We will apply algorithm \ref{algo_multiple_degrees} to the
pair $(\mathcal{M}^{\mathcal{E}}_{\mathcal{F}} (\varphi),W)$.

We see the matrix
\[P=
\begin{pmatrix}1&
0&
0&
0&
0&
0&
0\\
0&
1&
0&
0&
0&
0&
0\\
0&
0&
0&
1&
0&
0&
0\\
0&
0&
0&
0&
1&
0&
0\\
0&
0&
1&
0&
0&
0&
0\\
0&
0&
0&
0&
0&
1&
0\\
0&
0&
0&
0&
0&
0&
1\\
\end{pmatrix}
\]
gives $\mathcal{M}^{\mathcal{E}}_{\mathcal{F}} (\varphi) P =
(M_1 | M_2)$, where
\[M_1 =
\begin{pmatrix}{y}_{1}^{2}&
{y}_{1} {y}_{2}&
{y}_{2}^{2}\\
0&
0&
{-{x}_{1}}\\
0&
0&
{x}_{2}\\
0&
{-{x}_{1}}&
0\\
0&
{x}_{2}&
0\\
{-{x}_{1}}&
0&
0\\
{x}_{2}&
0&
0\\
0&
0&
0\\
0&
0&
0\\
\end{pmatrix},
\qquad
M_2=
\begin{pmatrix}0&
0&
0&
0\\
0&
0&
0&
{y}_{1}\\
0&
0&
{y}_{1}&
0\\
0&
{y}_{1}&
0&
{-{y}_{2}}\\
{y}_{1}&
0&
{-{y}_{2}}&
0\\
0&
{-{y}_{2}}&
0&
0\\
{-{y}_{2}}&
0&
0&
0\\
0&
0&
{x}_{1}&
{x}_{2}\\
{x}_{1}&
{x}_{2}&
0&
0\\
\end{pmatrix}.\]
Moreover, the columns of $M_i$ have degree $d_i$,
where $d_1 = (2,2)$ and $d_2 = (1,3)$.

First we apply algorithm \ref{algo_single_degree}
to the pair $(M_1,W)$.
Our computer algebra system provides the following Gr\"obner
basis and change of basis matrices:
\[
G_1=
\begin{pmatrix}{y}_{2}^{2}&
{y}_{1} {y}_{2}&
{y}_{1}^{2}\\
{-{x}_{1}}&
0&
0\\
{x}_{2}&
0&
0\\
0&
{-{x}_{1}}&
0\\
0&
{x}_{2}&
0\\
0&
0&
{-{x}_{1}}\\
0&
0&
{x}_{2}\\
0&
0&
0\\
0&
0&
0\\
\end{pmatrix},
\qquad
C_1 =
\begin{pmatrix}0&
0&
1\\
0&
1&
0\\
1&
0&
0\\
\end{pmatrix}.
\]
The leading terms of the columns of $G_1$ are listed below
together with the corresponding weight calculations:
\begin{align*}
  &y_2^2 f_1,& &\weight (y_2^2) + w_1 = (0,0,0,2) + (1,1,0,0) = (1,1,0,2),\\
  &y_1 y_2 f_1,& &\weight (y_1 y_2) + w_1 = (0,0,1,1) + (1,1,0,0) = (1,1,1,1),\\
  &y_1^2 f_1,& &\weight (y_1^2) + w_1 = (0,0,2,0) + (1,1,0,0) = (1,1,2,0).
\end{align*}
If we let $V_1=\{(1,1,0,2),(1,1,1,1),(1,1,2,0)\}$,
then the algorithm returns the pair $(C_1,V_1)$.

Next we apply the same procedure to the pair $(M_2, W)$.
The computed Gr\"obner basis and change of basis matrices are
\[G_2=
\begin{pmatrix}0&
0&
0&
0\\
{y}_{1}&
0&
0&
0\\
0&
0&
{y}_{1}&
0\\
{-{y}_{2}}&
{y}_{1}&
0&
0\\
0&
0&
{-{y}_{2}}&
{y}_{1}\\
0&
{-{y}_{2}}&
0&
0\\
0&
0&
0&
{-{y}_{2}}\\
{x}_{2}&
0&
{x}_{1}&
0\\
0&
{x}_{2}&
0&
{x}_{1}\\
\end{pmatrix},
\qquad
C_2=
\begin{pmatrix}0&
0&
0&
1\\
0&
1&
0&
0\\
0&
0&
1&
0\\
1&
0&
0&
0\\
\end{pmatrix}.
\]
Below are the leading terms of the columns of $G_2$ and the weights
calculated from them:
\begin{align*}
  &x_2 f_8& &\weight (x_2) + w_8 = (0,1,0,0) + (0,0,1,2) = (0,1,1,2),\\
  &x_2 f_9,& &\weight (x_2) + w_9 = (0,1,0,0) + (0,0,2,1) = (0,1,2,1),\\
  &x_1 f_8& &\weight (x_1) + w_8 = (1,0,0,0) + (0,0,1,2) = (1,0,1,2),\\
  &x_1 f_9,& &\weight (x_1) + w_9 = (1,0,0,0) + (0,0,2,1) = (1,0,2,1).
\end{align*}
Hence, if $V_2 = \{(0,1,1,2),(0,1,2,1),(1,0,1,2),
(1,0,2,1)\}$, this run of algorithm \ref{algo_single_degree} outputs the pair $(C_2,V_2)$.

Finally, after combining these intermediate steps, algorithm
\ref{algo_multiple_degrees} constructs the matrix
\[
C = P (C_1 \oplus C_2) =
\begin{pmatrix}0&
0&
1&
0&
0&
0&
0\\
0&
1&
0&
0&
0&
0&
0\\
0&
0&
0&
0&
0&
0&
1\\
0&
0&
0&
0&
1&
0&
0\\
1&
0&
0&
0&
0&
0&
0\\
0&
0&
0&
0&
0&
1&
0\\
0&
0&
0&
1&
0&
0&
0\\
\end{pmatrix},\]
the list of weights
\begin{gather*}
  V = V_1 \cup V_2 = \{(1,1,0,2),(1,1,1,1),(1,1,2,0),\\
  (0,1,1,2),(0,1,2,1),(1,0,1,2),(1,0,2,1)\},
\end{gather*}
and returns the pair $(C,V)$. We conclude
\begin{align*}
  E \cong & \left(\bigwedge^2 \CC^2 \otimes \Sym^2 \CC^2\right)
\otimes_\CC A(-2,-2)\\
  \oplus & \left( \CC^2 \otimes \mathbb{S}_{(2,1)} \CC^2 \right)
\otimes_\CC A(-1,-3)
\end{align*}
as a representation of $\GL_2 (\CC) \times \GL_2 (\CC)$.
\end{ex}

\begin{remark}
Rather than computing a Gr\"obner basis for $\im\varphi$, algorithm \ref{algo_multiple_degrees} splits $\mathcal{M}^{\mathcal{E}}_{\mathcal{F}} (\varphi)$ into submatrices with columns of the same degree and then applies algorithm \ref{algo_single_degree} to each one of them. If  $\mathcal{M}^{\mathcal{E}}_{\mathcal{F}} (\varphi)$ has columns in degrees $d_1,\ldots,d_l$, then its reduced Gr\"obner basis (even a truncated one) could have elements in other degrees which are not part of a minimal generating set of $\im\varphi$. By computing Gr\"obner bases in single degrees we can avoid this issue altogether, hence producing another minimal map.
\end{remark}

\begin{remark}
Algorithms \ref{algo_single_degree} and \ref{algo_multiple_degrees} return the matrix $C = \mathcal{M}^{\mathcal{E}'}_{\mathcal{E}} (\id_E)$. The existence of a homogeneous basis of weight vectors $\tilde{\mathcal{E}}$ of $E$ such that $\mathcal{M}^{\mathcal{E}'}_{\tilde{\mathcal{E}}} (\id_E)$ is upper (resp. lower) triangular, is guaranteed by theorem \ref{main_thm}. However algorithms \ref{algo_single_degree} and \ref{algo_multiple_degrees} do not provide any means to recover the matrix $\mathcal{M}^{\mathcal{E}'}_{\tilde{\mathcal{E}}} (\id_E)$.
\end{remark}

\subsection{Going forward}\label{sec_algo_forward}
For a map $\varphi \colon E\rightarrow F$ of free modules in $\mod_{\righttoleftarrow T} A$, algorithm \ref{algo_multiple_degrees} provides a tool to recover the weights of $E$ from the weights of $F$ by ``moving backwards along $\varphi$''. A natural question is whether it is possible to go ``forward'' instead and recover the weights of  $F$ from those of $E$. Here are the assumptions for this section.
\begin{enumerate}[label=(\alph*),wide]
\item $\varphi\colon E\rightarrow F$ is a map of free modules in $\mod_{\righttoleftarrow T} A$.
\item The dual map $\varphi^{\vee} \colon F^\vee \to E^\vee$ is minimal.
\item The map $\varphi$ is provided in matrix form $\mathcal{M}^{\mathcal{E}}_{\mathcal{F}} (\varphi)$, where $\mathcal{E}$ is a homogeneous basis of $E$ and $\mathcal{F}$ is a homogeneous basis of $F$.
\item $\TT^n \langle \mathcal{E} \rangle$ is equipped with a position up (resp. down) module term ordering.
\item $E$ admits a homogeneous basis of weight vectors $\tilde{\mathcal{E}} = \{\tilde{e}_1,\ldots,\tilde{e}_r\}$ such that $\mathcal{M}^{\mathcal{E}}_{\tilde{\mathcal{E}}} (\id_E)$ is upper (resp. lower) triangular.
\item $V = \{v_1,\ldots,v_r\}$ is an ordered list with $v_i = \weight (\tilde{e}_i)$, $\forall i\in\{1,\ldots,r\}$.
\end{enumerate}

For a matrix $M$, let $M^\top$ denote the transpose of $M$. For an ordered list of weights $W = \{w_1,\ldots,w_s\}$, let $-W$ denote the ordered list of weights $\{-w_1,\ldots,-w_s\}$.

\begin{algo}[weight propagation going forward along a map]\label{algo_forward}
\algrenewcommand\algorithmicrequire{\textbf{Input:}}
\algrenewcommand\algorithmicensure{\textbf{Output:}}
\begin{algorithmic}[1]
\Require{\hfill
\begin{itemize}
\item $\mathcal{M}^{\mathcal{E}}_{\mathcal{F}} (\varphi)$, a matrix as in assumption (c) above
\item $V$, a list of weights as in assumption (f) above
\end{itemize}
}
\Ensure{\hfill
  \begin{itemize}
  \item $C = \mathcal{M}^{\mathcal{F}}_{\mathcal{F}'} (\id_F)$,
    a change of basis in $F$ such that the leading terms of the
    columns of $\mathcal{M}^{(\mathcal{F}')^\vee}_{\mathcal{E}^\vee} (\varphi^\vee)$
    are all different
  \item $W = \{w_1,\ldots,w_s\}$, a list of weights such that
    $w_i = \weight (\tilde{f}_i)$ for a homogeneous basis of weight
    vectors $\{\tilde{f}_1,\ldots,\tilde{f}_s\}$ of $F$
  \end{itemize}
~\hfill
}
\Function{PropagateForward}{$\mathcal{M}^{\mathcal{E}}_{\mathcal{F}} (\varphi),V$}
	\If{the module term ordering on $\TT^n \langle \mathcal{E} \rangle$ is position up}
		\State equip $\TT^n \langle \mathcal{E}^\vee \rangle$ with a position down module term ordering
	\ElsIf{the module term ordering on $\TT^n \langle \mathcal{E} \rangle$ is position down}
		\State equip $\TT^n \langle \mathcal{E}^\vee \rangle$ with a position up module term ordering
	\EndIf
	\State set $(C',W') := \text{\textsc{Propagate}} (\mathcal{M}^{\mathcal{E}}_{\mathcal{F}} (\varphi)^\top, -V)$
	\State set $C := (C')^\top$
	\State set $W := -W'$
	\State \textbf{return} $(C,W)$
\EndFunction
\end{algorithmic}
\end{algo}

We will present examples of algorithm \ref{algo_forward}
in action at the end of section \ref{sec_algo_res},
as part of our application to free resolutions.

\begin{prop}\label{algo_forward_prop}
Under the assumptions of this section, applying algorithm \ref{algo_forward} yields a pair $(C, W=\{w_1,\ldots,w_s\})$ that satisfies the following properties:
\begin{enumerate}[label=\Roman*.,ref=\Roman*,wide]
\item $F$ admits a homogeneous basis $\mathcal{F}' = \{f'_1,\ldots,f'_s\}$ such that $C = \mathcal{M}^{\mathcal{F}}_{\mathcal{F}'} (\id_F)$ and the terms $\LT (\varphi^\vee ((f'_1)^\vee)),\ldots,\LT (\varphi^\vee ((f'_s)^\vee))$ are all different;
\item $F$ admits a homogeneous basis of weight vectors $\tilde{\mathcal{F}} = \{\tilde{f}_1,\ldots,\tilde{f}_s\}$ such that $\mathcal{M}_{\mathcal{F}'}^{\tilde{\mathcal{F}}} (\id_E)$ is lower (resp. upper) triangular and $\weight (\tilde{f}_i) = w_i$, for all $i \in \{1,\ldots,s\}$.
\end{enumerate}
\end{prop}
\begin{proof}
Algorithm \ref{algo_forward} applies algorithm \ref{algo_multiple_degrees} to the dual map $\varphi^{\vee} \colon F^\vee \to E^\vee$, which is assumed to be minimal. Notice that the matrix $\mathcal{M}^{\mathcal{E}}_{\mathcal{F}} (\varphi)^\top$ used on line 7 of algorithm \ref{algo_forward}, is in fact the matrix $\mathcal{M}_{\mathcal{E}^\vee}^{\mathcal{F}^\vee} (\varphi^\vee)$, the matrix of $\varphi^\vee$ with respect to the dual bases ${\mathcal{E}^\vee}$ of $E^\vee$ and ${\mathcal{F}^\vee}$ of $F^\vee$.

If $\mathcal{M}^{\mathcal{E}}_{\tilde{\mathcal{E}}} (\id_E)$ is upper (resp. lower) triangular, then $\mathcal{M}_{\mathcal{E}^{\vee}}^{\tilde{\mathcal{E}}^{\vee}} (\id_{E^\vee}) = \mathcal{M}^{\mathcal{E}}_{\tilde{\mathcal{E}}} (\id_E)^\top$ is lower (resp. upper) triangular. Therefore, to use algorithm \ref{algo_multiple_degrees} correctly, $\TT^n \langle \mathcal{E}^\vee \rangle$ must be equipped with a position down (resp. up) module term ordering, as is done on lines 2-6.

The weights of $E$ are provided in $V$ as part of the input. Thanks to proposition \ref{dual_free_mod_structure}, $E^\vee \cong (E/\mathfrak{m} E)^* \otimes_\KK A$, while $E\cong (E/\mathfrak{m} E) \otimes_\KK A$ by proposition \ref{free_mod_structure}. Moreover, by proposition \ref{weight_in_dual}, the weights of $(E/\mathfrak{m} E)^*$ are the opposites of the weights of $E/\mathfrak{m} E$. Thus $\weight(\tilde{e}_i^\vee) = -\weight(\tilde{e}_i)$, and that is why we apply algorithm \ref{algo_multiple_degrees} with the list $-V$ as input.

After algorithm \ref{algo_multiple_degrees} is applied, the results must be transferred from $F^\vee$ to $F$. For that reason, we transpose $C'$ on line 8 and switch to the opposite weights of $W'$ on line 9.
\end{proof}

Proposition \ref{algo_forward_prop} requires that the map $\varphi^{\vee} \colon F^\vee \to E^\vee$ is minimal. Unfortunately, $\varphi\colon E\to F$ being minimal does not, in general, imply that $\varphi^{\vee} \colon F^\vee \to E^\vee$ is minimal.

\begin{ex}\label{non_minimal_dual}
Let $A=\mathbb{C} [x]$. The map of free $A$-modules
\[
A(-1) \xrightarrow{
\left(
\begin{smallmatrix}
x\\ x
\end{smallmatrix}
\right)
} A^2
\]
is minimal. However, the dual map
\[
A^2 \xrightarrow{
\left(
\begin{smallmatrix}
x & x
\end{smallmatrix}
\right)
} A(1)
\]
is not minimal because the element $(1,-1)^\top$ belongs to the $\CC$-vector space generated by the coordinate basis of $A^2$ and is sent to zero (compare with proposition \ref{minimality}).
\end{ex}

The following is a useful criterion for the minimality of a dual map, when computing the weights with respect to a torus contained inside a group $G$.
\begin{prop}
Let $G$ be an algebraic group over $\KK$ and let $\varphi\colon E\to F$ be a non zero map of free modules in $\mod_{\righttoleftarrow G} A$. If $F/\mathfrak{m} F$ is an irreducible representation of $G$, then the dual map $\varphi^\vee \colon F^\vee \to E^\vee$ is minimal.
\end{prop}
\begin{proof}
Because $F/\mathfrak{m} F$ is an irreducible representation of $G$, the graded vector space $F/\mathfrak{m} F$ is concentrated in a single degree $d\in\ZZ^m$. Equivalently, $F$ is generated in a single degree $d$ and $F_d$ is an irreducible representation of $G$. Looking at duals, we have that $F^\vee$ is generated in degree $-d$ and $(F^\vee)_{-d} \cong \Hom_\KK (F_d,\KK)$ by proposition \ref{dual_free_mod_structure}. Then $(F^\vee)_{-d}$ is irreducible because $F_d$ is. Notice that $(\ker\varphi^\vee)_{-d}$ is a subrepresentation of $G$ inside $(F^\vee)_{-d}$, hence $(\ker\varphi^\vee)_{-d} = (F^\vee)_{-d}$ or $(\ker\varphi^\vee)_{-d} = 0$ by the irreducibility of $(F^\vee)_{-d}$. The first option would imply $\varphi^\vee$ is the zero map, given that $F^\vee$ is generated in degree $-d$, and this would violate the assumption that $\varphi$ is non zero. Therefore $(\ker\varphi^\vee)_{-d} = 0$ and thus $\varphi^\vee$ is minimal by part \ref{min2} of proposition \ref{minimality}.
\end{proof}

\subsection{Weight propagation along resolutions}\label{sec_algo_res}
We can now develop an algorithm to propagate weights along minimal free resolutions. The following will be assumed throughout this section.
\begin{enumerate}[label=(\alph*),ref=(\alph*),wide]
\item The complex of free modules and maps in $\mod_{\righttoleftarrow T} A$ 
\[F_\bullet\colon \quad 0\to F_m \xrightarrow{d_m} F_{m-1} \to \ldots \to F_i \xrightarrow{d_i} F_{i-1} \to \ldots \to F_1 \xrightarrow{d_1} F_0\]
with $m\leq n$, is a minimal free resolution of a module $M$ in  $\mod_{\righttoleftarrow T} A$.
\item For all $i\in\{1,\ldots,n\}$, the map $d_i$ is provided in matrix form $\mathcal{M}^{\mathcal{F}_i}_{\mathcal{F}_{i-1}} (d_i)$, where $\mathcal{F}_i$ is a homogeneous basis of $F_i$ (also for $i=0$).
\item\label{cond1} There exists $c\in\{0,\ldots,n\}$ such that $F_c$ admits a homogeneous basis of weight vectors $\tilde{\mathcal{F}}_c$ such that $\mathcal{M}^{\mathcal{F}_c}_{\tilde{\mathcal{F}}_c} (\id_{F_c})$ is upper (resp. lower) triangular.
\item $V_c$ is an ordered list containing the weights of the elements of $\tilde{\mathcal{F}}_c$.
\item\label{cond2} For all $i\in\{1,\ldots,c\}$, the maps $d_i^\vee \colon F^\vee_{i-1} \to F^\vee_i$ are minimal.
\item For all $i\in\{0,\ldots,n\}$, $\TT^n \langle \mathcal{F}_i \rangle$ is equipped with a position up (resp. down) module term ordering.
\end{enumerate}
The reason for requiring conditions \ref{cond1}--\ref{cond2} is so that the weights can be specified for any one free module in the complex $F_\bullet$. The idea is then to propagate weights backwards along the maps $d_{c+1},\ldots,d_m$ (using algorithm \ref{algo_multiple_degrees}) and forward along the maps $d_c,\ldots,d_1$ (using algorithm \ref{algo_forward}). As seen in examples \ref{exa:4} and \ref{exa:5}, additional change of bases will be required along the way to ensure the bases used are related to bases of weight vectors by triangular change of bases.

\begin{remark}
A minimal free resolution of a module $M$ is typically obtained, in a computational setting, from a presentation of $M$. If the presentation is minimal, then it is also the first differential $d_1\colon F_1\to F_0$ in the resolution. In our experience, homogeneous bases of weight vectors are often the most natural choice to write a matrix for the presentation $d_1$. Using the notation above, this means we have $\tilde{\mathcal{F}}_0 = \mathcal{F}_0$; then $\mathcal{M}^{\mathcal{F}_0}_{\tilde{\mathcal{F}}_0} (\id_{F_0})$ is the identity matrix and assumption \ref{cond1} is easily met.
\end{remark}

\begin{remark}
Sometimes a minimal free resolution of a module $M$ can be constructed starting from some differential $d_{c+1}\colon F_{c+1}\to F_c$ other than the first one. Our algorithm is designed to deal with this more general setup.
For examples of resolutions constructed from the middle (or from the end) we invite the reader to consult \cite{Galetto:2013aa}.
\end{remark}

\begin{algo}[weight propagation along a resolution]\label{algo_res}
\algrenewcommand\algorithmicrequire{\textbf{Input:}}
\algrenewcommand\algorithmicensure{\textbf{Output:}}
\begin{algorithmic}[1]
\Require{\hfill
\begin{itemize}
\item $\mathcal{M}^{\mathcal{F}_1}_{\mathcal{F}_{0}} (d_1),\ldots,\mathcal{M}^{\mathcal{F}_m}_{\mathcal{F}_{m-1}} (d_m)$, matrices of the maps in a minimal free resolution as in assumption (b) above
\item $V_c$, a list of weights as in assumption (d) above
\end{itemize}
}
\Ensure{\hfill
  \begin{itemize}
  \item $(V_0,\ldots,V_m)$, a tuple of lists, with $V_i$ a list
    of weights for a homogeneous basis of weight vectors of $F_i$
  \end{itemize}
~\hfill
}
\Function{PropagateResolution}{$\mathcal{M}^{\mathcal{F}_1}_{\mathcal{F}_{0}} (d_1),\ldots,\mathcal{M}^{\mathcal{F}_m}_{\mathcal{F}_{m-1}} (d_m),V_c$}
	\State $C_c := \mathcal{M}^{\mathcal{F}_c}_{\mathcal{F}_c} (\id_{F_c})$
	\For{$i\in\{1,\ldots,m-c\}$}
		\State $(C_{c+i},V_{c+i}) := \text{\textsc{Propagate}} (C^{-1}_{c+i-1} \mathcal{M}^{\mathcal{F}_{c+i}}_{\mathcal{F}_{c+i-1}} (d_{c+i}), V_{c+i-1})$
	\EndFor
	\For{$i\in\{1,\ldots,c\}$}
		\State $(C_{c-i},V_{c-i}) := \text{\textsc{PropagateForward}} (\mathcal{M}^{\mathcal{F}_{c-i+1}}_{\mathcal{F}_{c-i}} (d_{c-i+1}) C^{-1}_{c-i+1}, V_{c-i+1})$
	\EndFor
	\State \textbf{return} $(V_0,\ldots,V_m)$
\EndFunction
\end{algorithmic}
\end{algo}

\begin{ex}
  \label{exa:6}
  Consider the same setup as in example \ref{exa:3}.
  Let $M$ be the module $A/(x_1,x_2,y_1^2,y_1 y_2,y_2^2)$.
  We can compute a minimal free resolution of $M$
  using a computer algebra system.
  The result is a complex
  \[F_\bullet\colon \quad 0\to F_4 \xrightarrow{d_4}
  F_3 \xrightarrow{d_3} F_2 \xrightarrow{d_2}
  F_1 \xrightarrow{d_1} F_0\]
  where $F_0 = A$.
  In particular, $F_0$ has a homogeneous basis (of
  weight vectors) $\tilde{\mathcal{F}}=\mathcal{F}=
  \{1_A\}$ and its only element has weight $\{(0,0,0,0)\}$.
  The maps $d_1,\ldots,d_4$ are represented by matrices
  with respect to homogeneous bases $\mathcal{F}_i$ of $F_i$.
  In our case, the map $d_1$ and the bases $\mathcal{F}_0$,
   $\mathcal{F}_1$ are supplied by the user so that
  \[\mathcal{M}^{\mathcal{F}_1}_{\mathcal{F}_0} (d_1)=
  \begin{pmatrix}
    x_1 & x_2 & y_1^2 & y_1 y_2 & y_2^2
  \end{pmatrix}.\]
  The remaining bases and matrices are determined by
  the software.
  We will apply algorithm \ref{algo_res} to the tuple
  $(\mathcal{M}^{\mathcal{F}_1}_{\mathcal{F}_0} (d_1),
  \mathcal{M}^{\mathcal{F}_2}_{\mathcal{F}_1} (d_2),
  \mathcal{M}^{\mathcal{F}_3}_{\mathcal{F}_2} (d_3),
  \mathcal{M}^{\mathcal{F}_4}_{\mathcal{F}_3} (d_4),V_0)$,
  where $V_0 = \{(0,0,0,0)\}$.
  Note the parameter $c=0$, so lines 6 through 8
  are not executed while lines 3 through 5 call
  upon algorithm \ref{algo_multiple_degrees}.
  
  The application of algorithm \ref{algo_multiple_degrees}
  to the pair $(\mathcal{M}^{\mathcal{F}_1}_{\mathcal{F}_0}
  (d_1),V_0)$ was carried out in example \ref{exa:3};
  there the output was the pair $(C,V)$, whereas here
  we call it $(C_1,V_1)$.
  Similarly, example \ref{exa:4} details
  the call to algorithm \ref{algo_multiple_degrees}
  for the pair
  $(C_1^{-1} \mathcal{M}^{\mathcal{F}_2}_{\mathcal{F}_1} (d_2),V_1)$,
  which outputs the pair $(C_2,V_2)$, and
  example \ref{exa:5} details
  the call to algorithm \ref{algo_multiple_degrees}
  for the pair
  $(C_2^{-1} \mathcal{M}^{\mathcal{F}_3}_{\mathcal{F}_2} (d_3),V_2)$,
  which outputs the pair $(C_3,V_3)$.
  We complete the picture by illustrating how algorithm
  \ref{algo_multiple_degrees} is applied to the pair
  $(C_3^{-1} \mathcal{M}^{\mathcal{F}_4}_{\mathcal{F}_3} (d_4),V_3)$.
  
  Recall that
  \[C_3 = 
  \begin{pmatrix}0&
0&
1&
0&
0&
0&
0\\
0&
1&
0&
0&
0&
0&
0\\
0&
0&
0&
0&
0&
0&
1\\
0&
0&
0&
0&
1&
0&
0\\
1&
0&
0&
0&
0&
0&
0\\
0&
0&
0&
0&
0&
1&
0\\
0&
0&
0&
1&
0&
0&
0\\
\end{pmatrix},\]
while
\begin{gather*}
  V_3 = \{(1,1,0,2),(1,1,1,1),(1,1,2,0),\\
  (0,1,1,2),(0,1,2,1),(1,0,1,2),(1,0,2,1)\}.
\end{gather*}
  The matrix for $d_4$ obtained via our computer
  algebra system is
  \[\mathcal{M}^{\mathcal{F}_4}_{\mathcal{F}_3} (d_4)=
  \begin{pmatrix}{-{y}_{2}}&
0\\
{y}_{1}&
{-{y}_{2}}\\
{-{x}_{2}}&
0\\
{x}_{1}&
0\\
0&
{y}_{1}\\
0&
{-{x}_{2}}\\
0&
{x}_{1}\\
\end{pmatrix},
  \]
  with both columns having degree $(2,3)$. Therefore:
  \[C_3^{-1} \mathcal{M}^{\mathcal{F}_4}_{\mathcal{F}_3} (d_4)=
  \begin{pmatrix}0&
{y}_{1}\\
{y}_{1}&
{-{y}_{2}}\\
{-{y}_{2}}&
0\\
0&
{x}_{1}\\
{x}_{1}&
0\\
0&
{-{x}_{2}}\\
{-{x}_{2}}&
0\\
\end{pmatrix}.
  \]
  Since both columns have the same degree, no further splitting
  of blocks is required and we can apply algorithm
  \ref{algo_single_degree} directly.
  Again using our computer algebra system, we obtain the
  following Gr\"obner basis and change of basis matrices:
  \[G=
  \begin{pmatrix}{y}_{1}&
0\\
{-{y}_{2}}&
{y}_{1}\\
0&
{-{y}_{2}}\\
{x}_{1}&
0\\
0&
{x}_{1}\\
{-{x}_{2}}&
0\\
0&
{-{x}_{2}}\\
\end{pmatrix},
  \qquad
  C_4 = 
  \begin{pmatrix}0&
1\\
1&
0\\
\end{pmatrix}.
  \]
  The leading terms of the columns of $G$ are
  \[
  \begin{pmatrix}
    0 & 0\\
    0 & 0\\
    0 & 0\\
    x_1 & 0\\
    0 & x_1\\
    0 & 0\\
    0 & 0\\
  \end{pmatrix},
  \]
  and their weights are given by
  \begin{gather*}
    (1,0,0,0)+(0,1,1,2) = (1,1,1,2),\\
    (1,0,0,0)+(0,1,2,1) = (1,1,2,1).
  \end{gather*}
  Thus the last list $V_4=\{(1,1,1,2),(1,1,2,1)\}$
  is composed and the algorithm concludes its run 
  by returning the tuple $(V_0,V_1,V_2,V_3,V_4)$.
  We have already indicated the representation theoretic
  structure of the modules $F_1$ through $F_3$ in previous 
  examples. From the weights in the list $V_4$ we deduce:
  \[F_4 \cong \left(\bigwedge^2 \CC^2 \otimes
    \mathbb{S}_{(2,1)} \CC^2 \right) \otimes_\CC A(-2,-3).\]
\end{ex}

\begin{prop}
Under the assumptions of this section, algorithm \ref{algo_res} returns a tuple $(V_0,\ldots,V_m)$, where, for all $i\in\{0,\ldots,m\}$, $V_i$ is an ordered list containing the weights of the elements of a homogeneous basis of weight vectors $\tilde{\mathcal{F}}_i$ of $F_i$.
\end{prop}
\begin{proof}
We construct inductively pairs $(C_{c+i},V_{c+i})$ satisfying the following properties:
\begin{itemize}[wide]
\item $\exists \mathcal{F}'_{c+i}$ homogeneous basis of $F_{c+i}$ such that $C_{c+i} = \mathcal{M}^{\mathcal{F}'_{c+i}}_{\mathcal{F}_{c+i}} (\id_{F_{c+i}})$;
\item $\exists \tilde{\mathcal{F}}_{c+i}$ homogeneous basis of weight vectors of $F_{c+i}$ such that $\mathcal{M}^{\mathcal{F}'_{c+i}}_{\tilde{\mathcal{F}}_{c+i}}  (\id_{F_{c+i}})$ is upper (resp. lower) triangular, and $V_{c+i}$ contains the weights of the elements of $\tilde{\mathcal{F}}_{c+i}$.
\end{itemize}
To start the induction, take $C_c =\mathcal{M}^{\mathcal{F}_c}_{\mathcal{F}_c} (\id_{F_c})$, the identity matrix, and $V_c$ as given in the input. Suppose $(C_{c+i-1},V_{c+i-1})$ has been constructed for $i>0$. Then
\[C^{-1}_{c+i-1} \mathcal{M}^{\mathcal{F}_{c+i}}_{\mathcal{F}_{c+i-1}} (d_{c+i}) = \mathcal{M}^{\mathcal{F}_{c+i}}_{\mathcal{F}'_{c+i-1}} (d_{c+i})\]
and the matrix $\mathcal{M}^{\mathcal{F}'_{c+i-1}}_{\tilde{\mathcal{F}}_{c+i-1}}  (\id_{F_{c+i-1}})$ is upper (resp. lower) triangular by the inductive hypothesis. Since all assumptions of algorithm \ref{algo_multiple_degrees} hold, we can apply it as indicated on line 4 to produce the next pair $(C_{c+i},V_{c+i})$. The properties of the pair are then a consequence of proposition \ref{algo_single_degree_prop}. The process ends when $c+i>m$.

Similarly, we can produce pairs $(C_{c-i},V_{c-i})$ for $i>0$, using algorithm \ref{algo_forward} as indicated on line 7. The process ends when $c-i<0$.
\end{proof}

\begin{remark}
If $F_\bullet$ is a minimal free resolution of a Cohen-Macaulay $A$-module $M$ (of grade $g$), then $F^\vee_\bullet$ is a minimal free resolution (of the module $\Ext^g_A (M,A)$) (see \cite[p. 12]{MR1251956}). In particular, the duals of all differentials in $F_\bullet$ are minimal maps. 
\end{remark}

\begin{ex}
Let $A=\mathbb{C} [x]$. The complex
\[
0 \to A(-1) \xrightarrow{
\left(
\begin{smallmatrix}
x\\ x
\end{smallmatrix}
\right)
} A^2
\]
is a minimal free resolution of $A/(x) \oplus A$. The dual of the (only) differential in the resolution is not minimal, as evidenced in example \ref{non_minimal_dual}. The module $A/(x) \oplus A$ has dimension 1 but depth 0, and therefore it is not Cohen-Macaulay.
\end{ex}

\begin{ex}
  \label{exa:7}
  In this example, we will examine a  minimal free resolution
  of the homogeneous coordinate ring of a Grassmannian
  in the Pl\"ucker embedding.
  The Grassmannian we consider is the one parametrizing
  two-dimensional vector subspaces of $\CC^5$.

  Let $A=\Sym (\bigwedge^2 \CC^5)$, the homogeneous coordinate
  ring of the ambient projective space.
  We identify $A$ with the polynomial ring $\CC [p_{i,j}]$,
  for $1\leq i<j\leq 5$, where each variable $p_{i,j}$
  represents a Pl\"ucker coordinate and corresponds
  to a decomposable tensor in $\bigwedge^2 \CC^5$ in the obvious way.
  We will assume $A$ is endowed with the degree reverse
  lexicographic ordering with the variables are sorted as follows:
  \[p_{1,2}> p_{1,3}> p_{2,3}> p_{1,4}> p_{2,4}> p_{3,4}> p_{1,5}> p_{2,5}> p_{3,5}> p_{4,5}.\]
  The group $\GL_5 (\CC)$ acts naturally on $\bigwedge^2 \CC^5$,
  hence on $A$.
  The maximal torus of diagonal matrices in $\GL_5 (\CC)$ acts
  on the variables $p_{i,j}$ making them into weight vectors
  with the following weights:
  \begin{align*}
    &\weight (p_{1,2}) = (1,1,0,0,0),& &\weight (p_{1,3}) = (1,0,1,0,0),\\
    &\weight (p_{2,3}) = (0,1,1,0,0),& &\weight (p_{1,4}) = (1,0,0,1,0),\\
    &\weight (p_{2,4}) = (0,1,0,1,0),& &\weight (p_{3,4}) = (0,0,1,1,0),\\
    &\weight (p_{1,5}) = (1,0,0,0,1),& &\weight (p_{2,5}) = (0,1,0,0,1),\\
    &\weight (p_{3,5}) = (0,0,1,0,1),& &\weight (p_{4,5}) = (0,0,0,1,1).
  \end{align*}
  The homogeneous coordinate ring of our Grassmannian is the
  $A$-module $M=A/I$, where $I$ is the ideal of $A$ generated
  by the Pl\"ucker equations.
  Our computation for a minimal free resolution of $M$ yields
  the complex
  \[F_\bullet\colon \quad 0\to F_3 \xrightarrow{d_3}
  F_2 \xrightarrow{d_2} F_1 \xrightarrow{d_1} F_0;\]
  the matrices of the differentials are described below.
  Note that the free modules $F_i$ have coordinate bases
  $\mathcal{F}_i$ and the sets of terms $\mathbb{T}^n \langle
  \mathcal{F}_i \rangle$ are all endowed with the term over
  position up module term ordering.

  One can apply algorithm \ref{algo_res} with parameter $c=0$,
  knowing that $F_0 = A$ and therefore $V_0 = \{(0,0,0,0,0)\}$.
  This approach determines the weights for homogeneous bases
  of weight vectors in all free modules $F_i$.
  In particular, one can see this way that $F_3 \cong
  (\bigwedge^5 \CC^5)^{\otimes 2} \otimes_\CC A(-5)$,
  since the list of weights $V_3 = \{(2,2,2,2,2)\}$.
  In order to exemplify forward propagation of weights,
  we follow a different approach. Namely,
  we apply algorithm \ref{algo_res} to the tuple
  $(\mathcal{M}^{\mathcal{F}_1}_{\mathcal{F}_0} (d_1),
  \mathcal{M}^{\mathcal{F}_2}_{\mathcal{F}_1} (d_2),
  \mathcal{M}^{\mathcal{F}_3}_{\mathcal{F}_2} (d_3),
  V_3),$
  where $V_3 = \{(2,2,2,2,2)\}$.
  Since $c=3$, lines 3 through 5 of algorithm \ref{algo_res}
  are skipped.

  The first pass of the for loop on lines 6 through 8 applies
  algorithm \ref{algo_forward} to the pair
  $(\mathcal{M}^{\mathcal{F}_3}_{\mathcal{F}_2} (d_3),V_3)$, where
  \[\mathcal{M}^{\mathcal{F}_3}_{\mathcal{F}_2} (d_3) =
  \begin{pmatrix}-{p}_{(3,4)} {p}_{(2,5)}+{p}_{(2,4)} {p}_{(3,5)}-{p}_{(2,3)} {p}_{(4,5)}\\
    -{p}_{(3,4)} {p}_{(1,5)}+{p}_{(1,4)} {p}_{(3,5)}-{p}_{(1,3)} {p}_{(4,5)}\\
    {p}_{(2,4)} {p}_{(1,5)}-{p}_{(1,4)} {p}_{(2,5)}+{p}_{(1,2)} {p}_{(4,5)}\\
    -{p}_{(2,3)} {p}_{(1,5)}+{p}_{(1,3)} {p}_{(2,5)}-{p}_{(1,2)} {p}_{(3,5)}\\
    -{p}_{(2,3)} {p}_{(1,4)}+{p}_{(1,3)} {p}_{(2,4)}-{p}_{(1,2)} {p}_{(3,4)}\\
  \end{pmatrix} =
  \begin{pmatrix}
    -f_{2,3,4,5}\\
    -f_{1,3,4,5}\\
    f_{1,2,4,5}\\
    -f_{1,2,3,5}\\
    -f_{1,2,3,4}
  \end{pmatrix}
  .\]
  Now algorithm \ref{algo_forward} calls algorithm
  \ref{algo_multiple_degrees} applied to the pair
  $(\mathcal{M}^{\mathcal{F}_3}_{\mathcal{F}_2} (d_3)^\top,-V_3)$, where
  \[\mathcal{M}^{\mathcal{F}_3}_{\mathcal{F}_2} (d_3)^\top =
  \begin{pmatrix}
    -f_{2,3,4,5}&
    -f_{1,3,4,5}&
    f_{1,2,4,5}&
    -f_{1,2,3,5}&
    -f_{1,2,3,4}
  \end{pmatrix},\]
  and the module term ordering on the codomain of
  $\mathcal{M}^{\mathcal{F}_3}_{\mathcal{F}_2} (d_3)^\top$
  is switched to term over position down.
  Since all columns of the matrix above have degree 2,
  there is no need to separate columns by degrees and we can
  apply algorithm \ref{algo_single_degree} directly.
  The next step is to calculate a Gr\"obner basis of the image of
  this matrix (truncated in degree 2) and arrange its terms in a 
  matrix so that the leading terms of the columns are in decreasing
  order from left to right.
  This operation produces a matrix
  \[G=
  \begin{pmatrix}
    f_{1,2,3,4} & f_{1,2,3,5} & f_{1,2,4,5} & f_{1,3,4,5}
    & f_{2,3,4,5}
  \end{pmatrix},\]
  with column-by-column leading terms
  \[
  \begin{pmatrix}
    p_{2,3} p_{1,4} & p_{2,3} p_{1,5} & p_{2,4} p_{1,5} &
    p_{3,4} p_{1,5} & p_{3,4} p_{2,5}
  \end{pmatrix}.
  \]
  In this case, algorithm \ref{algo_single_degree} returns
  the change of basis matrix
  \begin{equation}\label{cob_matrix}
    \begin{pmatrix}0& 0& 0& 0&
      {-1}\\
      0& 0& 0& {-1}&
      0\\
      0& 0& 1& 0&
      0\\
      0& {-1}& 0& 0&
      0\\
      {-1}& 0& 0& 0&
      0\\
    \end{pmatrix}
  \end{equation}
  and the list of weights obtained as follows:
  \begin{align*}
    &\weight (p_{2,3} p_{1,4}) + (-2,-2,-2,-2,-2) = (-1,-1,-1,-1,-2),\\
    &\weight (p_{2,3} p_{1,5}) + (-2,-2,-2,-2,-2) = (-1,-1,-1,-2,-1),\\
    &\weight (p_{2,4} p_{1,5}) + (-2,-2,-2,-2,-2) = (-1,-1,-2,-1,-1),\\
    &\weight (p_{3,4} p_{1,5}) + (-2,-2,-2,-2,-2) = (-1,-2,-1,-1,-1),\\
    &\weight (p_{3,4} p_{2,5}) + (-2,-2,-2,-2,-2) = (-2,-1,-1,-1,-1).
  \end{align*}
  This output is passed back to algorithm \ref{algo_forward}
  which returns the pair $(C_2,V_2)$, where $C_2$ is the transpose
  of the matrix in equation (\ref{cob_matrix}) and
  \[V_2 = \{(1,1,1,1,2),(1,1,1,2,1),(1,1,2,1,1),(1,2,1,1,1),
  (2,1,1,1,1)\}.\]
  From this we deduce
  $F_2 \cong \mathbb{S}_{(2,1,1,1,1)} \CC^5 \otimes_\CC A(-3)$.

  Next we apply algorithm \ref{algo_forward} to the pair
  $(\mathcal{M}^{\mathcal{F}_2}_{\mathcal{F}_1} (d_2) C_2^{-1},V_2)$, where
  \[\mathcal{M}^{\mathcal{F}_2}_{\mathcal{F}_1} (d_2) =
  \begin{pmatrix}{-{p}_{1,5}}& {p}_{2,5}& {p}_{3,5}&
    {p}_{4,5}&
    0\\
    {p}_{1,4}& {-{p}_{2,4}}& {-{p}_{3,4}}& 0&
    {-{p}_{4,5}}\\
    {-{p}_{1,3}}& {p}_{2,3}& 0& {-{p}_{3,4}}&
    {p}_{3,5}\\
    {p}_{1,2}& 0& {p}_{2,3}& {p}_{2,4}&
    {-{p}_{2,5}}\\
    0& {-{p}_{1,2}}& {-{p}_{1,3}}& {-{p}_{1,4}}&
    {p}_{1,5}\\
  \end{pmatrix}.
  \]
  This results in a call to algorithm \ref{algo_single_degree}
  for the pair
  $((\mathcal{M}^{\mathcal{F}_2}_{\mathcal{F}_1} (d_2) C_2^{-1})^\top,-V_2)$,
  where
  \[(\mathcal{M}^{\mathcal{F}_2}_{\mathcal{F}_1} (d_2) C_2^{-1})^\top
  =
  \begin{pmatrix}0& {p}_{4,5}& {-{p}_{3,5}}& {p}_{2,5}&
    {-{p}_{1,5}}\\
    {-{p}_{4,5}}& 0& {p}_{3,4}& {-{p}_{2,4}}&
    {p}_{1,4}\\
    {p}_{3,5}& {-{p}_{3,4}}& 0& {p}_{2,3}&
    {-{p}_{1,3}}\\
    {-{p}_{2,5}}& {p}_{2,4}& {-{p}_{2,3}}& 0&
    {p}_{1,2}\\
    {p}_{1,5}& {-{p}_{1,4}}& {p}_{1,3}& {-{p}_{1,2}}&
    0\\
  \end{pmatrix}\]
  and the codomain of this matrix has the term over position down
  ordering. Our calculation gives the Gr\"obner basis
  matrix (with columns in decreasing order of their leading terms
  from left to right)
  \[
  G=\begin{pmatrix}{-{p}_{1,5}}& {-{p}_{2,5}}& {-{p}_{3,5}}&
    {-{p}_{4,5}}&
    0\\
    {p}_{1,4}& {p}_{2,4}& {p}_{3,4}& 0&
    {-{p}_{4,5}}\\
    {-{p}_{1,3}}& {-{p}_{2,3}}& 0& {p}_{3,4}&
    {p}_{3,5}\\
    {p}_{1,2}& 0& {-{p}_{2,3}}& {-{p}_{2,4}}&
    {-{p}_{2,5}}\\
    0& {p}_{1,2}& {p}_{1,3}& {p}_{1,4}&
    {p}_{1,5}\\
  \end{pmatrix},
  \]
  and the change of basis matrix
  \begin{equation}\label{cob_matrix2}
  \begin{pmatrix}0& 0& 0& 0&
    1\\
    0& 0& 0& {-1}&
    0\\
    0& 0& 1& 0&
    0\\
    0& {-1}& 0& 0&
    0\\
    1& 0& 0& 0&
    0\\
  \end{pmatrix}.
  \end{equation}
  The matrix of column-by-column leading terms of $G$ is
  \[
  \begin{pmatrix}0& 0& 0& 0&
    0\\
    0& 0& 0& 0&
    0\\
    0& 0& 0& 0&
    0\\
    {p}_{1,2}& 0& 0& 0&
    0\\
    0& {p}_{1,2}& {p}_{1,3}& {p}_{1,4}&
    {p}_{1,5}\\
  \end{pmatrix}\]
  so our weight computation goes as follows:
  \begin{align*}
    \weight (p_{1,2}) + (-1,-2,-1,-1,-1) = (0,-1,-1,-1,-1),\\
    \weight (p_{1,2}) + (-2,-1,-1,-1,-1) = (-1,0,-1,-1,-1),\\
    \weight (p_{1,3}) + (-2,-1,-1,-1,-1) = (-1,-1,0,1-,-1),\\
    \weight (p_{1,4}) + (-2,-1,-1,-1,-1) = (-1,-1,-1,0,-1),\\
    \weight (p_{1,5}) + (-2,-1,-1,-1,-1) = (-1,-1,-1,-1,0).
  \end{align*}
  This information is returned to our current run of algorithm
  \ref{algo_forward} which outputs the pair $(C_1,V_1)$, where
  $C_1$ is the transpose of the matrix in equation \ref{cob_matrix2}
  and 
  \[V_1 = \{(0,1,1,1,1),(1,0,1,1,1),(1,1,0,1,1),(1,1,1,0,1),
  (1,1,1,1,0)\}.\]
  We deduce $F_1 \cong \bigwedge^4 \CC^5 \otimes_\CC A(-2)$.

  Finally we apply algorithm \ref{algo_forward} to the pair
  $(\mathcal{M}^{\mathcal{F}_1}_{\mathcal{F}_0} (d_1) C_1^{-1},V_1)$, where
  \[\mathcal{M}^{\mathcal{F}_1}_{\mathcal{F}_0} (d_1)=
  \begin{pmatrix}f_{1,2,3,4} & f_{1,2,3,5} & f_{1,2,4,5} & f_{1,3,4,5} &
    f_{2,3,4,5} \end{pmatrix}.\]
  In turn, this applies algorithm \ref{algo_single_degree} to the
  pair $((\mathcal{M}^{\mathcal{F}_1}_{\mathcal{F}_0} (d_1) C_1^{-1})^\top,
  -V_1)$, where
  \[(\mathcal{M}^{\mathcal{F}_1}_{\mathcal{F}_0} (d_1) C_1^{-1})^\top=
  \begin{pmatrix}f_{2,3,4,5} \\ -f_{1,3,4,5} \\ f_{1,2,4,5} \\
    -f_{1,2,3,5} \\ f_{1,2,3,4} \end{pmatrix}=
  \begin{pmatrix}{p}_{3,4} {p}_{2,5}-{p}_{2,4} {p}_{3,5}+{p}_{2,3} {p}_{4,5}\\
    -{p}_{3,4} {p}_{1,5}+{p}_{1,4} {p}_{3,5}-{p}_{1,3} {p}_{4,5}\\
    {p}_{2,4} {p}_{1,5}-{p}_{1,4} {p}_{2,5}+{p}_{1,2} {p}_{4,5}\\
    -{p}_{2,3} {p}_{1,5}+{p}_{1,3} {p}_{2,5}-{p}_{1,2} {p}_{3,5}\\
    {p}_{2,3} {p}_{1,4}-{p}_{1,3} {p}_{2,4}+{p}_{1,2} {p}_{3,4}\\
  \end{pmatrix}\]
  and once again the codomain is endowed with the term over position
  down ordering. Since this matrix consists of a single column
  it is equal to its Gr\"obner basis matrix
  and the corresponding change of basis matrix is simply
  $\begin{pmatrix} 1\end{pmatrix}$. The leading term of this
  column is
  \[
  \begin{pmatrix}
    0\\0\\0\\0\\p_{2,3}p_{1,4}
  \end{pmatrix}
  \]
  which has weight
  \[
  \weight (p_{2,3} p_{1,4}) + (-1,-1,-1,-1,0) = (0,0,0,0,0).
  \]
  Algorithm \ref{algo_single_degree} concludes by passing its
  output to algorithm \ref{algo_forward} which in turn terminates
  by returning the pair $(C_0,V_0)$, where
  $C_0 = \begin{pmatrix} 1\end{pmatrix}$ and
  $V_0 = \{(0,0,0,0,0)\}$. This says $F_0 \cong A$, as expected.

  Finally, algorithm \ref{algo_res} ends by collecting all the
  intermediate results and returning the tuple $(V_0,V_1,V_2,V_3)$.
\end{ex}

\subsection{Weight propagation for graded components}\label{sec_algo_comps}
Our last algorithm can be used to recover the weights of the graded components of a module starting from a presentation. The following is assumed.
\begin{enumerate}[label=(\alph*),wide]
\item The sequence
\[F_1 \xrightarrow{d_1} F_0 \xrightarrow{\pi} M \to 0\]
is a presentation of the module $M$ in $\mod_{\righttoleftarrow T} A$.
\item The map $d_1$ is provided in matrix form $\mathcal{M}^{\mathcal{F}_1}_{\mathcal{F}_0} (d_1)$, where $\mathcal{F}_0$ is a homogeneous basis of $F_0$ and $\mathcal{F}_1$ is a homogeneous basis of $F_1$.
\item $\TT^n \langle \mathcal{F}_0 \rangle$ is equipped with a position up (resp. down) module term ordering.
\item $F_0$ admits a homogeneous basis of weight vectors $\tilde{\mathcal{F}}_0 = \{\tilde{f}_1,\ldots,\tilde{f}_s\}$ such that $\mathcal{M}^{\mathcal{F}_0}_{\tilde{\mathcal{F}}_0} (\id_{F_0})$ is upper (resp. lower) triangular.
\item $W = \{w_1,\ldots,w_s\}$ is an ordered list with $w_i = \weight (\tilde{f}_i)$, $\forall i\in\{1,\ldots,s\}$.
\end{enumerate}

\begin{algo}[weight propagation for graded components]\label{algo_graded_comps}
\algrenewcommand\algorithmicrequire{\textbf{Input:}}
\algrenewcommand\algorithmicensure{\textbf{Output:}}
\begin{algorithmic}[1]
\Require{\hfill
\begin{itemize}
\item $d$, a degree
\item $\mathcal{M}^{\mathcal{F}_1}_{\mathcal{F}_{0}} (d_1)$, the matrix of
  a presentation of $M$ as in assumption (b) above
\item $W$, a list of weights as in assumption (e) above
\end{itemize}
}
\Ensure{\hfill
  \begin{itemize}
  \item $V$, a list
    of weights for a basis of weight vectors of the graded
    component of $M$ of degree $d$
  \end{itemize}
~\hfill
}
\Function{PropagateGradedComponents}{$d,\mathcal{M}^{\mathcal{F}_1}_{\mathcal{F}_0} (d_1),W$}
   \State compute $\mathcal{G}$, homogeneous Gr\"obner basis of $\im d_1$, using $\mathcal{M}^{\mathcal{F}_1}_{\mathcal{F}_0} (d_1)$
   \State obtain $L$, a list of degree $d$ terms in $\TT^n \langle \mathcal{F}_0\rangle$ that are not divisible by the leading terms of elements of $\mathcal{G}$
   \State form $N$, a matrix with columns the component vectors (in the basis $\mathcal{F}_0$) of the elements of $L$
   \State $(C,V) := \text{\textsc{Propagate}} (N,W)$
   \State \textbf{return} $V$
\EndFunction
\end{algorithmic}
\end{algo}

\begin{ex}
  \label{exa:gr1}
  Consider the setup of example \ref{exa:3}.
  Our module $M$ is $A$ modulo the ideal $(x_1,x_2,y_1^2,y_1 y_2,
  y_2^2)$; $M$ is presented by the map $d_1 \colon F_1 \to F_0$
  where
  \[\mathcal{M}^{\mathcal{F}_1}_{\mathcal{F}_0} (d_1)=
  \begin{pmatrix}
    x_1 & x_2 & y_1^2 & y_1 y_2 & y_2^2
  \end{pmatrix}\]
  with respect to the the coordinate bases $\mathcal{F}_i$ of the
  free modules $F_i$.
  In particular, $F_0 = A$ and its unique basis element
  is a weight vector with weight $(0,0,0,0)$.

  We apply algorithm \ref{algo_graded_comps} to the tuple
  $(d,\mathcal{M}^{\mathcal{F}_1}_{\mathcal{F}_0} (d_1), W)$, with
  $d=(0,1)$ and $W=\{(0,0,0,0)\}$ to find the weights of the
  graded component of degree $(0,1)$ of $A/I$.

  Our software system finds
  \[\mathcal{G} = \{x_2,x_1,y_2^2,y_1 y_2,y_1^2\}\]
  to be a Gr\"obner basis of $I$. The terms in $A$ of degree $(0,1)$
  are $y_1$ and $y_2$; since neither is divisible by the leading
  terms of elements of $\mathcal{G}$, we form the matrix:
  \[N=
  \begin{pmatrix}
    y_1 & y_2
  \end{pmatrix}\]
  and apply algorithm \ref{algo_single_degree} to the pair 
  $(N,W)$ (via a preliminary call to algorithm
  \ref{algo_multiple_degrees}).
  It is easy to see that algorithm \ref{algo_single_degree}
  computes the following weights:
  \begin{align*}
    &\weight (y_2) + (0,0,0,0) = (0,0,0,1),\\
    &\weight (y_1) + (0,0,0,0) = (0,0,1,0).
  \end{align*}
  Then algorithm \ref{algo_graded_comps} returns the list of weights
  \[V=\{(0,0,0,1),(0,0,1,0)\},\]
  which tells us that $(A/I)_{(0,1)} \cong \CC \otimes \CC^2$ as a
  representation of $\GL_2 (\CC) \times \GL_2 (\CC)$.
\end{ex}

\begin{ex}
  \label{exa:gr2}
  Consider the setup of example \ref{exa:7}.
  The module $M$ is $A$ modulo the ideal $I$ generated by
  the Pl\"ucker equations of our Grassmannian.
  More explicitly, $M$ is presented by the map
  $d_1 \colon F_1 \to F_0$ where
  \[\mathcal{M}^{\mathcal{F}_1}_{\mathcal{F}_0} (d_1)=
  \begin{pmatrix}
    f_{1,2,3,4} & f_{1,2,3,5} & f_{1,2,4,5} & f_{1,3,4,5} & f_{2,3,4,5}
  \end{pmatrix},\]
  and the polynomials $f_{i,j,k,l}$ are the same as those
  introduced in example \ref{exa:7}.
  The module $F_0 = A$ has a unique basis element which
  is a weight vector of weight $(0,0,0,0,0)$.

  We apply algorithm \ref{algo_graded_comps} to the tuple
  $(d,\mathcal{M}^{\mathcal{F}_1}_{\mathcal{F}_0} (d_1), W)$, with
  $d=2$ and $W=\{(0,0,0,0,0)\}$.

  The polynomials in the presentation above form a Gr\"obner basis
  of $I$ and their leading terms are in order:
  \[
  \begin{array}{lllll}
    p_{2,3} p_{1,4}&p_{2,3} p_{1,5}&p_{2,4} p_{1,5}&p_{3,4} p_{1,5}&p_{3,4} p_{2,5}.
  \end{array}
  \]
  Via our computer algebra system we find the following list of
  50 terms of degree 2 in $A$ that are not divisible by the leading
  terms above:
  \[
  \begin{array}{lllll}
    p_{1,2}^2&p_{1,2} p_{1,3}&p_{1,2} p_{2,3}&p_{1,2} p_{1,4}&p_{1,2} p_{2,4}\\
    p_{1,2} p_{3,4}&p_{1,2} p_{1,5}&p_{1,2} p_{2,5}&p_{1,2} p_{3,5}&p_{1,2} p_{4,5}\\
    p_{1,3}^2&p_{1,3} p_{2,3}&p_{1,3} p_{1,4}&p_{1,3} p_{2,4}&p_{1,3} p_{3,4}\\
    p_{1,3} p_{1,5}&p_{1,3} p_{2,5}&p_{1,3} p_{3,5}&p_{1,3} p_{4,5}&p_{2,3}^2\\
    p_{2,3} p_{2,4}&p_{2,3} p_{3,4}&p_{2,3} p_{2,5}&p_{2,3} p_{3,5}&p_{2,3} p_{4,5}\\
    p_{1,4}^2&p_{1,4} p_{2,4}&p_{1,4} p_{3,4}&p_{1,4} p_{1,5}&p_{1,4} p_{2,5}\\
    p_{1,4} p_{3,5}&p_{1,4} p_{4,5}&p_{2,4}^2&p_{2,4} p_{3,4}&p_{2,4} p_{2,5}\\
    p_{2,4} p_{3,5}&p_{2,4} p_{4,5}&p_{3,4}^2&p_{3,4} p_{3,5}&p_{3,4} p_{4,5}\\
    p_{1,5}^2&p_{1,5} p_{2,5}&p_{1,5} p_{3,5}&p_{1,5} p_{4,5}&p_{2,5}^2\\
    p_{2,5} p_{3,5}&p_{2,5} p_{4,5}&p_{3,5}^2&p_{3,5} p_{4,5}&p_{4,5}^2.
  \end{array}
  \]
  These terms are assembled into a one-row matrix $N$ and
  algorithm \ref{algo_single_degree} is applied to the pair
  $(N,\{(0,0,0,0,0)\})$.
  We provide a sample of the weight computations performed:
  \begin{align*}
    &\weight (p_{1,2}^2) + (0,0,0,0,0) = (2,2,0,0,0),\\
    &\weight (p_{1,2} p_{1,3}) + (0,0,0,0,0) = (2,1,1,0,0),\\
    &\weight (p_{1,2} p_{3,4}) + (0,0,0,0,0) = (1,1,1,1,0).
  \end{align*}
  By analyzing a complete list of weights returned by the algorithm,
  it is possible to conclude that $(A/I)_2 \cong \mathbb{S}_{(2,2)}
  \CC^5$ as a representation of $\GL_5 (\CC)$.
\end{ex}

\begin{prop}\label{algo_graded_comps_prop}
Under the assumptions of this section, algorithm \ref{algo_graded_comps} returns an ordered list $V$ with the weights of the elements in a basis of weight vectors of $M_d$.
\end{prop}
\begin{proof}
We organize the proof into six separate steps.

\begin{mydescription}[wide,style=nextline]
\item[Step 1: a basis of $M_d$.] Let $\mathcal{G}$ be a Gr\"obner basis of $\im d_1$ in the module term ordering on  $\TT^n \langle \mathcal{F}_0 \rangle$. Suppose $\mathcal{F}_0 = \{f_1,\ldots,f_s\}$. Define the set
\[\mathcal{B}_d := \{tf_i \in \TT^n \langle \mathcal{F}_0 \rangle \mid \deg (tf_i) = d \text{ and } \forall g\in \mathcal{G}, \LT (g) \nmid tf_i\},\]
consisting of all the degree $d$ terms of $\TT^n \langle \mathcal{F}_0 \rangle$ that are not multiples of the leading terms of some element in $\mathcal{G}$. By Macaulay's basis theorem \cite[Cor. 2.4.11]{MR1790326}, the residue classes of elements in $\mathcal{B}_d$ modulo $\im d_1$ form a $\KK$-basis of $M_d$. In particular, $\dim_\KK \langle \mathcal{B}_d\rangle_\KK = \dim_\KK M_d$, where $\langle \mathcal{B}_d\rangle_\KK$ is the $\KK$-vector subspace of $(F_0)_d$ generated by the terms in $\mathcal{B}_d$.

\item[Step 2: a subrepresentation of $(F_0)_d$.]
Define 
\[\tilde{\mathcal{B}}_d := \{t\tilde{f}_i \in F_0 \mid tf_i \in \mathcal{B}_d\},\]
and consider $\langle \tilde{\mathcal{B}}_d\rangle_\KK$, the $\KK$-vector subspace of $(F_0)_d$ generated by $\tilde{\mathcal{B}}_d$. Notice that all elements of $\tilde{\mathcal{B}}_d$ are $\KK$-linearly independent, hence $\dim_\KK \langle \tilde{\mathcal{B}}_d\rangle_\KK = \dim_\KK \langle \mathcal{B}_d\rangle_\KK = \dim_\KK M_d$. Moreover each element $t\tilde{f}_i\in\tilde{\mathcal{B}}_d$ is a weight vector for the action of $T$, since the terms $t\in\TT^n$ are weight vectors, by proposition \ref{weight_of_term}, and the $\tilde{f}_i$ come from a homogeneous basis of weight vectors of $F_0$.

\item[Step 3: $(F_0)_d = (\ker \pi)_d \oplus \langle \tilde{\mathcal{B}}_d\rangle_\KK$.]
Let $t\tilde{f}_j \in \tilde{\mathcal{B}}_d$. Since $\mathcal{M}^{\mathcal{F}_0}_{\tilde{\mathcal{F}}_0} (\id_{F_0})$ is upper (resp. lower) triangular, so is its inverse $\mathcal{M}_{\mathcal{F}_0}^{\tilde{\mathcal{F}}_0} (\id_{F_0})$. Suppose $\mathcal{M}_{\mathcal{F}_0}^{\tilde{\mathcal{F}}_0} (\id_{F_0}) = (u_{i,j})$, for some $u_{i,j} \in \KK$. Then
\[t\tilde{f}_j = \sum_{i=1}^j u_{i,j} tf_i \quad \left(\text{resp. } t\tilde{f}_j = \sum_{i=j}^s u_{i,j} tf_i \right).\]
Seeing how $\TT^n \langle \mathcal{F}_0 \rangle$ is equipped with a position up (resp. down) module term ordering, we deduce $\LT (t\tilde{f}_j) = tf_j\in \mathcal{B}_d$.

Now consider any element $f\in \langle \tilde{\mathcal{B}}_d\rangle_\KK$. Since $f$ is a $\KK$-linear combination of elements in $\tilde{\mathcal{B}}_d$, we must have $\LT (f) \in \mathcal{B}_d$ by what just observed. This implies that $\LT (f)$ is not divisible by the leading term of any element of $\mathcal{G}$, therefore the remainder of $f$ upon division by the elements of $\mathcal{G}$ is $f$ itself.

Suppose $f\in (\ker \pi)_d \cap \langle \tilde{\mathcal{B}}_d\rangle_\KK$. Since $\ker \pi = \im d_1$ and $\mathcal{G}$ is a Gr\"obner basis of $\im d_1$, the remainder of $f$ upon division by elements of $\mathcal{G}$ is zero. This forces $f=0$ and hence the sum of $(\ker \pi)_d$ and $\langle \tilde{\mathcal{B}}_d\rangle_\KK$, as subspaces of $(F_0)_d$, is direct. Looking at dimensions we obtain:
\[\dim_\KK (\ker \pi)_d + \dim_\KK \langle \tilde{\mathcal{B}}_d\rangle_\KK = \dim_\KK (\ker \pi)_d + \dim_\KK M_d = \dim_\KK (F_0)_d,\]
since $\pi$ is surjective. We conclude $(F_0)_d = (\ker \pi)_d \oplus \langle \tilde{\mathcal{B}}_d\rangle_\KK$. Notice that $(\ker\pi)_d$ is a subrepresentation of $T$ in $(F_0)_d$, since $\pi$ is $T$-equivariant. Therefore the direct sum decomposition holds as a decomposition of representations of $T$.

\item[Step 4: an explicit section of $\pi$ in degree $d$.]
We will define a $T$-equivariant map $\hat{\varphi} \colon M_d \to F_0$ such that $\forall m\in M_d$, $\pi (\hat{\varphi} (m)) = m$, a section of $\pi$ in degree $d$. Recall that the elements of $\tilde{\mathcal{B}}_d$ are weight vectors in $\langle \tilde{\mathcal{B}}_d\rangle_\KK$, so that $\pi ( \tilde{\mathcal{B}}_d)$ is a set of weight vectors in $M_d$. Given the decomposition $(F_0)_d = (\ker \pi)_d \oplus \langle \tilde{\mathcal{B}}_d\rangle_\KK$, we have $\pi (\langle \tilde{\mathcal{B}}_d\rangle_\KK) = M_d$. Since $\dim_\KK \langle \tilde{\mathcal{B}}_d\rangle_\KK = \dim_\KK M_d$, $\pi ( \tilde{\mathcal{B}}_d)$ is actually a basis of weight vectors of $M_d$.  Let us define $\hat{\varphi}$ on $\pi ( \tilde{\mathcal{B}}_d)$ by setting $\hat{\varphi} (\pi (\tilde{b})) := \tilde{b}$, $\forall \tilde{b} \in \tilde{\mathcal{B}}_d$. Extending linearly gives a $\KK$-linear map $\hat{\varphi} \colon M_d \to F_0$. Since $\pi (\hat{\varphi} (\pi (\tilde{b}))) = \pi(\tilde{b})$, $\forall \tilde{b} \in \tilde{\mathcal{B}}_d$, the map $\hat{\varphi}$ is a section of $\pi$ in degree $d$.

To show that $\hat{\varphi}$ is $T$-equivariant, it is enough to observe that, $\forall \tilde{b} \in \tilde{\mathcal{B}}_d$, $\pi (\tilde{b})$ is a weight vector with the same weight as $\tilde{b}$, by proposition \ref{weight_in_equivariant_map}, because $\pi$ is $T$-equivariant.

\item[Step 5: the map $\varphi \colon E\to F_0$.]
By the universal property of free modules in $\mod_{\righttoleftarrow T} A$, $\exists ! \varphi \colon M_d \otimes_\KK A \to F_0$ morphism in $\mod_{\righttoleftarrow T} A$ such that $\hat{\varphi} = \varphi \circ i_{M_d}$, where $i_{M_d} \colon M_d \to M_d \otimes_\KK A$ sends an element $m\in M_d$ to $m \otimes 1_A$. Set $E := M_d \otimes_\KK A$. Because $\pi(\tilde{\mathcal{B}}_d)$ is a basis of weight vectors of $M_d$, the set $\tilde{\mathcal{E}} := \{\pi (\tilde{b}) \otimes 1_A \mid \tilde{b} \in \mathcal{B}_d\}$ is a homogeneous basis of weight vectors of $E$. Moreover, $\forall \tilde{b} \in \tilde{\mathcal{B}}_d$, we have
\[\varphi (\pi (\tilde{b}) \otimes 1_A ) = \hat{\varphi} (\pi (\tilde{b})) = \tilde{b} \in \langle \tilde{\mathcal{B}}_d\rangle_\KK \subseteq (F_0)_d.\]
Notice that $i_{M_d}$ is an isomorphism in degree $d$, and that $\hat{\varphi}$ is injective because it is a section of $\pi$ in degree $d$; therefore $\varphi$ is injective in degree $d$. Because $E$ is generated in degree $d$, we conclude that $\varphi$ is a minimal map by part \ref{min2} of proposition \ref{minimality}.

\item[Step 6: an explicit matrix of $\varphi$.]
We will describe the matrix $N:= \mathcal{M}^{\tilde{\mathcal{E}}}_{\tilde{\mathcal{F}}_0} (\varphi)$. The elements in $\tilde{\mathcal{B}}_d$ are, by definition, of the form $t\tilde{f}_i$, where $tf_i\in\mathcal{B}_d$. For $t\tilde{f}_i \in \tilde{\mathcal{B}}_d$, we have, by construction,
\[\varphi (\pi (t\tilde{f}_i) \otimes 1_A) = t \tilde{f}_i.\]
Therefore the column of $N$ corresponding to the element $\pi (t \tilde{f}_i) \otimes 1_A\in \tilde{\mathcal{E}}$ has the term $t$ in the $i$-th row and zeros everywhere else. In other words, $N$ is the matrix whose columns are the column vectors of terms in $\mathcal{B}_d$ expressed in the homogeneous basis $\mathcal{F}_0$ of $F_0$. Since $\mathcal{B}_d$ can be obtained explicitly after computing a Gr\"obner basis of $\im d_1$, this construction of the matrix $N$ can be carried out explicitly.
\end{mydescription}

Finally we can use algorithm 2, with input the matrix of $\varphi$ just described and the list of weights $W$ of $F_0$, to recover the weights of $E = M_d\otimes_\KK A$. Since these are the same as the weights of $M_d$, this concludes the proof.
\end{proof}

\subsection{Computing over subfields}\label{subfields}
Let $\mathbb{L}$ and $\KK$ be fields with $\KK\subseteq\mathbb{L}$. Consider the polynomial ring $A_\mathbb{L} := \mathbb{L} [x_1,\ldots,x_n]$ with a positive $\ZZ^m$-grading and identify the polynomial ring $A_\KK := \KK [x_1,\ldots,x_n]$ with a (graded) subring of $A_\mathbb{L}$. As usual, all our modules (over $A_\mathbb{L}$ or $A_\KK$) will be finitely generated and graded.
Consider a graded $A_\mathbb{L}$-submodule $M$ of the free module $F = \bigoplus_{d\in\ZZ^m} A_\mathbb{L} (-d)^{\beta_d}$. Following \cite[Defin. 2.4.14]{MR1790326}, $M$ is defined over $\KK$ if there exist elements $m_1,\ldots,m_l$ in the free $A_\KK$-module $\bigoplus_{d\in\ZZ^m} A_\KK (-d)^{\beta_d} \subseteq F$ which generate $M$ as an $A_\mathbb{L}$-module.

If $M$ is defined over $\KK$, then
\begin{itemize}[wide]
\item computing the reduced Gr\"obner basis $\mathcal{G}$ of $M$ over $\KK$ or over $\mathbb{L}$, using the elements $m_1,\ldots,m_l$, yields the same result, by \cite[Prop. 2.4.16.b]{MR1790326};
\item the leading terms of the elements of $\mathcal{G}$ do not depend on the field used for the computation, by \cite[Prop. 2.4.16.a]{MR1790326};
\item the matrix of the change of basis between the elements $m_1,\ldots,m_l$ and the elements of $\mathcal{G}$ has entries in $A_\KK$.
\end{itemize}
The third bullet point is an immediate consequence of the first one.

To explain how this affects our algorithms, let $T$ be a torus over $\mathbb{L}$ with an $\mathbb{L}$-linear action on $A_\mathbb{L}$ that is compatible with grading and multiplication.
Let $\varphi \colon E\to F$ be a minimal map of free modules in the category $\mod_{\righttoleftarrow T} A_\mathbb{L}$. Suppose there exist homogeneous bases $\mathcal{E}$ of $E$ and $\mathcal{F}$ of $F$ such that $\mathcal{M}^{\mathcal{E}}_{\mathcal{F}} (\varphi)$ has entries in $\KK$; equivalently, $\im\varphi$ is defined over $\KK$. Our previous observations imply that the steps used in algorithm \ref{algo_single_degree} do not depend on the field. In practice, using algorithm \ref{algo_single_degree} or \ref{algo_multiple_degrees} to recover the weights of $E$ from the weights of $F$ and the map $\varphi$, will produce the same result whether we carry out our computations over $\mathbb{L}$ or over $\KK$.

Algorithms \ref{algo_forward}, \ref{algo_res} and \ref{algo_graded_comps} are based on algorithm \ref{algo_multiple_degrees} so we expect them to work over subfields as well. Indeed they do, because of the following additional comments. If $\varphi \colon E\to F$ is a minimal map of free modules in $\mod_{\righttoleftarrow T} A_\mathbb{L}$ with $\im\varphi$ defined over $\KK$, then:
\begin{itemize}[wide]
\item the image of the dual map $\varphi^\vee$ is also defined over $\KK$ (clearly since a matrix of $\varphi^\vee$ is the transpose of a matrix of $\varphi$);
\item the syzygy module of a Gr\"obner basis of $\im\varphi$ is also defined over $\KK$ (this implies a minimal free resolution of a module can be computed over $\KK$);
\item each graded component of $\im\varphi$ has a basis consisting of those terms that are not divisible by the leading terms of the elements in a Gr\"obner basis of $\im\varphi$, in particular this basis does not depend on the field.
\end{itemize}

\begin{remark}
The possibility of performing our algorithms over a subfield $\KK$ of $\mathbb{L}$ is especially useful in the case where computations over $\mathbb{L}$ are not feasible. To further illustrate the issue, we discuss the setup for our implementation \cite{Galetto:2013ab} in the software system \texttt{Macaulay2} \cite{Grayson:uq}.

Let $G$ be a complex semisimple algebraic group and let $T$ be a maximal torus in $G$. Every finite dimensional representation $V$ of $G$ is uniquely determined by its weights (counted with multiplicity) for the action of $T$. Moreover $V$ decomposes uniquely into a direct sum of irreducible representations parametrized by so-called highest weights. Given a complete list of weights of $V$ for the action of $T$, the highest weights can be recovered using a recursive formula of Freudenthal \cite[\S22.3]{Humphreys:1978fk} which holds over $\mathbb{C}$.

Our objects of interest are modules in $\mod_{\righttoleftarrow T} A_\mathbb{C}$ that are defined over $\mathbb{Q}$. While it is not possible to compute over $\mathbb{C}$, we can compute over $\mathbb{Q}$ in \texttt{Macaulay2}. In particular, we can calculate minimal free resolutions and graded components over the rationals. The implementation of our algorithms provides lists of weights that can then be interpreted and decomposed over the complex numbers.
\end{remark}

\bibliographystyle{amsalpha}
\bibliography{/Users/fez/BibTeX/math.bib}
%\addcontentsline{toc}{section}{References}

\end{document}